\documentclass[11pt,reqno]{amsart}
\makeatletter
\@namedef{subjclassname@2020}{\textup{2020} Mathematics Subject Classification}
\makeatother
\usepackage{amssymb,amscd,amsthm}
\usepackage[T1]{fontenc}
\usepackage{blkarray}

\usepackage{amsmath,amssymb,graphicx,mathrsfs} 
\usepackage[colorlinks=true,urlcolor=blue,citecolor=red,linkcolor=blue]{hyperref}

\let\frak\mathfrak
\let\Bbb\mathbb

\def\>{\relax\ifmmode\mskip.666667\thinmuskip\relax\else\kern.111111em\fi}
\def\<{\relax\ifmmode\mskip-.333333\thinmuskip\relax\else\kern-.0555556em\fi}
\def\vsk#1>{\vskip#1\baselineskip}
\def\vv#1>{\vadjust{\vsk#1>}\ignorespaces}
\def\vvn#1>{\vadjust{\nobreak\vsk#1>\nobreak}\ignorespaces}

  \let\ssize\scriptstyle
\let\sssize\scriptscriptstyle

\let\Medskip\medskip
\def\medskip{\par\Medskip}
\let\Bigskip\bigskip
\def\bigskip{\par\Bigskip}

\let\Maketitle\maketitle
\def\maketitle{\Maketitle\thispagestyle{empty}\let\maketitle\empty}

\newtheorem{thm}{Theorem}[section]
\newtheorem{cor}[thm]{Corollary}
\newtheorem{lem}[thm]{Lemma}
\newtheorem{prop}[thm]{Proposition}

\numberwithin{equation}{section}

\theoremstyle{definition}
\newtheorem{defn}[thm]{Definition}
\newtheorem{rem}[thm]{Remark}

\newtheorem{example}[thm]{Example}

\let\mc\mathcal
\let\nc\newcommand

\let\al\alpha
\let\bt\beta
\let\dl\delta

\let\gm\gamma

\let\la\lambda
\let\La\Lambda

\let\phi\varphi
\let\si\sigma
\let\Si\Sigma

\let\thi\vartheta

\let\Ups\Upsilon
\let\om\omega
\let\Om\Omega

\let\der\partial
\let\Hat\widehat

\let\ox\otimes
\let\Tilde\widetilde

\let\geq\geqslant

\let\leq\leqslant

\let\on\operatorname
\let\bi\bibitem
\let\bs\boldsymbol

\def\C{{\mathbb C}}
\def\Z{{\mathbb Z}}
\def\Q{{\mathbb Q}}

\def\Pb{{\mathbb P}}
\def\H{{\mathbb H}}

\def\F{{\mc F}}

\def\+#1{^{\{#1\}}}

\def\rank{\on{rank}}

\def\d{\on{d}\!}

\def\gl{\mathfrak{gl}}

\def\sln{\mathfrak{sl}_N}

\def\beq{\begin{equation}}
\def\eeq{\end{equation}}
\def\be{\begin{equation*}}
\def\ee{\end{equation*}}

\nc{\bea}{\begin{eqnarray*}}
\nc{\eea}{\end{eqnarray*}}
\nc{\bean}{\begin{eqnarray}}
\nc{\eean}{\end{eqnarray}}
\nc{\bal}{\begin{align*}}
\nc{\eal}{\end{align*}}
\nc{\baln}{\begin{align}}
\nc{\ealn}{\end{align}}

\def\h{{\mathfrak h}}

\nc{\Il}{{\mc I_{\bs\la}}}
\nc{\bla}{{\bs\la}}
\nc{\Fla}{\F_\bla}
\nc{\tfl}{{T^*\Fla}}
\nc{\GL}{{GL_n(\C)}}
\nc{\GLC}{{GL_n(\C)\times\C^*}}

\let\sd s 

\def\ddk_#1{\kk_{#1}\<\>\frac\der{\der\<\>\kk_{#1}}}

\def\bul{\mathbin{\raise.2ex\hbox{$\sssize\bullet$}}}
\def\intt{\mathchoice
{\mathop{\raise.2ex\rlap{$\,\,\ssize\backslash$}{\intop}}\nolimits}
{\mathop{\raise.3ex\rlap{$\,\sssize\backslash$}{\intop}}\nolimits}
{\mathop{\raise.1ex\rlap{$\sssize\>\backslash$}{\intop}}\nolimits}
{\mathop{\rlap{$\sssize\<\>\backslash$}{\intop}}\nolimits}}

\let\kk q 
\let\cc c

\let\Ko K

\def\GZ/{Gelfand-Zetlin}
\def\KZ/{{\slshape KZ\/}}
\def\qKZ/{{\slshape qKZ\/}}
\def\XXX/{{\slshape XXX\/}}

\nc{\slnl}{{\sln (\lambda)}}
\nc{\PCN}{{   (\C[x])^N   }}
\nc{\di}{\on{Diag}}
\nc{\dio}{\on{Diag}_0}
\nc{\Mm}{{\mc M}}
\nc{\Nn}{{\mc N}}

\nc{\A}{{\mc C}}

\nc{\PCr}{{  P  (\C[x])^n   }}

\nc{\Pk}{{(\bs{P}^1)^k}}

\nc{\N}{{\Bbb N}}

\nc{\Ll}{{\mc L}}

\nc{\ord}{{\on{ord}\,}}

\nc{\Sing}{{\on{Sing}\,}}
\nc{\sing}{{\on{Sing}\,}}

\nc{\Hess}{{\on{Hess}}}

\nc{\R}{{\Bbb R}}

\let\on\operatorname
\nc{\Kk}{{\bs K}}
\nc{\Ap}{{A_\Phi(z)}}
\nc{\ap}{{A_\Phi(z)}}

\nc{\sv}{{\sing V}}
\nc{\cd}{{\C^n-\Delta}}
\nc{\UT}{{U^0}}   
\nc{\ep}{\epsilon}

\usepackage[all,cmtip]{xy}
\usepackage[vcentermath]{youngtab}
\usepackage{empheq}
\usepackage{bm}

\newlanguage\fakelanguage
\newcommand\cyr{\fontencoding{OT2}\fontfamily{wncyr}\selectfont
   \language\fakelanguage}
\DeclareTextFontCommand{\textcyr}{\cyr}

\numberwithin{equation}{section}

\usepackage{calligra}

\DeclareMathOperator{\HOM}{\mathscr{H}\text{\kern -3pt {\calligra\large om}}\,}

\makeatletter
\newsavebox{\@brx}
\newcommand{\llangle}[1][]{\savebox{\@brx}{\(\m@th{#1\langle}\)}%
  \mathopen{\copy\@brx\kern-0.5\wd\@brx\usebox{\@brx}}}
\newcommand{\rrangle}[1][]{\savebox{\@brx}{\(\m@th{#1\rangle}\)}%
  \mathclose{\copy\@brx\kern-0.5\wd\@brx\usebox{\@brx}}}
\makeatother

\newcommand{\bsh}{\begin{shaded}}
\newcommand{\esh}{\end{shaded}}

\newcommand{\ic}{\sqrt{-1}}
\newcommand{\hol}{\mathscr O}
\newcommand\pnorm[2]{\lVert #2\rVert_{#1}}
\newcommand\snorm[1]{\lVert#1\rVert_\infty}

\numberwithin{figure}{section}

\begin{document}

\title[a Levinson type theorem on complex domains]{Asymptotic solutions for linear ODEs with not-necessarily meromorphic coefficients: a Levinson type theorem on complex domains, and applications}
\author[G.\,Cotti, D.\,Guzzetti, D.\,Masoero]{G.\,Cotti$\>^{\circ}$, D.\,Guzzetti$\>^{\star,\circledast}$, D.\,Masoero$\>^{\circ,\oslash}$}
\address{Departamento de Matem\'atica, Instituto Superior T\'ecnico, Av.\,Rovisco Pais
1049-001 Lisboa, Portugal}
\address{Instituto Superior de Agronomia, Universidade de Lisboa, Tapada da Ajuda, 1349-017 Lisboa, Portugal}
\address{Scuola Internazionale Superiore di Studi Avanzati, Via Bonomea 265, 34136 Trieste, Italy\newline INFN Sezione di Trieste, via Valerio 2, 34127 Trieste, Italy\vskip1,5mm}
\email{giordano.cotti@tecnico.ulisboa.pt, guzzetti@sissa.it, davidem@isa.ulisboa.pt}
\subjclass[2020]{Primary: 34M03, 34M30; Secondary: 34M25, 34M35. }
\maketitle
{\small\begin{center}
\textit{ $^\circ\>$Grupo de F\'isica Matem\'atica \\
Departamento de Matem\'atica, Instituto Superior T\'ecnico\\ Av. Rovisco Pais
1049-001 Lisboa, Portugal\/\\
\vskip1.5mm
$^\star\!$ Scuola Internazionale Superiore di Studi Avanzati\\
Via Bonomea 265, 34136 Trieste, Italy\\
\vskip1.5mm
$^\circledast\!$ INFN Sezione di Trieste\\
Via Valerio 2, 34127 Trieste, Italia\\
\vskip1.5mm
$^\oslash\!$  Instituto Superior de Agronomia, Universidade de Lisboa\\
Tapada da Ajuda, 1349-017 Lisboa, Portugal}
\end{center}}
{\let\thefootnote\relax
\footnotetext{\vskip5pt 
\noindent
$^\circ\>$\textit{ E-mail}:  giordano.cotti@tecnico.ulisboa.pt, guzzetti@sissa.it, davidem@isa.ulisboa.pt}}

\begin{abstract}
In this paper, we consider systems of linear ordinary differential equations, with analytic coefficients on big sectorial domains, which are asymptotically diagonal for large values of $|z|$. Inspired by \cite{Lev48}, we introduce two conditions on the dominant diagonal term (the $L$-{\it condition}) and on the perturbation term (the {\it good decay condition}) of the coefficients of the system, respectively. Assuming the validity of these conditions, we then show the existence and uniqueness, on big sectorial domains, of an {\it asymptotic} fundamental matrix solution, i.e.\,\,asymptotically equivalent (for large $|z|$) to a fundamental system of solutions of the unperturbed diagonal system. Moreover, a refinement (in the case of subdominant solutions) and a generalization (in the case of systems depending on parameters) of this result are given. 

As a first application, we address the study of a class of ODEs with not-necessarily meromorphic coefficients, the leading diagonal term of the coefficient being a generalized polynomial in $z$ with real exponents. We provide sufficient conditions on the coefficients ensuring the existence and uniqueness of an asymptotic fundamental system of solutions, and we give an explicit description of the maximal sectors of validity for such an asymptotics. Furthermore, we also focus on distinguished examples in this class of ODEs arising in the context of open conjectures in Mathematical Physics relating Integrable Quantum Field Theories and affine opers ({\it ODE/IM correspondence}). Notably, our results fill two significant gaps in the mathematical literature pertaining to these conjectural relations.

Finally, as a second application, we consider the classical case of ODEs with meromorphic coefficients. Under an {\it adequateness} condition on the coefficients (allowing ramification of the irregular singularities), we show that our results reproduce (with a shorter proof) the main asymptotic existence theorems of Y.\,Sibuya \cite{Sh62,Sh68} and W.\,Wasow \cite{Was65} in their optimal refinements: the sectors of validity of the asymptotics are maximal, and the asymptotic fundamental system of solutions is unique. 

\end{abstract}

\newpage
\tableofcontents

\section{Introduction}

\noindent{1.1.} {\bf Asymptotically diagonal systems. }Consider the linear ordinary differential equation
\beq\label{eq1} \frac{dY}{dz}=\left(\La(z)+R(z)\right)Y,\qquad z\in U,
\eeq where 
\begin{enumerate}
\item $U\subseteq\mathbb C$ is an unbounded simply-connected domain given by the union of lines, whose distance from 0 is greater than a fixed $a>0$ and whose slopes vary in an open interval 
 $\left]\tau_{\rm min};\tau_{\rm max}\right[\subseteq\R$, see Figure \ref{19ottobre2023-1};
\item $\La(z), R(z)$ are analytic $n\times n$ matrix-valued functions on $U$;
\item $\La(z)={\rm diag}(\La_{11}(z),\dots,\La_{nn}(z))$ is a diagonal matrix.
\end{enumerate} 
The arguments of points $z$ of $U$ vary in the interval $\left] \tau_{\min};\tau_{\rm max}+\pi\right[$. The choice of a determination of the argument of points of $U$ allows us to actually see $U$ as a subdomain\footnote{In the main body of the paper, the more general case $U\subseteq \Tilde{\C^*}$ will be considered. Here, for the sake of simplicity of exposition, we limit ourselves to the case $U\subseteq \C$.} of $\widetilde{\mathbb C^*}$, the universal cover of $\mathbb C^*:=\mathbb C\setminus\{0\}$. 
\vskip1,5mm
\begin{figure}
\centerline{\includegraphics[width=0.5\textwidth]{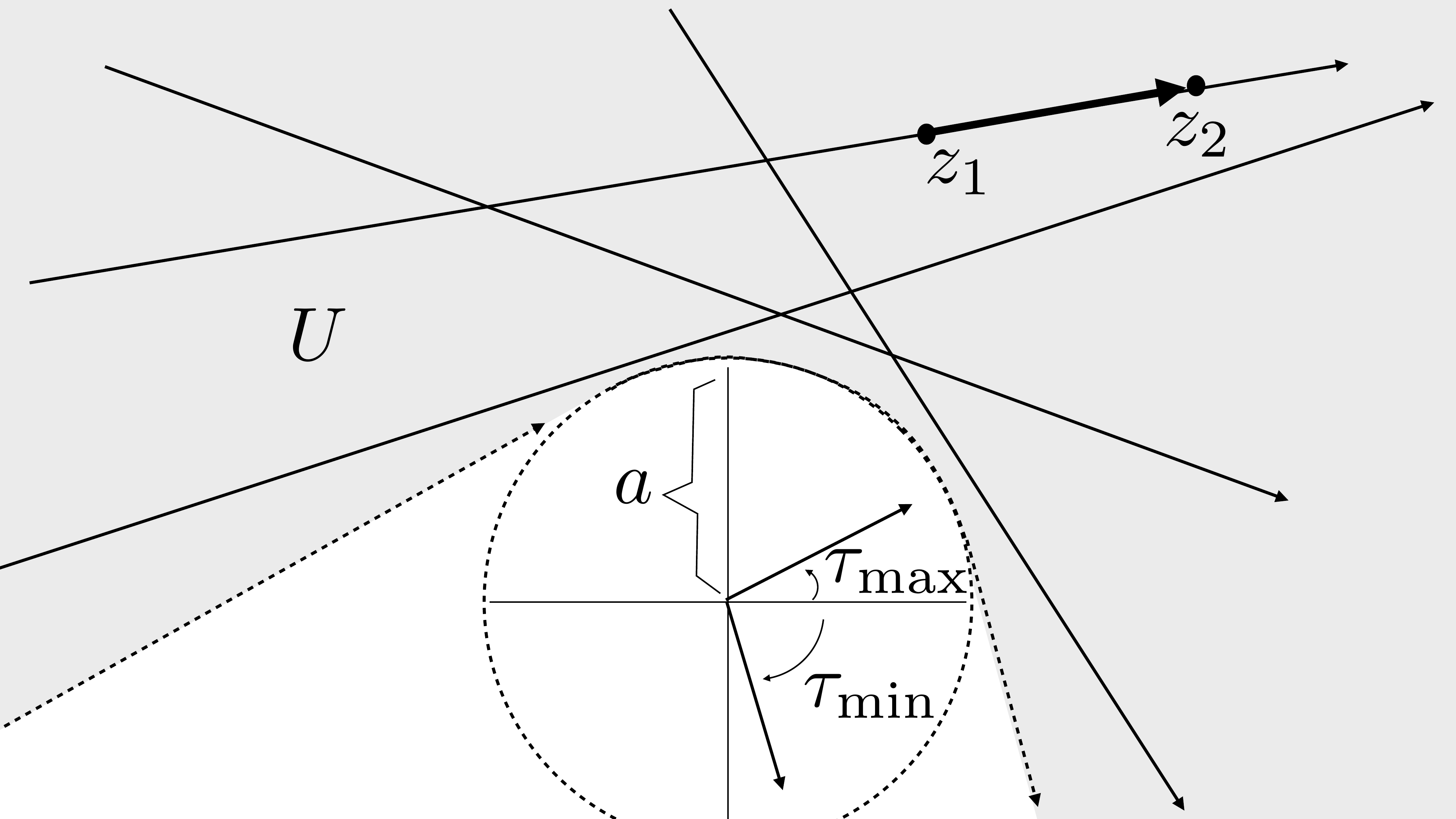}}
\caption{The domain $U$, union of lines (some are represented) with natural orientation. Also an example of an oriented segment $[z_1,z_2]$ is represented.}
\label{19ottobre2023-1}
\end{figure}
Assume that the summand $R$ is ``small'' (in a sense to be defined) in the regime $|z|\to \infty$. In such a case, equation \eqref{eq1} can be seen as a ``perturbation'' of the diagonal equation
\beq
\label{eq2} \frac{dY}{dz}=\La(z)Y.
\eeq It is then natural to investigate up to which extent solutions of \eqref{eq2} approximate solutions of \eqref{eq1} for big values of $|z|$.
A fundamental matrix solution for the unperturbed equation \eqref{eq2} is 
\[Y_o(z):=\exp\int_{z_o}^z\La(\zeta)\d \zeta,\quad \text{for an arbitrarily fixed }z_o\in U.
\]
In this paper, we address the following question:

\begin{center}
{\bf Q: }\emph{Under which conditions on the triple $(U,\La,R)$, does there exist a unique fundamental matrix solution $Y\colon U\to GL(n,\C)$ of \eqref{eq1} asymptotically equivalent to $Y_o$, that is
\[Y(z)Y_o(z)^{-1}=I_n+o(1),\quad z\to\infty,\quad z\in U?
\]}
\end{center}

Inspired by the work of N.\,Levinson \cite{Lev48} for linear ODEs on the real line, our main result, Theorem \ref{mainth}, will provide an answer to question {\bf Q}. It will give, at the same time, sufficiently general conditions on $(U,\La,R)$ to be applicable in a wide range of situations, as we are going to expose.
\vskip2mm
\noindent{1.2.} {\bf Description of the main results. } In order to formulate the main results of the paper, we need first to introduce two preliminary definitions.
\vskip1,5mm
The first notion we want introduce is the $L${\it -condition} for the pair $(U,\La)$.
\vskip1,5mm
Any line $\ell\subseteq U$ is the boundary of an half-plane $\mathbb H\subseteq U$. Hence, the natural complex orientation of $\mathbb H$ induces an orientation of $\ell=\der\mathbb H$. We will refer to this orientation as the {\it natural orientation} of $\ell$.

Given $z_1,z_2\in U$, denote by $[z_1,z_2]$ the oriented segment $\{(1-t)z_1+tz_2,\,0\leq t\leq 1\}$, with orientation from $z_1$ to $z_2$. Let $\mathcal P_U$ be the set of all oriented segments in $U$, obtained by truncating the naturally oriented lines in $U$. See Figure \ref{19ottobre2023-1}.
\vskip1,5mm
We will say that the pair $(U,\La)$ satisfies the\footnote{Here $L$  stands for {\it Levinson}, being the condition above our generalization to the complex domain of some simpler conditions, due to N.\,Levinson, required in the real case, see \cite{Lev48},\cite[Ch.1 and references therein]{East89}.} $L${\it -condition} if for any pair $(i,j)$, with $i,j=1,\dots,n$ and $i\neq j$, the set of real numbers
\[A_{i,j}:=\left\{{\rm Re}\int_{[z_1,z_2]}\left(\La_{ii}(\zeta)-\La_{jj}(z)\right)\d\zeta,\quad [z_1,z_2]\in\mathcal P_U\right\}\subseteq \R
\]has an upper bound and/or a lower bound.
\vskip1,5mm
The second notion we want to introduce is the {\it good decay condition} for the pair $(U,R)$.
\vskip1,5mm
Given a line $\ell\subseteq U$, denote by $d_\ell$ the distance of $\ell$ from the origin $0\in\C$.
We will say that $(U,R)$ satisfies the {\it good decay condition} if (the restriction of) $R$ is $L^1$-integrable along any line $\ell\subseteq U$, and moreover
\[|\!| R|_\ell|\!|_1\to 0, \quad \text{for }d_\ell\to+\infty,\quad |\!| \cdot|\!|_1\text{ being the 1-norm on }L^1(\ell, dz).
\]
We are now able to state the main result of the paper.
\begin{thm}[Theorem \ref{mainth}]\label{introthm1}
If $(U,\La)$ satisfies the $L$-condition, and $(U,R)$ the good decay condition, then the differential equation \eqref{eq1} admits a unique fundamental matrix solution $Y\colon U\to GL(n,\C)$ such that
\beq\label{eq3}
Y(z)=\left(I_n+o(1)\right)\exp\int_{z_o}^z\La(\zeta)\d\zeta,\quad z\to\infty,\quad \al\leq \arg z\leq \bt+\pi,
\eeq for arbitrary $\al,\bt$ satisfying $\tau_{\rm min}<\al\leq\bt<\tau_{\rm max}$.
\end{thm}
For short, we will refer to the fundamental matrix solution $Y$ satisfying \eqref{eq3} as the {\it asymptotic solution} of \eqref{eq1}.
\vskip1,5mm
In Sections \ref{20settembre2023-2} and \ref{20settembre2023-1}, we also obtain a refinement and a generalization, respectively, of the main result above. 
\vskip1,5mm
On the one hand, in Section \ref{20settembre2023-2}, we focus on the study of those columns of $Y(z)$ defining {\it subdominant} solutions of \eqref{eq1} on $U$. These are vector-valued solutions $y_{\rm sub}(z)$ of \eqref{eq1} corresponding to the $j_o$-th columns of $Y(z)$, where $j_o\in\{1,\dots,n\}$ is such that all the sets $A_{i,j_o}$ have lower bound for any $i=1,\dots,n$ with $i\neq j_o$. In particular, the solutions $y_{\rm sub}(z)$ are dominated by any other non-zero vector-valued solutions $y(z)$, in the sense that\footnote{Here $|\cdot|$ is an arbitrarily fixed norm on the space of matrices $M(n,\C)$.} $|y_{\rm sub}(z)|=O(|y(z)|)$ for $z\to\infty$ along any line $\ell\subseteq U$. In our Theorem \ref{subth} we show that the validity of the asymptotic estimate given by the $j_o$-column of both sides of \eqref{eq3} is valid in a bigger domain $\widetilde U$ than $U$, under mild assumptions on $(\widetilde U,\La,R)$.
\vskip1,5mm
On the other hand, in Section \ref{20settembre2023-1} we address the case of equation \eqref{eq1} with coefficients holomorphically depending on further parameters varying in an open connected subset of $\C^m$, with $m\geq 1$. Under uniform analogs of both the $L$-condition and the good decay condition, together with a uniform integrability assumption, our Theorem \ref{mainth-bis} provides a parametric analog of Theorem \ref{introthm1}.
\vskip1,5mm

We assert that Theorem \ref{mainth} introduces at least two significant innovations to the field. Firstly, it addresses differential equations \eqref{eq1} with coefficients that may not be meromorphic, and, in fact, does not require their analytic continuations to be single-valued, or even to exist at all around the singularity. Secondly, it is important to emphasize that our result not only predicts the existence of asymptotic solutions but also establishes their uniqueness. The former innovation has practical applications in addressing challenging open problems in mathematical physics: we will explain how Theorem \ref{mainth} can be used to fill gaps in the literature on the {\it ODE/IM correspondence}\,\footnote{\,IM stands for ``Integrable Models''.}. The latter, when applied to the classical case of ODEs with meromorphic coefficients on $\Pb^1(\C)$, allows us to derive more robust versions of classical results regarding the existence and uniqueness of asymptotic solutions.

\vskip2mm

\noindent{1.3.} {\bf Applications. } As a first application of our results, we address the study of a family of linear ODEs with not-necessarily meromorphic coefficients on $\mathbb P^1(\C)$. These equations are indeed defined on $U:=\{z\in\Tilde{\C^*}\colon |z|>a\}$ and are of the form \eqref{eq1}, with 
\[\Lambda(z)=\sum_{k=0}^h \Lambda^{(k)} z^{\sigma_k -1},
\quad
              \Lambda^{(k)}\text{ diagonal matrices},
 \quad 
              1=\sigma_0>\sigma_1>\dots>\sigma_h>0,\quad \si_i\in\R_{>0},
\]
and $R\colon U\to M(n;\C)$ analytic satisfying the good decay condition on $U$ (this holds, for example, if $R=O(|z|^{-1-\dl})$ on $U$, for some $\dl>0$).
\vskip2mm
For ODEs of such a form,  
a notion of {\it Stokes rays} can be defined, see Definition \ref{9giugno2023-2}. In Section \ref{secadeqtripl1}, under adequate conditions\footnote{In particular $\La$ has to admit adequate tuples, as defined in Definition \ref{18settembre2023-3}.} on $\La$, we show that these Stokes rays can be used to select the maximal sectors of $U$ on which the matrix $\La$ satisfies the $L$-condition. Consequently, for adequate $\La$, on these maximal sectors the existence and uniqueness of the asymptotic solution is guaranteed by Theorem \ref{mainth}, see Theorem \ref{5giugno2023-3}. Moreover, we also specialize to this case our results on subdominant solutions (Theorem \ref{22giugno2023-5}) and to the case of equations depending on parameters (Theorem \ref{23ottobre2023-1}).
\vskip1,5mm
Examples of equations of the form above arise in the context of the {\it ODE-IM correspondence}. The ODE/IM correspondence is a mostly conjectural duality between Integrable Quantum Field Theories (IM) and linear
ordinary differential operators with analytic coefficients (ODE).
It states that every solution of the Bethe (Ansatz) Equations
originating from a Quantum Field Theory can be represented
as the spectral determinant of a linear differential operator, see e.g. \cite{CM1,CM2,DT,DDT,BLZ,FF,FH,LZ,MRV,GLVW}. Therefore, any state of any integrable quantum field theory should correspond to a linear differential
operator.
\vskip1,5mm
Oftentimes, quantum field theories come in families parametrized by moduli, one of which can be identified with the degree at $z=\infty$ of the coefficient of the corresponding linear differential operator. As a result, this degree is not necessarily a rational number.

Because of this and
due to the lack, prior to the present paper, of a theory of not-necessarily meromorphic ODEs, the literature on
the ODE/IM correspondence suffers of
two fundamental gaps: the first is the construction of a distinguished basis of solutions at $z=\infty$, and the
second is the description of the maximal sector for subdominant solutions (the latter solutions were indeed constructed by ad-hoc methods \cite{MRV}).
\vskip1,5mm
In Section \ref{ODEIM}, these gaps are filled by Theorem \ref{thm:matrixodeim}, deduced as a special case of the results of Sections  \ref{secstok}--\ref{18settembre2023-6}. Furthermore, we also apply this new result to the most studied instance of the ODE/IM correspondence, the duality between
quantum $\mathfrak{g}^{(1)}$-Drinfeld--Sokolov hierarchy \cite{BLZ2,FF0} (also known
as Quantum $\mathfrak{g}$-KdV) and ${}^L\mathfrak{g}^{(1)}$-opers on $\C^*$.
\vskip1,5mm
As a second application, we consider  the classical case of linear ODEs with meromorphic coefficients on $\mathbb P^1(\C)=\C\cup\{\infty\}$, namely 
\beq\label{eq4}
\frac{dY}{dz}=z^{r-1}A(z)Y,\quad A(z) \text{ analytic on }\mathbb P^1(\C)\setminus\{|z|>a\}, 
\eeq
where $r\in\N$ is the {\it Poincar\'e rank} of the irregular singularity at $z=\infty$. This equation admits formal fundamental systems of solutions of the form
\beq\label{eq5}
Y_{\rm for}(z)=F(z)G(z),
\eeq
where $G(z)$ is a well-defined analytic function of some root\footnote{The irregular singularity $z=\infty$ is said to be {\it ramified} if $p>1$, and $p$ is called  {\it index of ramification}. This is the {\it Fabry phenomenon} \cite{Fab85}.} $z^{1/p}$, and $F(z)$ is a formal power series of $z^{-1}$, see e.g.\,\,\cite{BJL79-2,MS16,Lod16}. Though the series $F(z)$ typically is not convergent, the purely formal solution $Y_{\rm for}(z)$ is well-known\footnote{In the monography \cite{Was65} this result is referred to as the {\it Main Asymptotic Existence Theorem}. This result marks the culmination of a very long sequence of investigations dating back to the 19th century. See the next section.} to prescribe the asymptotics (in the sense of Poincar\'e) of genuine fundamental systems of solutions of \eqref{eq4} in any sector of 
$\{z\in\Tilde{\C^*}\colon |z|>a\}$, with sufficiently small angular opening, see e.g.\,\,\cite{Sh62,Was65,Sh68,BJL79-3,Sib90}. On small sectors as such, however, the uniqueness of these fundamental systems of solutions is not guaranteed by their asymptotics \eqref{eq5} only. In the ramified case ($p>1$), it is a difficult problem to give an {\it a priori} estimate of the optimal angular opening guaranteeing both the existence and uniqueness of an asymptotical fundamental matrix solution.
\vskip1,5mm
In Section \ref{secmaetw},  
by replacing equation \eqref{eq4} with a gauge equivalent one (the gauge equivalence given by an arbitrary truncation of the series $F(z)$), we recover an equation of the form \eqref{eq1}. Under adequate conditions\footnote{In particular, the system \eqref{eq4} must admit {\it $r$-adequate} tuples, as defined in Definition \ref{radeq}.} on the coefficient $A$, we show that it is possible to characterize the maximal sectors on which the $L$-condition holds, and consequently Theorem \ref{introthm1} applies. This leads to a very short proof of the {\it Main Asymptotic Existence Theorem} of \cite{Was65},  in its optimal refinement (i.e.\,\,on maximal sectors), see Theorem \ref{5giugno2023-7}. Hence, for these adequate systems (including ramified cases), we prove the existence and uniqueness of a fundamental matrix solution with asymptotics \eqref{eq5} on optimal sectors.
\vskip2mm
\noindent{1.4.} {\bf Some historical remarks. }We would like to conclude this introduction by briefly contextualizing our results into a historical perspective\footnote{\,The interested reader can found a detailed historical analysis on the study of asymptotical solutions of ODEs, in the period 1817--1920, in the essay \cite{Sch77}.}. The study of asymptotic solutions for linear ODEs has its origins in pioneering works focused on ``{\it approximate}'' solutions. Although the concept of ``approximation'' was loosely defined, and the methods employed lacked rigor, the contributions of G.\,Carlini \cite{Car17}, J.\,Liouville \cite{Lio37}, and G.\,Green \cite{Gre37} represented the initial explorations in this field.
\vskip1,5mm
It was only with the works of G.\,Stokes \cite{Sto56,Sto64,Sto71} and H.\,Hankel \cite{Han68} that these methods were refined, leading to the introduction of series approximations for actual ODE solutions. Stokes, on one hand, considered formal power series solutions as approximations, in the sense that their arbitrary truncations almost satisfy the original ODE. Hankel, on the other hand, introduced formal power series solutions that were deemed approximate because they were {\it semi-convergent} (a concept previously defined by A.M.\,Legendre \cite[pag.\,13]{Leg98}) to true analytic solutions. Hankel also pioneered the use of exact analytic solutions to compute their semi-convergent expansions, establishing a direct connection between the two.
Additionally, it is worth noting that both Stokes and Hankel independently discovered the {\it Stokes phenomenon} during their research and provided explanations for it, although they did not delve further into this topic.
\vskip1,5mm
A significant milestone in this area of research was marked by the contributions of H.\,Poincar\'e. Building upon substantial advancements in the general theory of ordinary differential equations made between 1833 and 1886 by a multitude of mathematicians\footnote{Among them we mention: A.-L.\,Cauchy, K.\,Weierstrass, C.\,Briot and J.\,Bouquet, L.\,Fuchs, L.W.\,Thom\'e, F.\,G.\,Frobenius, M.E.\,Fabry.}, Poincar\'e achieved a pioneering feat. He replaced the vague concept of an ``approximate solution'' with the precise notion of an ``{\it asymptotic solution}'' by introducing a theory of asymptotic series. Focusing on $n^{\rm th}$-order scalar ODEs with polynomial coefficients and an irregular singularity at $z=\infty$ of rank 1, as detailed in \cite{Poi85, Poi86}, Poincar\'e accomplished the following:
\begin{enumerate}
\item He derived exact analytic solutions using a Laplace transform technique, resulting in integral representations applicable in the vicinity of $z=\infty$.
\item He introduced a rigorous concept of asymptotic series and computed the asymptotic expansions of the aforementioned Laplace integrals for a fixed argument $\arg z$.
\item He identified these asymptotic expansions as formal solutions of the given ODE, as previously described\footnote{In the theory of Thom\'e and Fabry, formal solutions, essentially always reducible to the form \eqref{eq5}, were referred to as {\it normal}, {\it logarithmic normal}, or (in the ramified case) {\it anormal} {\it series solutions} of the linear ODE.} by L.W. Thom\'e \cite{Tho77, Tho84} and M.E. Fabry \cite{Fab85}.
 \end{enumerate}

Poincar\'e's contributions marked a pivotal moment in the field, establishing the foundation for the study of asymptotic solutions in the context of ODEs, which quickly gained significant attention. 
From 1897 to 1910, J.\,Horn made substantial advancements in this field, drawing inspiration from Poincar\'e's ideas and producing a series of significant papers.  
Describing all of Horn's technical innovations and refinements concisely is a formidable task. However, it's worth noting that Horn extended Poincar\'e's findings to address higher rank irregular singularities, initially focusing on $2^{\rm nd}$-order linear ODEs \cite{Hor97a}, later expanding to arbitrary $n^{\rm th}$-order linear ODEs \cite{Hor00}. Most notably, he tackled the challenge of extending the validity of asymptotic solutions into {\it angular sectors} centered around irregular singularities, see e.g.\,\cite{Hor97b,Hor98}. Remarkably, Horn's work in \cite{Hor00} established -- for $n^{\rm th}$-order scalar ODEs (with coefficients known only asymptotically near an irregular singularity at $z=\infty$)-- the existence of a fundamental system of solutions whose asymptotics expansions are given by formal solutions along any ray in the $z$-plane, albeit with a finite number of exceptions. 
Horn's contributions were central in broadening Poincar\'e's ideas across a wide range of equations. His versatile techniques shed new light on the nature of asymptotic solutions. In the realm of ODE research spanning from the 1890s to the 1910s, Horn's work distinctly stands out as superior to that of his contemporaries.
\vskip1,5mm
All the works discussed thus far have addressed the case of scalar differential equations. It is only through the work of G.\,Birkhoff that the theory of asymptotic solutions for systems of linear ODEs achieved a clear formulation. By extending the concepts of formal series solutions and the rank of an irregular singularity to systems of linear ODEs, Birkhoff's primary interest shifted towards constructing asymptotic solutions in the vicinity of irregular singularities of arbitrary rank. He accomplished this through a suitable modification of Poincar\'e's and Horn's techniques based on Laplace transforms.

More precisely, in \cite{Bir09}, Birkhoff considered the class of systems of linear ODEs which are equivalent, for large $|z|$, to a {\it canonical system} of the form
\beq\label{eq6}
z\frac{dY}{dz}=P(z)Y,\quad P(z)\text{ matrix with polynomial entries},
\eeq
with an irregular singularity at $z=\infty$ of arbitrary rank.
Initially, he constructed formal solutions $Y_{\rm for}(z)$ for equation \eqref{eq6}, which can essentially be reduced to the form \eqref{eq5}. Subsequently, he demonstrated the existence of fundamental systems of solutions $Y(z)$, with asymptotic behavior $Y_{\rm for}(z)$, within any sector centered around $z=\infty$ and with a sufficiently small angular opening. Furthermore, Birkhoff provided a Laplace integral representation, similar to those used by Poincar\'e, for each entry of $Y(z)$. 
Birkhoff also claimed that {\it any} system of linear ODEs is holomorphically equivalent to a canonical one: this is false, as shown later by F.R.\,Gantmacher \cite{Gan59}, P.\,Masani \cite{Mas59} and H.L.\,Turritin \cite{Tur63}. This approximately represents the state of the art on asymptotic solutions as described in the monography \cite{Inc26}.
\vskip1,5mm
In a further series of papers \cite{Hor12,Hor16,Hor44a,Hor44b}, the already mentioned J.\,Horn showed the following result for linear systems of ODE: if the leading term of the coefficient matrix, at a singularity of second kind, has all distinct eigenvalues, then for sectors with sufficiently large opening one has {\it uniqueness} of the asymptotic fundamental system of solutions, which can be represented as a Laplace integral.  
These are, to the best of our knowledge, the first references in the literature of ODEs where the uniqueness problem for asymptotic solutions has been addressed.
\vskip1,5mm
Following Horn's and Birkhoff's contributions, the literature saw an abundance of results, with a progression of refinements and generalizations. Notable works in this regard include those by W.J.\,Trjitzinsky \cite{Tri33,Tri35}, J.A.\,Lappo-Danilevsky \cite{Lap34}, M.\,Hukuhara \cite{Huk37}, and H.L. Turritin \cite{Tur55}. 
The most sophisticated formulations of these results, with complete proofs, can be found in the subsequent references \cite{CL56} and \cite{Sh62,Was65,Sh68}. In all these works, the proof of the {\it existence} of asymptotic solutions (with Poincar\'e asymptotics given by a formal solution) is given on small sectorial domains. The details of the proofs are quite delicate, and to the best of our knowledge the references \cite{Sh62,Was65,Sh68} are essentially the only ones in which complete proofs are given.
\vskip1,5mm

On the one hand, in \cite{CL56}, E.A.\,Coddington and N.\,Levinson's proof   begins by establishing the result on the {\it positive real axis}, building upon  { N.\,Levinson's  seminal  work \cite{Lev48}. The latter   treats the perturbation to the diagonal part of the differential system as an inhomogeneous term, with a solution given by variation of parameters, leading to an integral equation where a fixed-point argument is applied. By showing that a solution of the integral equation is also a solution of the differential equation with prescribed asymptotics (but not vice versa), the existence result is proved. Subsequently, \cite{CL56} extends the result to the complex sectorial case (under genericity assumptions of the coefficients), applying the {\it real argument  radially}. It is  so shown that the asymptotic behaviour of a solution along a radial direction can be extended to a {\it small}  sector in the $z$-plane.} This extension relies on theorems of A.H.\,Phragm\'en and E.L.\,Lindel\"of, as described in \cite[Chapter 5, Sections 4 and 5]{CL56}. 

On the other hand, Y.\,Sibuya \cite{Sh62,Sh68} and W.\,Wasow \cite{Was65} used a different approach to  establish the existence result. The initial step involves formulating a nonlinear ODE for a gauge transformation, which can be formally solved by $F(z)$ in \eqref{eq5}. This gauge transforms  the original equation \eqref{eq4} into a simpler differential equation, solved by the function $G(z)$  in \eqref{eq5}. The subsequent step is to demonstrate the existence of a genuine solution $\Hat F(z)$, with asymptotic behavior $F(z)$, for the nonlinear equation. Such a solution $\Hat F(z)$ is constructed entry-wise, with an integral equation for each matrix entry. In particular, each integral equation is set on a carefully chosen $z$-dependent contour. In order to reach the final  $G(z)$, a finite sequence of intermediate such gauge transformations is generally required, each of which essentially reduces the sector of applicability of the existence result.  As a consequence, in the general ramified case, one can only assert that existence holds on a sufficiently small sector (see, for example, Theorem 19.1  in \cite{Was65}). The {\it minimal}  angular  opening  of the sector was  later established  in   \cite[Sect.\,4]{BJL79-3}. 

Our approach differs in the following way: after reducing \eqref{eq4} to the form \eqref{eq1} via a  gauge equivalence given by a truncation of $F$, we set an integral equation directly for the solution $Y(z)$ itself on $z$-dependent 
straight lines in a domain $U$ (on which both the $L$-condition and the good decay condition are satisfied). 
Contrary to the strategy in \cite{CL56}, these integration contours do not respect the radial symmetry of $U$ but instead  ``span''  it. This allows us to deduce not only the existence but also the uniqueness of $Y(z)$ on $U$. 
Indeed,  as in Levinson's approach, a fixed point argument of Banach--Caccioppoli type provides uniqueness for the solution of the integral equation. However, the spanning of the domain by integration contours is crucial to prove the {\it equivalence} between the solution of the integral 
equation and that of the differential equation {\it with asymptotic conditions}. Consequently, the uniqueness of solutions for the differential equation is established.  The above equivalence  cannot be proved by the real or radial argument in Levinson's approach (see the proof of Lemma \ref{lemdiffeqintegeq} and Remark \ref{30settembre2024-1} for further details).   

In conclusion, our approach offers a shorter proof of the Sibuya-Wasow asymptotic existence theorem and provides a stronger conclusion, namely the uniqueness of the asymptotic solutions. Moreover, it is more general, as it applies to a broader class of equations with non-meromorphic coefficients.

\vskip1,5mm 
 The approach introduced in the seminal paper by N. Levinson  \cite{Lev48} has served as a wellspring of inspiration for numerous subsequent works and developments, to which we direct the reader for further details and references in \cite{East89}. It is worth noting that the Russian literature has further elaborated upon this methodology in various specific instances, with ad hoc adaptations to specific cases, see for example \cite{Rap51,Fed66,Fed93}. In particular, in  \cite[Chapter 5]{Fed93}, it was applied to the  WKB analysis, which subsequently inspired a weaker and particular version of our results on subdominat solutions in the ODE/IM literature, see \cite{MRV}.

\vskip1,5mm
Let us conclude by mentioning that, for linear meromorphic systems of ODEs, the existence and uniqueness (on sufficiently large sectors) of asymptotic fundamental systems of solutions has been subsequently proved with techniques of $k$-{\it summability} and {\it multi-sommability} applied to formal power series solutions, see \cite{Ram78,Ram80,Eca81,Bra91,Bal99,MS16,Lod16} and references therein. We stress that for the wide class of differential equations considered in this paper the study of formal solutions is, at present, unapproachable. Already the case of coefficients with an essential singularity represents an interesting open problem: see the interesting paper \cite{BCJ15}, in which a first description of formal solutions is given for several classes of such systems.

\vskip2mm
\noindent{1.5.} {\bf Structure of the paper. } 
\vskip1,5mm
Sections \ref{secnot}-\ref{secprmth} -- In these sections, we establish our notations and introduce the $L$-decay and good decay conditions. The main result of our paper, Theorem \ref{mainth}, is proven here.
\vskip1,5mm
Sections \ref{20settembre2023-2}-\ref{secsubpar} -- In this part, we both refine the main result for subdominant solutions and extend it to the parametric case.
\vskip1,5mm
Sections \ref{secstok}-\ref{sec33} -- Here, we delve into a family of ODEs with coefficients that may not necessarily be meromorphic, specifically, generalized polynomials in $z$ with real exponents. We introduce the concepts of Stokes rays and adequate tuples and present and prove Theorem \ref{5giugno2023-3} as the first application of our main result. 
\vskip1,5mm
Subdominant and parametric cases are addressed in Sections \ref{sec34} and \ref{18settembre2023-6}, respectively. Additionally, in Section \ref{ODEIM}, we explore the applications of these results in the context of the ODE/IM correspondence.
\vskip1,5mm
Sections \ref{secforsol}-\ref{secmaetw} -- In these sections, we shift our focus to classical case of meromorphic linear systems on $\mathbb P^1$. We review existing results on formal solutions, including ramified cases, introduce the concept of $r$-adequateness of tuples, and establish the existence and uniqueness of asymptotic solutions on large sectors. This provides a refinement of the main asymptotic existence theorems of Y.\,Sibuya and W.\,Wasow.
\vskip1,5mm
Sections \ref{sec4sub}-\ref{sec4par} -- Finally, in these sections, we address the subdominant case and briefly touch on the parametric case.

\vskip2mm
\noindent{\bf Acknowledgements. } This research was supported by the FCT projects UIDB/00208/2020, 2021.01521.CEECIND, 2021.00091.CEECIND, and 2022.03702.PTDC (GENIDE). 
The authors are members of the  COST Action CA21109 CaLISTA.
D.\,Guzzetti  is a member of  the European Union's H2020 research and innovation programme under the Marie Sk\l{}odowska-Curie grant No. 778010 {\it IPaDEGAN}. He is  a member  of INdAM/GNFM, and of
 the Project ``Mathematical Methods in Non
Linear Physics'' (MMNLP), Commissione Scientifica Nazionale 4 - Fisica
Teorica (CNS4) of the Istituto Nazionale di Fisica Nucleare (INFN). 
He is member of the PRIN 2022 (2022TEB52W) - PE1 - ``The charm of integrability: from nonlinear waves to random matrices'', and of the EU project Caligola HORIZON-MSCA-2021-SE-01, Project ID: 101086123. 
Moreover, D.\,Guzzetti is grateful for hospitality to the Grupo de F\'isica Matem\'atica and to the Faculdade de Ci\^encias da Universidade de Lisboa. D.\,Guzzetti thanks prof. T.\,Mochizuki for suggesting Example \ref{17novembre2024-1}. Finally, the authors thank the anonymous referee for several suggestions, which significantly improved the exposition of this paper.
\vskip2mm
\noindent{\bf Conflict of interest.} On behalf of all authors, the corresponding author states that there is no conflict of interest.
\vskip2mm
\noindent{\bf Data availability statement. }No datasets were generated or analysed during the current study.

\section{A Levinson type theorem on the complex domain}\label{sec2}
\subsection{Notations and preliminary notions.}\label{secnot} Let $\Pi\colon\Tilde{\C^*}\to\C^*$ be the universal cover of $\C^*$. Points of $\Tilde{\C^*}$ can be uniquely represented in polar form $z=\rho e^{\ic \theta}$, where $\rho\in\R_{>0}$ is the modulus $|z|$ of $z$, and $\theta\in\R$ is the argument $\arg z$.
\vskip1,5mm
We call {\it lines} those curves on $\Tilde{\C^*}$ whose projection are straight lines on $\C^*$. These admit the
parametrisation
\[\rho=\frac{a_1}{a_2\cos\theta+a_3\sin\theta},\quad a_1,a_2,a_3\in\R,\quad a_1\neq 0,\,\,(a_2,a_3)\neq (0,0).
\]
Given $\phi\in\R$ and $a\in\R_{>0}$, we define the closed half-plane $\H_{\phi,a}$, and the  line $\ell_{\phi,a}$, by the formulas 
$$
\H_{\phi,a}:=\left\{z\in\Tilde{\C^*}\colon { \arg z\in \left]\phi-\frac{\pi}{2};\phi+\frac{\pi}{2}\right[},~|z|\geq \frac{a}{\cos(\phi-\arg z)}\right\},\quad  \ell_{\phi,a}:=\der\H_{\phi,a}.
$$
The half-plane $\H_{\phi,a}$ inherits the orientation of $\Tilde{\C^*}$, and consequently its boundary $\ell_{\phi,a}:=\der\H_{\phi,a}$ is oriented (with angular direction {$\tau:=\phi-\pi/2$}).  See Figure \ref{26giugno2023-1}. This allows to introduce a natural total order relation $\leq$ on points of $\ell_{\phi,a}$, and to distinguish two infinite points $\pm\infty$ on it.

\begin{figure}
\centerline{\includegraphics[width=0.5\textwidth]{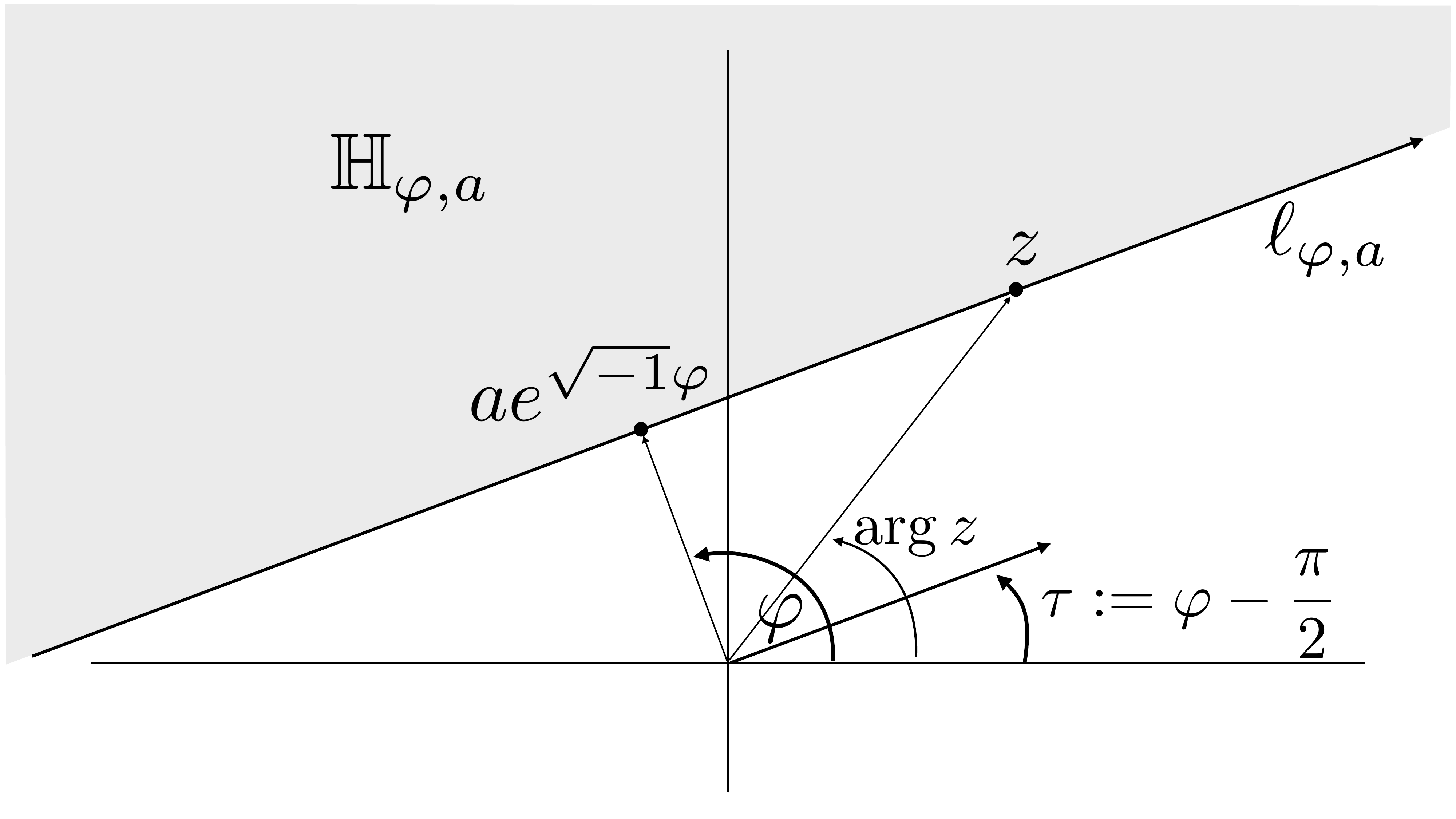}}
\caption{The half plane $\H_{\phi,a}$. The angles are represented mod $2\pi$. Equivalently, the figure may be intended as the representation of the projection  onto $\C$ of the sheet of $\Tilde{\C^*}$ where $\H_{\phi,a}$ lies.}
\label{26giugno2023-1}
\end{figure}

\vskip1,5mm
For any continuos function $f$ on $ \ell_{\phi,a}$ and any $p\geq 1$, we define the $L^p$-norm $\pnorm{p}{f}\in\R_{\geq 0}\cup\{+\infty\}$ as
\beq
\pnorm{p}{f}=\left(\int_{\left]\phi-\frac{\pi}{2};\phi+\frac{\pi}{2}\right[}  ~\left|f\left(\frac{a}{\cos(\phi-\theta)}e^{\ic \theta}\right)\right|^p\frac{a}{\cos^2(\phi-\theta)}~\d\theta\right)^{\frac{1}{p}}.
\eeq 
The closure of the space of continuous $L^p$-integrable (i.e.\,\,with finite $L^p$-norm) functions, with respect to the norm $\| \cdot \|_p$,
is the Banach space $L^p(\ell_{\phi,a})$, see e.g.\,\,\cite[Ch.\,1,\S1.2]{Heb00}. These definitions naturally extend to functions with values in a finite dimensional normed $\C$-vector space $(V,|\cdot|)$.
\vskip1,5mm
For any  interval $I\subseteq \R$, we define
\beq
\H_{I,a}:=\bigcup_{\phi\in I}\H_{\phi,a}.
\eeq
In particular, if $I$ is the open interval
$\left]\phi_{\min};\phi_{\max}\right[$,
\beq
\H_{I,a}=\left\{z\in\Tilde{\C^*}\colon \arg z\in\left]\phi_{\min}-\frac{\pi}{2};\phi_{\max}+\frac{\pi}{2}\right[,\quad |z|\geq \thi_{I,a}(\arg z)\right\},
\eeq
where $\thi_{I,a}\colon \left]\phi_{\min}-\frac{\pi}{2};\phi_{\max}+\frac{\pi}{2}\right[\to \R$ is defined by
\begin{empheq}[left =\thi_{I,a}(\phi){=}\empheqlbrace]{align*}
&\frac{a}{\cos\left(\phi_{\min}-\phi\right)},&\phi_{\min}-\frac{\pi}{2}<\phi\leq \phi_{\min},\\
&a,&\phi_{\min}\leq\phi\leq\phi_{\max},\\
&\frac{a}{\cos\left(\phi-\phi_{\max}\right)},&\phi_{\max}\leq\phi<\phi_{\max}+\frac{\pi}{2}.
\end{empheq}
\begin{figure}
\centerline{\includegraphics[width=0.5\textwidth]{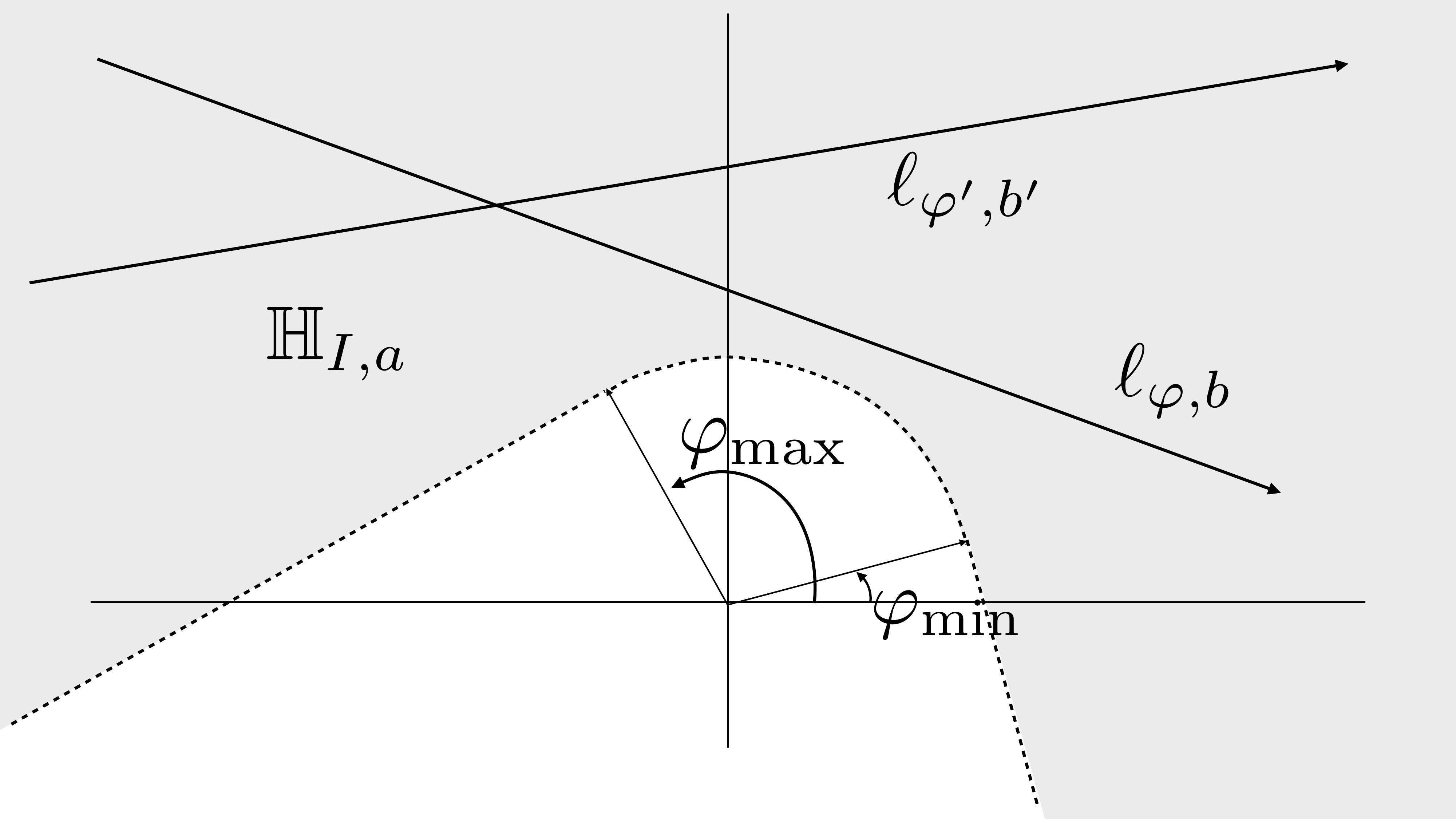}}
\caption{The domain $\H_{I,a}$. The angles are represented mod $2\pi$. Equivalently, the figure is the projection  onto $\C$.  Two lines $\ell_{\phi,b}$ and $\ell_{\phi^\prime,b^\prime}$ are also represented, with $\phi_{\mathrm{min}}<\phi<\phi^\prime<\phi_{\mathrm{max}}$
}
\label{26giugno2023-2}
\end{figure}
See Figure \ref{26giugno2023-2}. In the paper, $\hol$ will denote the sheaf of holomorphic functions on $\Tilde{\C^*}$. If $\Om\subseteq\Tilde{\C^*}$ is an open subset, and $V$ is a finite dimensional $\C$-vector space, we set
$\hol(\Om;V):=\hol(\Om)\ox_\C V$.  If $C$ is not open, we define $\hol(C;V)$ as inductive limit ${\varinjlim}\hol(\Om;V)$, with $\Om$ open sets containing $C$, see e.g.\,\,\cite[Cor.\,1,Th.\,3.3.1,pg.151]{God98}.

\subsection{The main theorem} \label{secmth}

Let $n\geq 1$. Denote by $\h(n,\C)$ the Lie algebra of diagonal matrices in $\gl(n,\C)$. We equip $\gl(n,\C)$ with an arbitrary norm $|\cdot|$. Let $I\subseteq\R$ an open interval, $J\subseteq I$ a closed interval, 
 and $a\in\R_{>0}$. For  $\La\in\hol(\H_{I,a};\h(n,\C))$ and $i,j\in\{1,\dots,n\}$, we set 
\beq
\label{14settembre2023-1}
\mathfrak{S}^{ij}_{J,a}:=\underset{\phi\in J}{\sup}\,\,\underset{b\geq a}{\sup}\,\,\underset{\substack{w\leq z\\(w,z)\in\ell_{\phi,b}^{\times 2}}}{\sup}\on{Re}\int_{w,\ell_{\phi,b}}^z\left(\La_{ii}(t)-\La_{jj}(t)\right){\rm d}t,
\eeq
\beq
\label{14settembre2023-2}
\mathfrak{I}^{ij}_{J,a}:=\underset{\phi\in J}{\inf}\,\,\underset{b\geq a}{\inf}\,\,\underset{\substack{w\leq z\\(w,z)\in\ell_{\phi,b}^{\times 2}}}{\inf}\on{Re}\int_{w,\ell_{\phi,b}}^z\left(\La_{ii}(t)-\La_{jj}(t)\right){\rm d}t.
\eeq

\vskip2mm
\begin{defn}
\label{9settembre2023-1}
A holomorphic function $\La\in\hol(\H_{I,a};\h(n,\C))$ satisfies the {\it $L$-condition on $I$} if, for any $i,j\in\{1,\dots,n\}$, and any $J\subseteq I$ closed interval, there exist $K_1,K_2\in\R$, depending on $J$,  such that one of the following mutually exclusive conditions holds: 
\begin{itemize}
\item either
\beq\tag{$L$-1}\label{l'1}
\mathfrak{S}^{ij}_{J,a}\leq K_1
\quad \text{ and }\quad 
\mathfrak{I}^{ij}_{J,a}=-\infty,
\eeq
\item or \beq\tag{$L$-2}\label{l'2}
\mathfrak{I}^{ij}_{J,a}\geq K_2.
\eeq
\end{itemize}
\end{defn}

\begin{defn}
\label{5giugno2023-1}
A holomorphic function $R\in\hol(\H_{I,a};\gl(n,\C))$ 
satisfies the {\it good decay condition} {\it on $I$} if for any closed interval $J\subseteq I$ and any  $b\geq a$, $R$ is $L^1$-integrable if restricted along any line $\ell_{\phi,b}$, with $\phi\in J$, and
\beq
\label{gdc}
\lim_{b\to+\infty}\underset{\phi\in J}{\sup}{\pnorm{1}{R|_{\ell_{\phi,b}}}}=0,
\eeq 
where $\|\cdot\|_1$ is the norm in $L^1(\ell_{\phi,b},dz)$. 
\end{defn}

\begin{example}
\label{13aprile2023-5}
The function $R\in\hol(\H_{I,a};\gl(n,\C))$ satisfies the good decay condition if there exists $\dl>0$ such that $R=O(z^{-1-\dl})$ on $\H_{J,a}$ for any closed interval $J\subseteq I$, that is
 \beq
 \underset{z\in \H_{J,a}}{\sup}|z^{1+\dl}R(z)|<\infty.
 \eeq

\end{example}
\begin{rem}
If $R\in\hol(\H_{I,a};\gl(n,\C))$ satisfies the good decay condition, then we have 
\beq
\label{9settembre2023-2}
M_{R,a,J}:=\underset{b\geq a}{\sup} \,\,\underset{\phi\in J}{\sup}\pnorm{1}{R|_{\ell_{b,\phi}}}<\infty.
\eeq
By \eqref{gdc}, $M_{R,a,J}$ can be made smaller by increasing $a$, and $M_{R,a,J}\to 0$ for $a\to +\infty$.

\end{rem}
\begin{rem}
Both the good decay condition and condition \eqref{9settembre2023-2} can be easily adapted to $\C$-valued functions defined on a single half-plane $\H_{\phi,a}$. The latter condition defines the so-called {\it Smirnov classes} $E^1(\H_{\phi,a})$. These functional spaces, and the more general spaces $E^p(\H_{\phi,a})$, are extensively studied in \cite[Ch.\,11]{Dur70}. The good decay condition, then, can be seen as defining a subclass of $E^1(\H_{\phi,a})$
\end{rem}

\begin{thm}\label{mainth}
Consider the differential equation
\beq
\label{diffeq}
     \frac{dY}{dz}
           =
       \left(\La(z)+R(z)\right)Y,
\eeq
 where the leading term $\La\in\mathscr O(\H_{I,a};\h(n,\C))$ satisfies the $L$-condition on $I$, and the perturbation term $R\in\mathscr O(\H_{I,a};\gl(n,\C))$ satisfies the good decay condition on $I$. There exists a unique holomorphic fundamental matrix solution $Y\colon \H_{I,a}\to GL(n,\C)$ with asymptotic behaviour,  for any closed interval  $J\subset I$ and   $z\to \infty$ in $\H_{J,a}$, given by 
$$
Y(z)=\left(I_n+o(1)\right)\exp\left(\int_{z_o}^z\La(w){\rm d}w\right),
$$
where $z_o$ is an arbitrary point of $\H_{I,a}$.
\end{thm}

\begin{cor}\label{maincor}
In the same assumptions of Theorem \ref{mainth}, if $R=O(z^{-1-\dl})$ on $\H_{I,a}$, then the unique solution $Y$  of Theorem \ref{mainth} is such that
$$
Y(z)=\left(I_n+O(z^{-\dl})\right)\exp\left(\int_{z_o}^z\La(w){\rm d}w\right),
$$
where $z_o$ is an arbitrary point of $\H_{I,a}$.
\end{cor}

The proof of Theorem \ref{mainth} will be given in the next sections. Before, we need to introduce some preliminary tools.

\subsection{The space $\mc H_{J,a}$, and the integral operators $K_\pm$} \label{sechja}

Let  $I=\left]\phi_{\min};\phi_{\max}\right[$ be an open interval. Given a domain $\H_{J,a}$,
with $J \subset I$ a closed interval, we define $\mc H_{J,a}$ to be the space of continuous bounded function
on $\H_{J,a}$ with values in $\C^n$, whose restriction to $\on{int} \H_{J,a}$ is analytic. The space $\mc H_{J,a}$ is Banach  if
equipped with the  norm \,\,\cite[Th.\,1.2]{IlYa08}
\[
\snorm{f}:=\underset{z\in\H_{J,a}}{\sup}{|f(z)|}.
\]

Let $\La\in\hol(\H_{I,a};\h(n,\C))$ be a diagonal matrix  satisfying the $L$-condition. For any fixed $j_o\in\{1,\dots,n\}$ and an arbitrary $i\in\{1,\dots,n\}$, only one between conditions \eqref{l'1} and \eqref{l'2} hold for the pair $(i,j_o)$. Define a partition $J_+(j_o)\cup J_-(j_o)=\{1,\dots, n\}$ as follows:
\begin{itemize}
\item if \eqref{l'1} holds for $(i,j_o)$, then $i\in J_+(j_o)$,
\item if \eqref{l'2} holds for $(i,j_o)$, then $i\in J_-(j_o)$.
\end{itemize}
\vskip1,5mm
For any fixed $j_o\in\{1,\dots,n\}$, define the diagonal matrices $W_{\pm,j_o}\in\hol(\H_{I,a}^{\times 2};\h(n,\C))$ by
\beq
\label{15settembre2023-1}
\begin{aligned}
W_{+,j_o}(\zeta_1,\zeta_2)_{ii}&:= \exp\int_{\zeta_1}^{\zeta_2}(\La_{ii}(t)-\La_{j_oj_o}(t))\d t,&& i\in J_+(j_o),\\
W_{+,j_o}(\zeta_1,\zeta_2)_{ii}&:=0,&&i\in J_-(j_o),
\\
\\
W_{-,j_o}(\zeta_1,\zeta_2)_{ii}&:=0,&&i\in J_+(j_o),\\
W_{-,j_o}(\zeta_1,\zeta_2)_{ii}&:= \exp\int_{\zeta_1}^{\zeta_2}(\La_{ii}(t)-\La_{j_oj_o}(t))\d t,&& i\in J_-(j_o),
\end{aligned}
\eeq
where the integral can be taken along any curve in $\H_{I,a}$ joining $\zeta_1$ and $\zeta_2$.
\vskip2mm
Given a function $R\in\hol(\H_{I,a};\gl(n,\C))$ satisfying the good decay condition,  and $J \subset I$ a closed sub-interval, we  introduce two integral operators $K_{\pm,j_o}\colon \mc H_{J,a}\to \mc H_{J,a}$, defined by
\begin{align}
\label{k+}
K_{+,j_o}[f;z]&:=\int_{\infty}^zW_{+,j_o}(t,z) R(t) f(t) \d t,\\
\label{k-}
K_{-,j_o}[f;z]&:=\int_{z}^\infty W_{-,j_o}(t,z) R(t) f(t) \d t,
\end{align}where
\begin{itemize}
\item the integral \eqref{k+} is taken from $-\infty$ to $z$ along any line $\ell_{\phi,b}\subseteq\H_{J,a}$ passing through $z$, 
\item the integral \eqref{k-} is taken from $z$ to $+\infty$ along any line $\ell_{\phi,b}\subseteq\H_{J,a}$ passing through $z$.
\end{itemize}

\begin{lem}\label{lemmak} 
 For any closed interval   $J \subset I$ the following hold.
\begin{enumerate}
\item The operators $K_{\pm,j_o}$ are well-defined bounded operators on $\mc H_{J,a}$. In particular, there exists
a constant $C>0$, independent on $R$, such that 
\beq\label{normk}
\snorm{K_{\pm,j_o}[f;~\cdot~]}\leq C\,\snorm{f}\,M_{R,a,J},
\eeq
where $M_{R,a,J}$ is defined in \eqref{9settembre2023-2}. 
Moreover, for $z\to\infty$ in $\H_{J,a}$,  
  \beq
  \label{18settembre2023-1}
  \lim_{z\to\infty}K_{\pm, j_o}[f;z]=0.
  \eeq
   
\vskip 0.2 cm    
\item If $R=O(z^{-1-\dl})$ on $\H_{I,a}$, then there exists a $C>0$ such that for any $f\in \mc H_{J,a}$ we have
\beq
|K_{\pm,j_o}[f;z]|\leq C \snorm{f}|z^{-\dl}|.
\eeq
\end{enumerate}
\end{lem}
\proof
To prove point (1), we need to prove that 
\begin{enumerate}
\item[(i)] the integrals \eqref{k+}, \eqref{k-} are finite for any $z\in \H_{J,a}$;
\item[(ii)] they are independent of the lines $\ell_{\phi,b}$, with $\phi\in J$,  passing through $z$;
\item[(iii)] $K_{\pm,j_o}[f;z]$ are  analytic functions of $z\in \on{int} \H_{J,a}$ and continuous at  $z\in \der\H_{J,a}$.
\item[(iv)]  $K_{\pm,j_o}[f;z]$ vanishes  as $z\to\infty$  in $\H_{J,a}$. 
\end{enumerate}
Point (i) and \eqref{normk} follow from conditions $L$ and the good decay condition on $R$, since
$$
\left|\int_{\infty,\ell_{\phi,b}}^zW_{+,j_o}(t,z) R(t) f(t) \d t\right|\leq \left(\underset{\substack{t\leq z\\(t,z)\in\ell_{\phi,b}^{\times 2}}}{\sup} |W_{+,j_o}(t,z)| \right) \pnorm{1}{R|_{\ell_{b,\phi}}} \snorm{f},
$$
$$
\left|\int_{z,\ell_{\phi,b}}^\infty W_{-,j_o}(t,z) R(t) f(t) \d t\right|\leq \left(\underset{\substack{z\leq t\\(t,z)\in\ell_{\phi,b}^{\times 2}}}{\sup} |W_{-,j_o}(t,z)| \right) \pnorm{1}{R|_{\ell_{b,\phi}}} \snorm{f}.
$$

As for point (ii),  let $\ell_{\phi_i,b_i}\subseteq\H_{J,a}$, with $i=1,2$, be two lines passing through $z$.  Since $J \subset  I$, there exists   $\phi^\prime\in I$ and a line $\ell_{\phi',b'}$ transversal to both $\ell_{\phi_i,b_i}$, $i=1,2$,  intersecting them respectively at the points $w_i\in\ell_{\phi_i,b_i}$, $i=1,2$, with $w_i < z$. We have
\begin{align*}
0=\left[\int_{w_1}^z-\int_{w_2}^z+\int_{w_2}^{w_1}\right]W_{+,j_o}(t,z) R(t) f(t) \d t.
\end{align*}
By taking the limit $b'\to+\infty$, the term $\int_{w_2}^{w_1}$ goes to zero, by \eqref{gdc}.  Hence, point (ii) is proved for $K_{+,j_o}[f;z]$. The case of $K_{-,j_o}[f;z]$ is similar.

To prove (iii), first let $z$ and $z_1$ be  poins of $\on{int} \H_{J,a}$.  Consider $K_{+,j_o}[f;z]$ and $ K_{+,j_o}[f;z_1]$, respectively along lines $\ell_{\phi,b}$ and $\ell_{\phi,b_1}$, with the same $\phi\in J$.  We claim that 
\begin{equation}
\label{13aprile2023-1}
K_{+,j_o}[f;z_1]-K_{+,j_o}[f;z]=\int_z^{z_1}W_{+,j_o}(t,z) R(t) f(t) \d t,
\end{equation}
where the integration in the r.h.s. is along any path joining $z$ to $z_1$. If the claim is true, then the following limit is well defined and analytic w.r.t.\,\,$z$:
$$ 
\frac{d K_{+,j_o}[f;z]}{dz}:=\lim_{z_1-z\to 0} \frac{K_{+,j_o}[f;z_1]-K_{+,j_o}[f;z]}{z_1-z}=W_{+,j_o}(z,z)R(z)f(z).
$$
 To prove the claim, as we did before, for suitable $\phi^\prime\in I$  we can take  an arbitrary $\ell_{\phi^\prime,b^\prime}$ transversal to $\ell_{\phi,b}$ and $\ell_{\phi,b_1}$, intersecting them in $w<z$ and $w_1<z_1$ respectively.  See figure \ref{13aprile2023-2}.  
Then 
\begin{align*}
0=\left[\int_{w,\ell_{\phi,b}}^z+\int_z^{z_1}- \int_{w_1,\ell_{\phi,b_1}}^{z_1}-\int_{w,\ell_{\phi^\prime,b^\prime}}^{w_1}\right]W_{+,j_o}(t,z) R(t) f(t) \d t.
\end{align*}
By taking the limit $b'\to+\infty$, the term $\int_{w}^{w_1}$ goes to zero, by \eqref{gdc}.  Successively, let $z\in \der \H_{J,a}$ and $z_1\in \on{int} \H_{J,a}$. The equality \eqref{13aprile2023-1} is proved in the same way, and implies continuity at $z$. 
The case of $K_{-,j_o}[f;z]$ is similar.

\begin{figure}
\centerline{\includegraphics[width=0.5\textwidth]{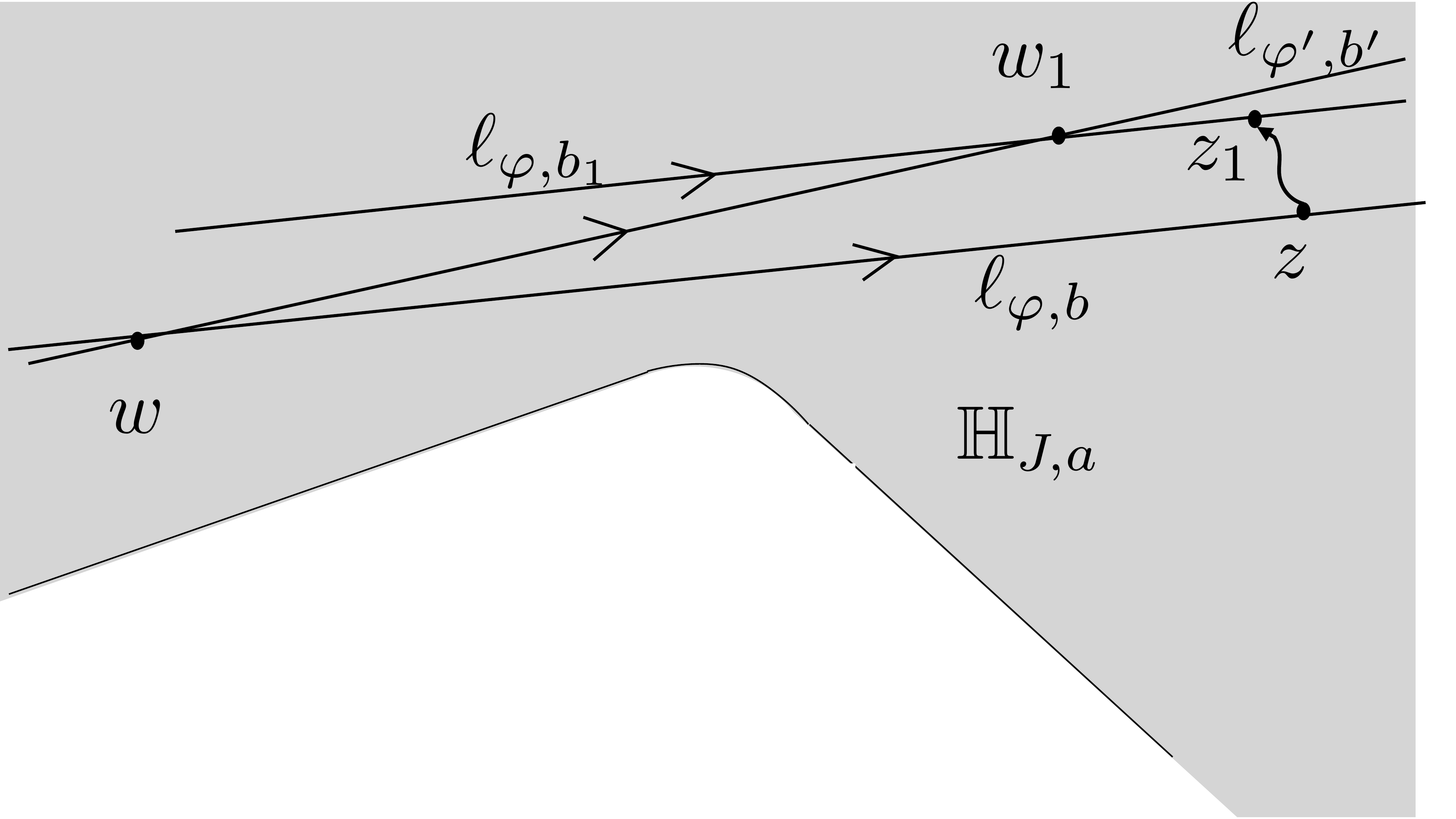}}
\caption{}
\label{13aprile2023-2}
\end{figure}

\begin{figure}
\centerline{\includegraphics[width=0.5\textwidth]{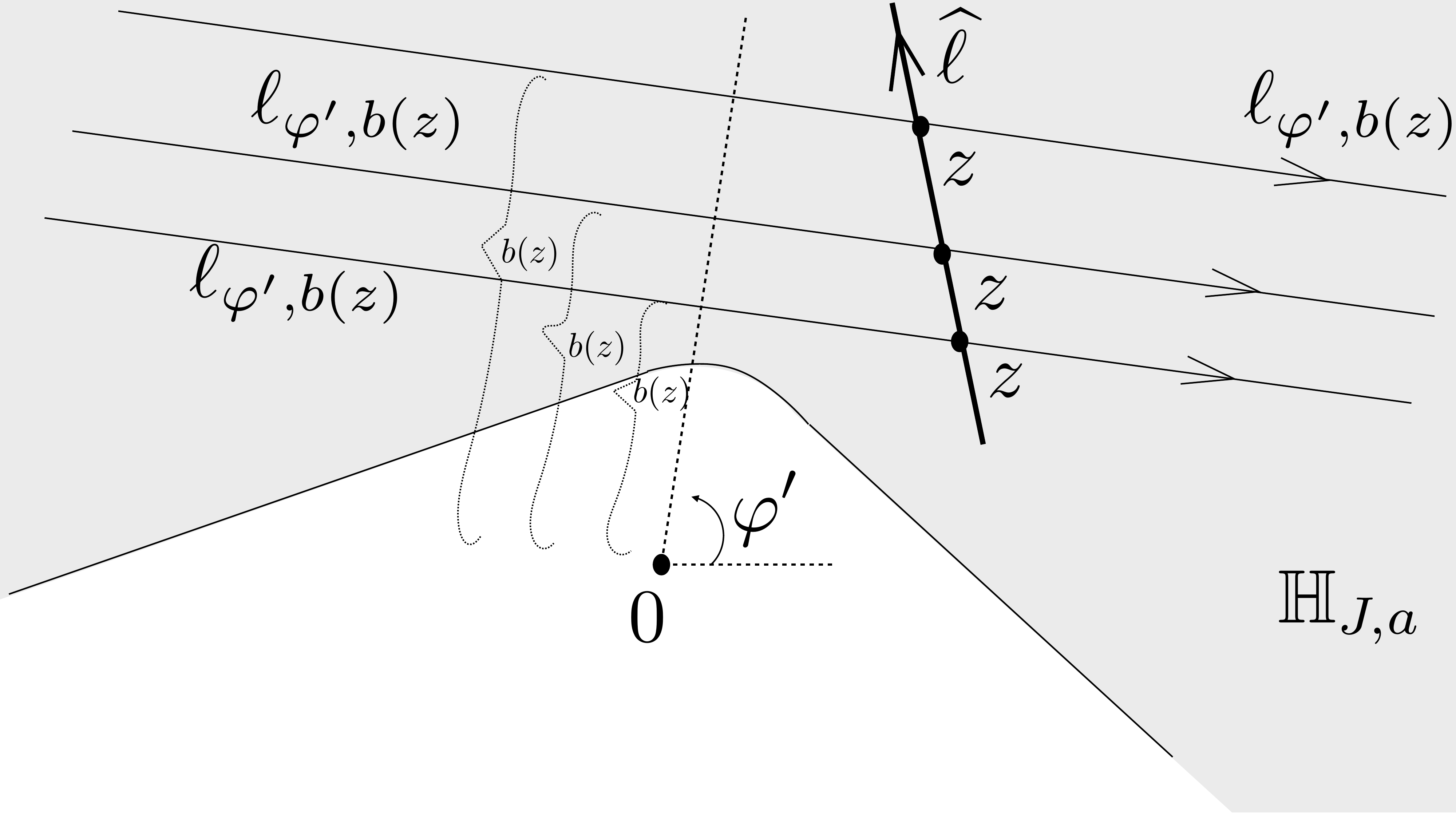}}
\caption{$z$ goes to infinity along $\widehat{\ell}$. }
\label{13aprile2023-3}
\end{figure}

\vskip1,5mm

As for point (iv), we show that $K_{+,j_o}[f;z]$ vanishes  for $z\to \infty$ along any line $\widehat{\ell}\subset \H_{J,a}$.
 By (ii),  \eqref{k+} depends only on $z$ and not on $\phi$, so we can take a line of integration $\ell_{\phi^\prime,b(z)}$ transverse to $\widehat{\ell}$, with $\phi^\prime\in I$ and $b(z)=|z|\cos(\phi^\prime-\arg z)$.   See figure \ref{13aprile2023-3}. 
By \eqref{normk},
\beq\label{decayk}
|K_{\pm,j_o}[f;z]|\leq C\,\snorm{f}\, \pnorm{1}{R|_{\ell_{\phi^\prime,b(z)}}},\quad C>0.
\eeq
Now, $b(z)\to +\infty$ when  $|z|\underset{\widehat{\ell}}\to\infty$, and   by \eqref{gdc}, the right-hand-side goes to zero. 
  The case of $K_{-,j_o}[f;z]$ is similar.

\vskip1,5mm
 Point (2) follows from inequality \eqref{decayk}.
\endproof

\subsection{Proof of Theorem \ref{mainth}} \label{secprmth}

Fix $j_o\in\{1,\dots,n\}$, and $z_o\in\H_{I,a}$. Consider the change of variables 
\beq
Y(z)=Z(z)\exp\int_{z_o}^z\La_{j_oj_o}(t)\d t,
\eeq
in equation \eqref{diffeq}. The function $Z(z)$ satisfies
\beq\label{modifieddiffeq}
\frac{dZ}{dz}=\left(\Tilde\La(z)+R(z)\right)Z(z),\quad \Tilde\La(z)=\La(z)-\La_{j_oj_o}(z)I.
\eeq
Denote by $(e_1,\dots, e_n)$ the standard basis of $\C^n$, where $e_i=(0,\dots,0,1_i,0,\dots,0)^{\rm T}$, for any $i=1,\dots,n$.
Theorem \ref{mainth} will be proved if we show that there exists a unique $\C^n$-valued solution $Z(z)$ of \eqref{modifieddiffeq} such that, when $z\to \infty$ in $ \H_{J,a}$, 
\beq\label{modifiedasym}
Z(z)=e_{j_o}+o(1)\quad \text{in }\H_{J,a},
\eeq for any closed interval $J\subseteq I$.
\vskip2mm
Consider the integral equation
\beq\label{integeq}
f(z)= e_{j_o}+K_{+,j_o}[f;z]-K_{-,j_o}[f;z],\quad f\in\mc H_{J,a}.
\eeq

\begin{lem}
\label{lemdiffeqintegeq}
The following conditions are equivalent:
\begin{enumerate}
\item a function $Z(z)$ is a solution of the integral equation \eqref{integeq};
\item a function $Z(z)$ is a solution of \eqref{modifieddiffeq} satisfying \eqref{modifiedasym}.
\end{enumerate}

\end{lem}
\proof
Let us first re-write the integral equation \eqref{integeq} in a more convenient form. Set
\beq
W_1(z):=W_{+,j_o}(z_o,z),\quad W_2(z):=W_{-,j_o}(z_o,z),\quad W(z)=W_1(z)+W_2(z)=\exp\int_{z_o}^z\Tilde\La(t)\d t,
\eeq
so that\beq
W_{+,j_o}(t,z)=W_1(z)W(t)^{-1},\quad W_{-,j_o}(t,z)=W_2(z)W(t)^{-1}.
\eeq
The integral equation \eqref{integeq} is thus 
\beq\label{integeq2}
f(z)=e_{j_o}+W_1(z)\int_{\infty}^z W(t)^{-1}R(t)f(t)\d t-W_2(z)\int_{z}^\infty W(t)^{-1}R(t)f(t)\d t.
\eeq
Notice that both the constant vector $e_{j_o}$ and the matrices $W_1(z),W_2(z),W(z)$ are solutions of the diagonal system 
\beq\label{diagdiffeq}
\frac{dT}{dz}(z)=\Tilde\La(z)T(z).
\eeq 
If $Z(z)$ is a solution of the integral equation \eqref{integeq}, 
then we can derive both sides of equation \eqref{integeq2}, with $f=Z$. We obtain
\begin{multline*}
Z'(z)=W'_1(z)\int_{\infty}^zW(t)^{-1}R(t)Z(t)\d t+W_1(z)W(z)^{-1}R(z)Z(z)\\-W'_2(z)\int_{z}^\infty W(t)^{-1}R(t)Z(t)\d t+W_2(z)W(z)^{-1}R(z)Z(z)\\
=\left(\Tilde\La(z)+R(z)\right)Z(z).
\end{multline*}
So, $Z(z)$ is a solution of \eqref{modifieddiffeq}, and \eqref{modifiedasym} holds by \eqref{18settembre2023-1} of Lemma \ref{lemmak}. 

Conversely, 
if $Z(z)$ is a solution of \eqref{modifieddiffeq} satisfying \eqref{modifiedasym},
then the function
\[
Z(z)-W_1(z)\int_{\infty}^z W(t)^{-1}R(t)Z(t)\d t+W_2(z)\int_{z}^\infty W(t)^{-1}R(t)Z(t)\d t
\]is a solution of \eqref{diagdiffeq}, as follows by differentiation. By condition \eqref{modifiedasym}, it equals \begin{equation}
\label{16novembre2024-1}
e_{j_o}+o(1),
\end{equation}
 where 
$$o(1)=\sum_{i\neq j_o} c_i e_i \exp\Bigl\{\int_{z_o}^z  (\Lambda_{ii}(t)-\Lambda_{j_oj_o}(t))dt \Bigr\}, \quad \quad\hbox{ for some } c_i\in\mathbb{C}.
$$
Now,  the {\it L-condition on $I$ }implies that for every $i\neq j_o$  there is a direction in $\mathbb{H}_{J,a}$ such that 
$\exp\Bigl\{\int_{z_o}^z  (\Lambda_{ii}(t)-\Lambda_{j_oj_o}(t))dt \Bigr\}\to \infty$. Hence,  necessarily  $o(1)\equiv 0$. 
So $Z(z)$ it is a solution of \eqref{integeq}. 
\endproof

\begin{rem}
\label{30settembre2024-1}
The ``conversely'' statement in the above proof does not hold under the original conditions given by Levinson. They are similar to our $L$-condition, but hold only along a half-line (or along rays within a small sector), preventing the conclusion that $o(1)=0$ in \eqref{16novembre2024-1}. 
As a result, Levinson’s theorem does not guarantee the uniqueness of the solution to the differential equation with the prescribed asymptotics; it establishes only the existence of such a solution.
\end{rem}

Equation \eqref{integeq} is  $ ({\rm Id}-K_{j_o})[f]=e_{j_o}$, 
    with the operator
$$ 
 K_{j_o}:~\mathcal{H}_{J,a}\longrightarrow \mathcal{H}_{J,a},\quad \quad  K_{j_o}:=K_{+,j_o}-K_{-,j_o}.
 $$ 

For any $J\subset I$,  by equations \eqref{gdc} and \eqref{normk}, and by replacing $a$ with $a'(J)>a$ if necessary, the operator $K_{j_o}$ has norm less that one: 
 $$ 
\|K_{j_o}\|:= \underset{f\in\mathcal{H}_{J,a} }{\sup}\,\,\frac{\|K_{j_o}[f]\|_\infty}{\|f\|_\infty}\leq 2 C\,M_{R,a,J}<1,
 $$
  Hence,  \eqref{integeq}   has a unique solution, given by
 the Neumann series  \cite[Chapter VI]{RS80}
 $$ 
 Z=  ({\rm Id}-K_{j_o})^{-1}[e_{j_o}]=\sum_{m=0}^\infty K_{j_o}^m[e_{j_o}]
 ~ \in \mathcal{H}_{J,a}.
 $$
 Equivalently, $K_{j_o}$ is a contraction for any closed interval $J\subset I$. Then, by Banach--Caccioppoli fixed point theorem, the integral equation \eqref{integeq} admits a unique vector solution in $\mc H_{J,a}$. 
 
 Finally, \eqref{18settembre2023-1} implies the asymptotic behaviour \eqref{modifiedasym}. 
By Lemma \ref{lemdiffeqintegeq}, it is a solution of the differential equation \eqref{modifieddiffeq} satisfying \eqref{modifiedasym} for any $\H_{J,a}$. Being a solution of a linear differential equation, it analytically extends to $\H_{I,a}$. Theorem \ref{mainth} is proved.
\vskip1,5mm
Corollary \ref{maincor} follows from the same argument and point (2) of Lemma \ref{lemmak}.

\subsection{The case of subdominant solutions}
\label{20settembre2023-2}

In this subsection, we consider a particular vector solution to the
differential equation (\ref{diffeq})
corresponding to an index $j_o$ such that {\it the condition \eqref{l'2}
holds on $I$ for all $(i,j_o)$}.
Such a solution, that we denote by $y_{j_o}$, is called {\it
subdominant} since Theorem \ref{mainth}
and the condition \eqref{l'2} imply that
\begin{equation}
 \big|y_{j_o}(z)\big|=O\big( |y(z)| \big) \mbox{ as $z \to +\infty$
along $\ell_{\varphi,b}$ for every $\varphi \in I$
 and $b\geq a$},
\end{equation}
if $y$ is an arbitrary non-zero vector solution to the same differential equation.
\vskip1,5mm
We show that, in the case of a subdominant solution, the result of
Theorem \ref{mainth} can be improved. The validity of the asymptotic
behaviour of the subdominant solutions can be extended, under certain
conditions, to a larger domain than the one claimed by Theorem
\ref{mainth}, namely a domain twice as ample.

In fact, for the index $j_o$, the {`natural'} domain on which to prove
existence and uniqueness of solutions of the integral equation
\eqref{integeq} is the `cut-plane'
$$\H_{[\phi],a}:=\H_{[\varphi-\pi,\varphi],a},$$
 instead of
the half-plane $\H_{\varphi,a}$.
The domain  $\H_{[\phi],a}$  is foliated by (parametric) straight
lines and half-lines whose
tangent at any point is the unit vector of argument
$\varphi-\dfrac{\pi}{2}$, see Figure \ref{19settembre2023-2}.
We denote these curves by
$\iota^{\varphi}_{a,b}$ with $b \in \R$. Their support is given explicitly by
\begin{equation}
\label{22giugno2023-1}
 \iota^{\varphi}_{a,b}=
\left\{ \begin{aligned}
 & \lbrace z \in \H_{[\phi],a}, | z|= \frac{b}{\cos (\varphi- \arg
z)} , \varphi-\frac\pi2< \arg z< \varphi+\frac\pi2\rbrace , &\mbox{ if
} b\geq a
 \\
 & \lbrace z \in \H_{[\phi],a}, | z|= \frac{b}{\cos (\varphi- \arg z)},
  \varphi-\frac\pi2 <  \arg z\leq \varphi-\frac\pi2 +\sin^{-1}(b/a)
\rbrace,  &\mbox{ if } 0<b<a,
 \\
 & \lbrace z \in \H_{[\phi],a}, | z|\geq a, \arg z=
\varphi-\frac\pi2 \rbrace,  &\mbox{ if } b=0,
 \\
 & \lbrace  z \in \H_{[\phi],a}, | z|= \frac{b}{\cos (\varphi- \arg z)} ,
 \varphi-\frac\pi2+ \sin^{-1}(b/a)\leq \arg z<
\varphi-\frac\pi2\rbrace,  &\mbox{ if } -a<b<0,
 \\
 &  \lbrace  z \in \H_{[\phi],a}, | z|= \frac{b}{\cos (\varphi- \arg z)} ,
 \varphi-\frac{3\pi}{2}< \arg z< \varphi-\frac\pi2\rbrace, &\mbox{ if
} b\leq -a.
\end{aligned}
\right.
\end{equation}
   Before we state the theorem on subdominant solutions, we extend the
condition \eqref{l'2}
and good decay condition to the curves $\iota^{\varphi}_{a,b}$.

\begin{figure}
\centerline{\includegraphics[width=0.43\textwidth]{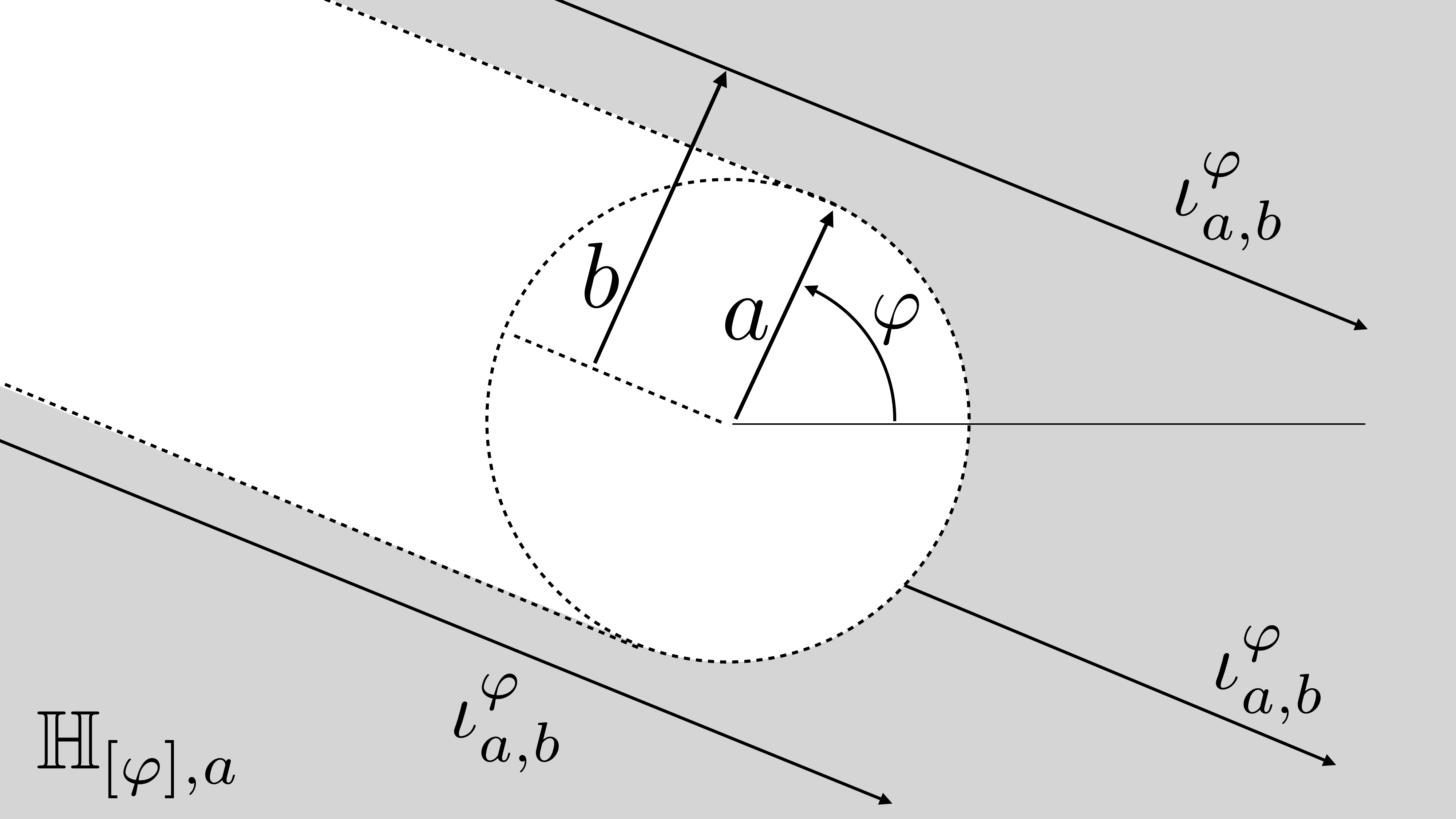}}
\caption{The domain $\H_{[\phi],a}$. The angle $\phi$ is represented modulo $2\pi$. 
}
\label{19settembre2023-2}
\end{figure}

\begin{defn}
\label{21giugno2023-1}
 Let $I=]\phi_\mathrm{min},\phi_\mathrm{max}[$ be an open interval,
and define $\widetilde{I}:= ]\phi_\mathrm{min}-\pi,\phi_\mathrm{max}[$
and $\H_{\widetilde{I},a}$ (see Figure \ref{26giugno2023-3}).
 The index $j_o \in \{1,\dots,n\}$ is said to be a {\it subdominant
index} {\it on $I$} for the holomorphic function
 $\La\in\hol(\H_{\widetilde{I},a};\h(n,\C))$
if for any  $J\subseteq I$ closed
interval, there exist a $K_2\in\R$ such that
\begin{equation}
\label{conditionS}
 \underset{\phi\in J}{\inf}\,\,\underset{b \in
\R}{\inf}\,\,\underset{\substack{w\leq z\\(w,z)
 \in(\iota^{\varphi}_{a,b})^{\times
2}}}{\inf}\on{Re}\int_{w,\iota^{\varphi}_{a,b}}^z\left(\La_{ii}(t)-\La_{j_oj_o}(t)\right){\rm
d}t\geq K_2,
  \qquad \forall i \in \lbrace 1 \dots n \rbrace.
\end{equation}
\end{defn}

\begin{figure}
\centerline{\includegraphics[width=0.5\textwidth]{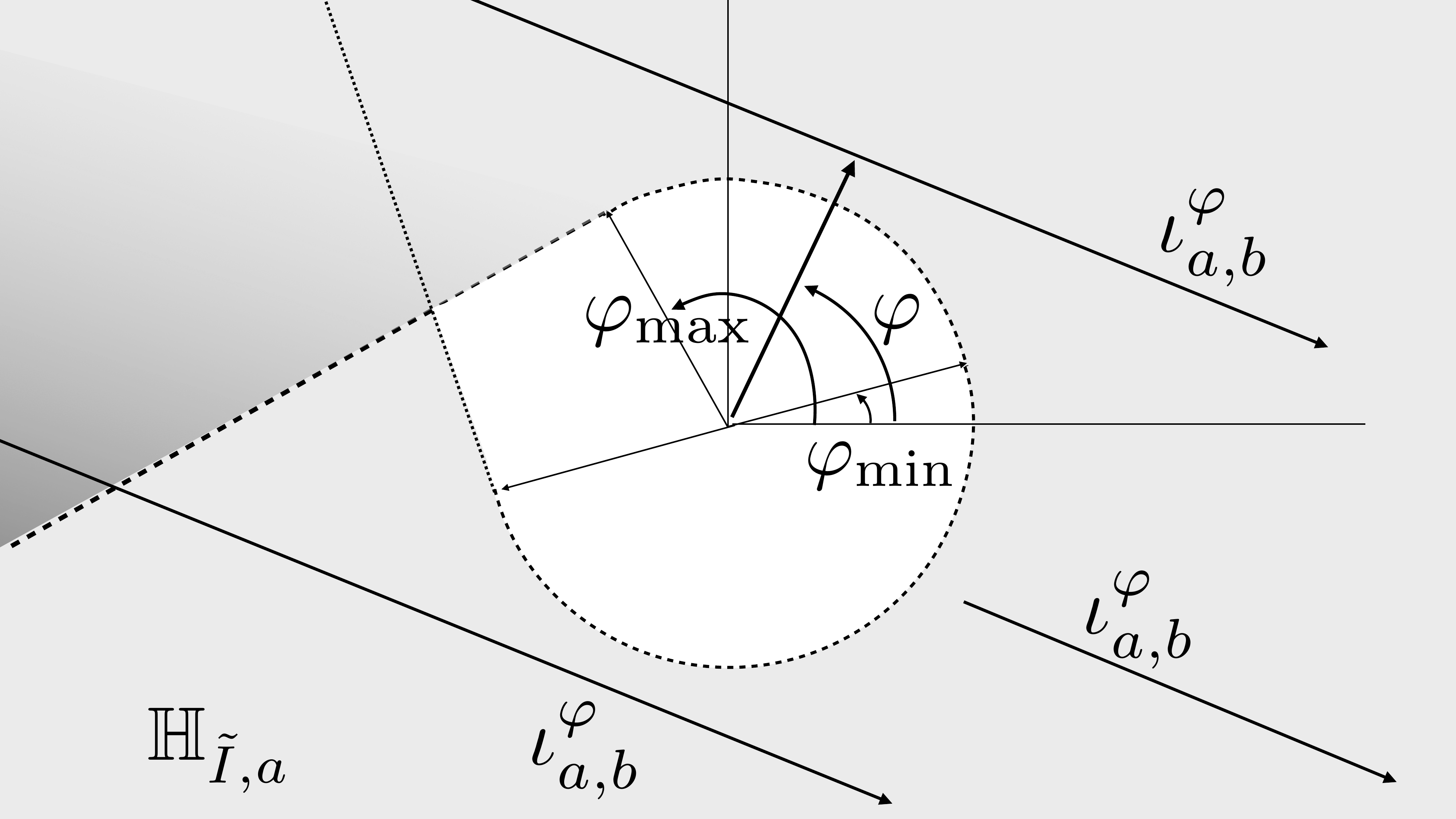}}
\caption{The domain $\H_{\widetilde{I},a}$. The angles are represented
mod $2\pi$. Equivalently, the figure is the projection  onto $\C$.
Three lines $\iota^{\varphi}_{a,b}$are also represented, with
$\phi_{\mathrm{min}}<\phi<\phi_{\mathrm{max}}$
}
\label{26giugno2023-3}
\end{figure}

\begin{defn}
\label{21giugno2023-1-bis}
A holomorphic function $R\in\hol(\H_{\widetilde{I},a};\gl(n,\C))$
satisfies the {\it S-decay condition} if for any closed interval $J\subseteq I$
\begin{equation}\label{eq:MRaJS}
 \widetilde{M}_{R,a,J}:= \underset{\phi\in J}{\sup} \; \underset{b \in
\R}{\sup}{\pnorm{1}{R|_{\iota^{\varphi}_{a,b}}}} < +\infty,
\end{equation}
and
\begin{equation}\label{eq:MRaJSinfty}
 \lim_{a\to +\infty} \widetilde{M}_{R,a,J}= 0.
\end{equation}
\end{defn}

\begin{thm}
\label{subth}
Consider the differential equation
\beq
\frac{dY}{dz}=\left(\La(z)+R(z)\right)Y(z),
\eeq with $\La\in\mathscr O(\H_{\widetilde{I},a};\h(n,\C))$  and
$R\in\mathscr O(\H_{\widetilde{I},a};\gl(n,\C))$.
Assume that $j_o$ is a subdominant index on $I$ for $\La$ and that $R$
satisfies the S-decay condition on $I$.
There exists a unique vector solution $y_{j_o}\in\mathscr
O(\H_{\widetilde{I},a};\C^n)$ with asymptotic behaviour,
\beq\label{subasym}
y_{j_o}(z)=\left(e_{j_o}+o(1)\right)\exp\left(\int_{z_o}^z\La_{j_oj_o}(w){\rm
d}w\right),
\eeq
for any closed interval $\widetilde{J}\subset \widetilde{I}$ and
$z\to \infty$ in $\H_{\widetilde{J},a}$, 
where $z_o$ is an arbitrary point of $\H_{\widetilde{I},a}$.
\end{thm}

\begin{cor}
\label{subcor}
In the same assumptions of Theorem \ref{subth}, if $R=O(z^{-1-\dl})$ on
$\H_{\widetilde{I},a}$, then the unique solution $y_{j_o}$  of Theorem
\ref{subth} has behaviour
 \beq
y_{j_o}(z)=\left(e_{j_o}+O(z^{-\dl})\right)\exp\left(\int_{z_o}^z\La_{j_oj_o}(w){\rm
d}w\right),
\eeq
for $z\to\infty$ in $\H_{\widetilde{J},a}$ and any closed
$\widetilde{J}\subset\widetilde{I}$.
\end{cor}
The proof of Theorem \ref{subth} closely follows the proof of Theorem
\ref{mainth}.
In this case the integral operator is of Volterra type \cite{Pog66,Cop88}.
This is the operator
$
\widetilde{K}_{j_o}\colon \mc H_{\widetilde{J},a}
       \to
           \mc H_{\widetilde{J},a}
              $,
 defined by
\begin{align}
\label{k-s}
\widetilde{K}_{j_o}[f;z]&:=\int_{z}^\infty W_{-,j_o}(t,z) R(t) f(t) \d t,
\end{align}
where
the integral \eqref{k-s} is taken from  $z$ to $+\infty$ along
any line $\iota^{\varphi}_{a,b}\subseteq\H_{\widetilde{J},a}$ passing
through $z$.

Repeating verbatim the proof of Lemma \ref{lemmak} points (1,2), we
deduce the following
analogous lemma.
\endproof
\begin{lem}\label{lemmaks}
 For any closed interval   $J \subset I$ the following hold.
\begin{enumerate}
\item The operator $\widetilde{K}_{j_o}$ is a well-defined bounded
operator on $\mc H_{\widetilde{J},a}$.
There exists a constant $C$, independent on $R$, such that
\begin{equation}
 \|\widetilde{K}_{j_o}[f;\cdot] \| \leq C \|f\|_{\infty} \widetilde{M}_{R,a,J}.
\end{equation}
Moeover,
  $$\lim_{z\to\infty}\widetilde{K}_{j_o}[f;z]=0\quad \hbox{in
$\H_{\widetilde{J},a}$.}$$

\vskip 0.2cm
\item If $R=O(z^{-1-\dl})$ on $\H_{\widetilde{J},a}$, then there
exists  $C>0$ such that
for any $f\in \mc H_{\widetilde{J},a}$ we have
\beq
|\widetilde{K}_{j_o}[f;z]|\leq C\snorm{f}|z^{-\dl}|.
\eeq
\end{enumerate}
\end{lem}

\subsubsection{Proof of of Theorem \ref{subth} and of Corollary \ref{subcor}}
Reasoning as in the proof of Theorem \ref{mainth}, we deduce that
$y_{j_o}$ is a solution of the differential
equation (\ref{diffeq}) satisfying the asymptotics (\ref{subasym}) if and only if
(for each $J \subset I$ closed, the restriction to
$\H_{\widetilde{J},a}$ of) the vector-valued function
$Z:= \exp\left(-\int_{z_o}^z\La_{j_oj_o}(w){\rm d}w\right)Y_{j_o}$
satisfies the integral
equation
\begin{equation}\label{inteqs}
 Z= e_{j_o}+ \widetilde{K}_{j_o}[Z].
\end{equation}
Due to Lemma \ref{lemmaks}(1) and formula (\ref{eq:MRaJSinfty}),
$\widetilde{K}_{j_o}$ is of norm less than one if $a$ is large enough.
Therefore, the above integral equation admits the unique solution $Z=({\rm Id}-\widetilde{K}_{j_o})^{-1}[e_{j_0}]   \in \H_{\widetilde{J},a}$, given by a Neumann series \cite{RS80}. 
This proves Theorem \ref{subth}.

If $R=O(z^{-1-\dl})$ on $\H_{\widetilde{I},a}$, then, due to Lemma
\ref{lemmaks}(2),
$Z_{j_o}(z)-e_{j_o}=O(z^{-\dl})$. This proves Corollary \ref{subcor}. $\Box$

\subsection{The main theorem with parameters} 
\label{20settembre2023-1}

Let $\Omega$ be a connected open set in $\mathbb{C}^m$, with $m\in \mathbb{N}_{\geq 1}$. The analog of Theorem \ref{mainth} holds when the matrices  $\Lambda$ and $R$  in \eqref{diffeq} holomorphically depend  on a parameter $\omega\in \Omega$ and satisfy the conditions below. 

\begin{defn}
\label{9settembre2023-6}
A holomorphic function $\La\in\hol(\H_{I,a}\times \Omega;~\h(n,\C))$ satisfies the {\it uniform integrability condition} if there exist positive functions $g_i:\H_{I,a}\rightarrow \mathbb{R}_{\geq 0}$, for $i=1,\dots,n$, which are $L^1$-integrable along  the segments\footnote{\, For any $\omega\in\Omega$,  a $z$-holomorphic  $\Lambda_{ii}(z,\omega)$ is $L^1$-integrable along any curve joining $\zeta_1$ and $\zeta_2$, and the result does not depend on the curve.} joining any two points $\zeta_1$ and $\zeta_2\in\H_{I,a}$, and 
\beq
\label{9settembre2023-3}
|\Lambda_{ii}(z,\omega)|\leq g_i(z),\quad \hbox{ for any $\omega\in\Omega$}.
\eeq
\end{defn}

The integrals \eqref{14settembre2023-1}-\eqref{14settembre2023-2} are defined in the same way in terms of $\Lambda(z,\omega)$,  and depend on $\omega$. 
\begin{defn}
\label{9settembre2023-1-bis}
A holomorphic function $\La\in\hol(\H_{I,a}\times \Omega;~\h(n,\C))$ satisfies the {\it  uniform  $L$-condition} {\it on $I$} if for any $i,j\in\{1,\dots,n\}$, and any closed interval $J\subseteq I$, there exist $K_1,K_2\in\R$, depending on $J$,  such that one of the following  mutually exclusive conditions holds:
\begin{itemize}
\item either
\beq\tag{$L$-1}
\underset{\omega\in \Omega}{\sup}\,\,\mathfrak{S}^{ij}_{J,a}\leq K_1
\quad \text{ and }\quad 
\underset{\omega\in \Omega}{\sup}\,\,\mathfrak{I}^{ij}_{J,a}=-\infty,
\eeq
\item or \beq\tag{$L$-2}
\underset{\omega\in \Omega}{\inf}\,\,\mathfrak{I}^{ij}_{J,a}\geq K_2.
\eeq
\end{itemize}
\end{defn}

 \begin{defn}
\label{5giugno2023-1-bis}
A holomorphic function $R\in\hol(\H_{I,a}\times \Omega;~\gl(n,\C))$  
satisfies the  {\it  uniform  good decay condition} {\it on $I$} if there is a positive function $g$, which is  $L^1$-integrable  for any closed interval $J\subseteq I$ along any line $\ell_{\phi,b}$, for any $b\geq a$ and $\phi\in J$, such that\footnote{\, Therefore, $R$ is {\it uniformly integrable} with respect to $\Omega$. In particular,  for any $\omega\in\Omega$,  $R(\cdot,\omega)$ is $L^1$-integrable along any line $\ell_{\phi,b}$} 
\beq
\label{9settembre2023-5}
|R(z,\omega)|\leq g(z) \hbox{ almost eveywhere in $\ell_{\phi,b}$, for any $\omega\in\Omega$},
\eeq
and moreover
\beq\label{gdc-bis}
\lim_{b\to+\infty}\left(\underset{\omega\in \Omega}{\sup}\,\,\underset{\phi\in J}{\sup}{\pnorm{1}{R(\cdot,\omega) |_{\ell_{\phi,b}}}}\right)=0.
\eeq 
\end{defn}

\begin{example}
The function $R\in\hol(\H_{I,a}\times \Om;\gl(n,\C))$ satisfies the uniform good decay condition if there exists $\dl>0$ such that $R=O(z^{-1-\dl})$ on $\H_{J,a}$, for any closed interval $J\subseteq I$, and uniformly on $\Om$, that is
 \beq
\underset{\omega\in \Omega}{\sup}\,\, \underset{z\in \H_{J,a}}{\sup}|z^{1+\dl}R(z,\omega)|<\infty.
 \eeq
\end{example}

\begin{rem}
If $R\in\hol(\H_{I,a};\gl(n,\C))$ satisfies the good decay condition, then we have 
\beq
\label{9settembre2023-2-bis}
M_{R,a,J}:=\underset{\omega\in \Omega}{\sup}\,\, \underset{b\geq a}{\sup} \,\,\underset{\phi\in J}{\sup}\pnorm{1}{R(\cdot,\omega)|_{\ell_{b,\phi}}}<\infty.
\eeq

\end{rem}

An  analog of Theorem \ref{mainth} reads as follows. 

\begin{thm}
\label{mainth-bis}Assume that $\La\in\hol(\H_{I,a}\times \Omega;~\h(n,\C))$ satisfies both the uniform integrability condition and the {uniform  $L$-condition} {on $I$}, and that $R\in\hol(\H_{I,a}\times \Omega;~\gl(n,\C))$ satisfies the  {uniform  good decay condition} {on $I$}. Then
the differential equation
\beq
\label{diffeq-bis}
\frac{dY}{dz}=\Bigl(\La(z,\omega)+R(z,\omega)\Bigr)Y
\eeq 
admits a unique holomorphic fundamental matrix solution $Y\colon \H_{I,a}\times \Omega\to GL(n,\C)$ which,   for any closed $J\subset I$ and   $z\to \infty$ in $\H_{J,a}$, behaves as 
$$
Y(z,\omega)=\left(I_n+o(1)\right)\exp\left(\int_{z_o}^z\La(t,\omega){\rm d}t\right)\quad\hbox{uniformly in $\Omega$},
$$
where $z_o$ is an arbitrary point of $\H_{I,a}$.
\end{thm}

\begin{cor}
In the same assumptions of Theorem \ref{mainth-bis}, if $R=O(z^{-1-\dl})$ on $\H_{I,a}$ uniformly in $\Omega$, then the above unique solution  
is such that
$$
Y(z,\omega)=\left(I_n+O(z^{-\dl})\right)\exp\left(\int_{z_o}^z\La(t,\omega){\rm d}t\right),\quad\hbox{uniformly in $\Omega$}.
$$
\end{cor}

\vskip 0.3 cm 
\subsection{Proof of Theorem \ref{mainth-bis}} 
It is similar to that of  Theorem \ref{mainth}, with the modifications below.  

\begin{lem}
\label{9settembre2023-9}
 If $\Lambda$ satisfies the uniform integrability condition of Definition \ref{9settembre2023-6}, then 
$W_{\pm,j_o}\in\hol(\H_{I,a}^{\times 2}\times \Omega;\h(n,\C))$.
\end{lem}

\proof
The functions 
 $W_{\pm,j_o}$ are defined as in \eqref{15settembre2023-1} in terms of the integrals $\exp\int_{\zeta_1}^{\zeta_2}(\La_{ii}(t,\omega)-\La_{j_oj_o}(t,\omega))\d t$.  
For any $i=1,\dots,n$, 
$$ 
F_i(\zeta_1,\zeta_2,\omega):=\int_{\zeta_1}^{\zeta_2}\La_{ii}(t,\omega) \d t$$
is analytic w.r.t.\,\,$(\zeta_1,\zeta_2)$, and by \eqref{9settembre2023-3} it is continuous\footnote{ If $f(x,y)$ defined on $\mathbb{R}^n\times \mathbb{R}^m$  is continuous in $y_o\in \mathbb{R}^m$ for almost every $x\in \mathbb{R}^n$, and $|f(x,y)|\leq g(x)$ for every $y$ and for almost every $x$, where $g$ is $L^1$-integrable, then $F(y):=\int f(x,y) dx$ exists, and is continuous in $y_o$.} w.r.t.\,\,$\omega=:(\omega_1,\dots,\omega_m)\in \Omega$. Take a section of $\Omega$ by fixing all the components except for $\omega_k$, and let $\gamma$ be a closed curve, in this section, contractible to a point. Then 
\beq
\label{9settembre2023-7}
\oint_\gamma F_i(\zeta_1,\zeta_2,\omega)d\omega_k=\oint_\gamma \left(\int_{\zeta_1}^{\zeta_2}\La_{ii}(t,\omega) \d t\right)\d\omega_k
=\int_{\zeta_1}^{\zeta_2} \left(\oint_\gamma \La_{ii}(t,\omega) \d\omega_k\right)\d t =0. 
\eeq 
The  last equality is a consequence of the simple-connectedness of the domain bounded by $\gamma$ and of the holomorphicity of $ \La_{ii}(t,\omega)$ on it. The exchange of order of integration follows from Fubini's theorem, because by \eqref{9settembre2023-3} we have
$$
\oint_\gamma \left(\int_{\zeta_1}^{\zeta_2}|\La_{ii}(t,\omega)| \d |t|\right)\d
|\omega_k| \leq \oint_\gamma \d |\omega_k| ~\int_{\zeta_1}^{\zeta_2} g_i(t) \d |t|<+\infty.
$$
By \eqref{9settembre2023-7} and  Morera's theorem, $F_i(\zeta_1,\zeta_2,\omega)$ is then analytic with respect to  $\omega_k$.
The lemma is proved by invoking Hartogs's theorem on separate holomorphicity. 
\endproof

Let  $\widehat{\mathcal{H}}_{J,a}$  be the space of {\it continuous bounded} function
on $\H_{J,a}\times\Omega $ with values in $\C^n$, whose restriction to $\on{int} \H_{J,a}\times\Omega$ is {\it analytic}. It is a Banach space with norm  $\snorm{f}:=\underset{\omega,\in\Omega~z\in\H_{J,a}}{\sup}{|f(z,\omega)|}$. 
 The following operators, analogous to  \eqref{k+}-\eqref{k-},  act  on  $\widehat{\mathcal{H}}_{J,a}$: \
\begin{align}
\label{k+bis}
K_{+,j_o}[f;z,\omega]&:=\int_{\infty}^zW_{+,j_o}(t,z,\omega) R(t,\omega) f(t,\omega) \d t,
\\
\label{k-bis}
K_{-,j_o}[f;z,\omega]&:=\int_{z}^\infty W_{-,j_o}(t,z,\omega) R(t,\omega) f(t,\omega) \d t.
\end{align}
Lemma \ref{lemmak} is then generalized by the following one.

\begin{lem}\label{lemmak-bis} 
Assume that $\La\in\hol(\H_{I,a}\times \Omega;~\h(n,\C))$ satisfies both the uniform integrability condition and the {uniform  $L$-condition} {on $I$}, and that $R\in\hol(\H_{I,a}\times \Omega;~\gl(n,\C))$ satisfies the  {uniform  good decay condition} {on $I$}. For any closed interval $J\subset I$ the following statements hold.
\begin{enumerate}
\item The operators $K_{\pm,j_o}$ are well-defined and bounded  on $\widehat{\mathcal{H}}_{J,a}$. In particular, there exists a constant $C>0$, independent of $R$ and $\om$, such that 
\beq\label{normk-bis}
\snorm{K_{\pm,j_o}[f;~\cdot,~\cdot]}\leq C\,\snorm{f}\,M_{R,a,J},
\eeq where $M_{R,a,J}$ is defined in \eqref{9settembre2023-2-bis}. Moreover, for $z\to \infty$ in $\H_{J,a}$, we have
  \beq
  \label{15settembre2023-3}
  \lim_{z\to\infty}\left( \underset{\omega\in \Omega}{\sup}\,\,K_{\pm, j_o}[f;z,\omega]\right)=0.
  \eeq
  \item If $R=O(z^{-1-\dl})$ on $\H_{I,a}$ uniformly on $\Om$, then there exists $C>0$ such that for any $f\in \widehat{\mc H}_{J,a}$ we have
$$
 \underset{\omega\in \Omega}{\sup}\,|K_{\pm,j_o}[f;z,\omega]|\leq C \snorm{f}|z^{-\dl}|.
 $$
\end{enumerate}
\end{lem}

\proof The structure of the proof is similar to that of  Lemma \ref{lemmak}. In particular, for point (1),
the proof that the facts (i), (ii), (iii) and (iv) hold uniformly in $\omega\in \Omega$
is exactly the same as in Lemma \ref{lemmak}. It remains to prove that:
\begin{enumerate}
\item[(v)]
 $K_{+,j_o}[f;z,\omega]$ analytically depends on $\omega\in \Omega$.
 \end{enumerate}
   We study the case of $K_+$, the one for $K_-$ being analogous.  The function $K_{+,j_o}[f;z,\omega]$ continuously depends on $\omega$. Indeed, the $L$-condition 
   and \eqref{9settembre2023-5} imply the uniform integrability 
 $$ 
 \left|
 W_{+,j_o}(t,z,\omega) R(t,\omega) f(t,\omega)
 \right|
 \leq C \|f\|_\infty ~g(t)\in L^1(\ell_{\phi,b},dt).
 $$
   As in Lemma \ref{9settembre2023-9}, consider a closed curve $\gamma$, contractible to a point,  in a section of $\Omega$ where only  $\omega_k$ varies. Then, the analyticity  of $K_{+,j_o}[f;z,\omega]$ w.r.t.\,\,$\omega_k$ is a consequence of  Morera's theorem, since
  {\small $$ 
 \oint_\gamma \left(\int_{\infty,\ell_{\phi,b}}^z W_{+,j_o}(t,z,\omega) R(t,\omega) f(t,\omega) \d t \right)\d\omega_k =\int_{\infty,\ell_{\phi,b}}^z   \left(\oint_\gamma   W_{+,j_o}(t,z,\omega) R(t,\omega)f(t,\omega) \d \omega_k\right)  \d t =0.
 $$
}

\noindent
 The last equality follows from the simple-connectedness of the domain bounded by $\gamma$ and from   the holomorphicity of the integrand w.r.t. $\omega$ (see assumption on $R$ and Lemma \ref{9settembre2023-9}). The exchange of order of integration follows from Fubini's theorem, since 
{\small $$ 
 \oint_\gamma \left(\int_{\infty,\ell_{\phi,b}}^z \Bigl|W_{+,j_o}(t,z,\omega)\Bigr| ~|R(t,\omega)|~| f(t,\omega)| \d |t| \right)\d|\omega_k|
 \leq 
 C \|f\|_\infty \oint_\gamma \d|\omega_k| 
 \int_{\infty,\ell_{\phi,b}}^z |R(t,\omega)| ~\d|t|<+\infty.
 $$
}By Hartogs's theorem on separate holomorphicity, we deduce (v).

Point (2) is similar to  point (2) of Lemma \ref{lemmak}.
\endproof

 The change  of variables 
$
Y(z,\omega)=Z(z,\omega)\exp\int_{z_o}^z\La_{j_oj_o}(t,\omega)\d t,
$
for fixed $j_o\in\{1,\dots,n\}$, and $z_o\in\H_{I,a}$,  yields
$$
\frac{dZ}{dz}=\left(\Tilde\La(z,\omega)+R(z,\omega)\right)Z,\quad \Tilde\La(z,\omega)=\La(z,\omega)-\La_{j_oj_o}(z,\omega)I.
$$
The proof that the equation above has a unique holomorphic  vector solution $Z(z,\omega)$  such that, for any closed interval $J\subseteq I$ and $z\to \infty$ in $ \H_{J,a}$, one has
\beq\label{modifiedasym-bis}
Z(z,\omega)=e_{j_o}+o(1)\quad \text{uniformly in $\Omega$},
\eeq 
 is essentially unchanged compared to the parameter-less case. 
Indeed, for  the integral equation
\beq\label{integeq-bis}
f(z,\omega)= e_{j_o}+K_{+,j_o}[f;z,\omega]-K_{-,j_o}[f;z,\omega],
 \eeq
the proof of Lemma \ref{lemdiffeqintegeq} is unchanged (it does not involve $\omega$).  
Now, \eqref{integeq-bis} is the equation  $
 ({\rm Id}-K_{j_o})[f]=e_{j_o}$,
    with operator 
 $K_{j_o}:\widehat{\mathcal{H}}_{J,a}\longrightarrow \widehat{\mathcal{H}}_{J,a}
 $, where $
 K_{j_o}:=K_{+,j_o}-K_{-,j_o}$.    
 For any $J\subset I$,  it follows from  \eqref{gdc-bis} and \eqref{normk-bis} that, by replacing $a$ with $a^\prime(J)>a$ if necessary,  $K_{j_o}$ has norm less that one.   Hence, \eqref{integeq-bis}  has a unique solution, given by   $ 
 Z=  ({\rm Id}-K_{j_o})^{-1}[e_{j_o}] 
 ~ \in \widehat{\mathcal{H}}_{J,a}
 $, expandable in  Neumann series. 
 Finally, \eqref{15settembre2023-3} implies the  behaviour \eqref{modifiedasym-bis}.

\subsection{Subdominant solutions with parameters}\label{secsubpar}
We consider the case when $\La\in\mathscr O(\H_{\widetilde{I},a}\times \Omega;\h(n,\C))$  and
$R\in\mathscr O(\H_{\widetilde{I},a}\times \Omega;\gl(n,\C))$, where $\Omega$ is an open connected set as in Section \ref{20settembre2023-1}. 

\begin{defn} In the same notations of Definition    \ref{21giugno2023-1}, 
an index $j_o \in \{1,\dots,n\}$ is {\it subdominant} {\it on $I$} for the holomorphic function
 $\La\in\hol(\H_{\widetilde{I},a}\times \Omega;\h(n,\C))$
if for  any $J\subseteq I$ closed
interval, there exist a $K_2\in\R$ such that
\begin{equation}
\label{conditionS-bis}
\underset{\omega\in \Omega}{\inf}\,\,\left(  \underset{\phi\in J}{\inf}\,\,\underset{b \in
\R}{\inf}\,\,\underset{\substack{w\leq z\\(w,z)
 \in(\iota^{\varphi}_{a,b})^{\times
2}}}{\inf}\on{Re}\int_{w,\iota^{\varphi}_{a,b}}^z\left(\La_{ii}(t,\omega)-\La_{j_oj_o}(t,\omega)\right){\rm
d}t\right) \geq K_2,
  \qquad \forall i \in \lbrace 1 \dots n \rbrace.
\end{equation}
\end{defn}

\begin{defn}
\label{21giugno2023-1-tris}
A holomorphic function $R\in\hol(\H_{\widetilde{I},a};\gl(n,\C))$
satisfies the {\it uniform S-decay condition} if   there is a positive function $g$ which, for any closed interval $J\subseteq I$, is  $L^1$-integrable  along any line $\iota_{a,b}^\phi$, for any $b\in \mathbb{R}$ and $\phi\in J$, such that  for any $\omega\in\Omega$
\beq
\label{9settembre2023-5-bis}
|R(z,\omega)|\leq g(z) \hbox{ almost eveywhere in $\iota_{a,b}^\phi$}.
\eeq
Moreover, for any closed interval $J\subseteq I$,
\begin{equation}
\label{eq:MRaJS-bis}
 \widetilde{M}_{R,a,J}:= \underset{\omega\in \Omega}{\sup} \,\,\left(\underset{\phi\in J}{\sup} \; \underset{b \in
\R}{\sup}{\pnorm{1}{R|_{\iota^{\varphi}_{a,b}}}}\right) < +\infty,
\end{equation}
and
\begin{equation}
\label{eq:MRaJSinfty-bis}
 \lim_{a\to +\infty} \underset{\omega\in \Omega}{\sup} \,\,\widetilde{M}_{R,a,J}= 0.
\end{equation}
\end{defn}

The following results are proved as in Sections \ref{20settembre2023-2} and  \ref{20settembre2023-1}. 

\begin{thm}
\label{subth-bis}
Consider the differential equation \eqref{diffeq-bis} 
 with $\La\in\mathscr O(\H_{\widetilde{I},a}\times \Omega;\h(n,\C))$  and
$R\in\mathscr O(\H_{\widetilde{I},a}\times \Omega;\gl(n,\C))$.
Assume that $j_o$ is a subdominant index on $I$, that $\La$ satisfies the uniform integrability condition \eqref{9settembre2023-3}  on $\widetilde{I}$, and that  $R$
satisfies the uniform S-decay condition on $I$. 
There exists a unique vector solution $y_{j_o}\in\mathscr
O(\H_{\widetilde{I},a}\times \Omega;\C^n)$ with asymptotic behaviour,
\beq\label{subasym-bis}
y_{j_o}(z,\omega)=\left(e_{j_o}+o(1)\right)\exp\left(\int_{z_o}^z\La_{j_oj_o}(t,\omega){\rm
d}t\right)\quad\hbox{ uniformly in $\Omega$},
\eeq
for any closed interval $\widetilde{J}\subset \widetilde{I}$ and
$z\to \infty$ in $\H_{\widetilde{J},a}$, 
where $z_o$ is an arbitrary point of $\H_{\widetilde{I},a}$.
\end{thm}

\begin{cor}
\label{subcor-bis}
Under  the same assumptions as in Theorem \ref{subth-bis}, if $R=O(z^{-1-\dl})$ on
$\H_{\widetilde{I},a}$ uniformly with respect to $\Omega$, then the unique solution $y_{j_o}$  of Theorem
\ref{subth-bis} has behaviour
\beq
\label{subasym-tris}
y_{j_o}(z,\omega)=\left(e_{j_o}+O(z^{-\dl})\right)\exp\left(\int_{z_o}^z\La_{j_oj_o}(t,\omega){\rm
d}t\right),\quad\hbox{ uniformly in $\Omega$},
\eeq
for $z\to\infty$ in $\H_{\widetilde{J},a}$ and any closed interval
$\widetilde{J}\subset\widetilde{I}$.
\end{cor}

\section{An ODE with not-necessarily meromorphic coefficients}

\subsection{Setup and Stokes rays}\label{secstok} Let $\C_a:=\{|z|>a>0\}$, and $\Tilde{\C}_a:=\Pi^{-1}(\C_a)$ (in the notations of Section \ref{secnot}). We apply Theorem \ref{mainth} to an important case of equation \eqref{diffeq},  with 
\beq
\label{2giugno2023-1}
             \Lambda(z)=\sum_{k=0}^h \Lambda^{(k)} z^{\sigma_k -1},
\quad
              \Lambda^{(k)} \in \h(n,\C),
 \quad 
              1=\sigma_0>\sigma_1>\dots>\sigma_h>0,\quad \si_i\in\R_{>0},
\eeq
and $R\in \hol(\Tilde{\C}_{a};\gl(n,\C))$, satisfying the good decay condition of Definition \ref{5giugno2023-1} on any open interval $I$ (for example, $R=O(|z|^{-1-\dl})$ on $\Tilde{\C}_{a}$).     This equation is the model to which a wide class of differential equations, such as those of Section \ref{16giugno2023-1}, can be reduced.    
We arrange a block-diagonal partition defined up to permutation:
\beq
\label{21giugno2023-2}
\Lambda(z)=\Lambda_1(z)\oplus \dots \oplus \Lambda_\ell(z),\qquad\quad
 \Lambda_i(z)=\sum_{k=0}^h \sigma_k\lambda_i^{(k)}  z^{\sigma_k -1} I_{s_i},\quad\quad {\lambda}_i^{(k)}\in\C,
 \eeq
where $ I_{s_i}$ is the identity matrix of dimension $s_i\in\N_{>0}$, with $s_1+\dots+s_\ell=n$, and
\beq
\label{18settembre2023-4}
\boldsymbol{\lambda}_i\neq \boldsymbol{\lambda}_j\hbox{ for } i\neq j,
\quad 
\hbox{ with }
\boldsymbol{\lambda}_i:=(\lambda^{(0)}_i,\lambda^{(1)}_i,\dots,\lambda^{(h)}_i), \quad i,j=1,\dots,\ell.
\eeq
 Let
\begin{equation}
\label{7settembre2023-2}
Q(z):= \int^z  \Lambda(\zeta)d\zeta =q_1(z)I_{s_1}\oplus\dots\oplus q_\ell(z)I_{s_\ell}
,
\quad 
\quad
q_i(z):=\sum_{k=0}^h \lambda_i^{(k)} z^{\sigma_k}.
\end{equation}
For any $i,j=1,\dots,\ell$, with $i\neq j$, set 
\beq
\label{19settembre2023-1}
k_{ij}:=\on{min}\{k~|~\lambda_i^{(k)}-\lambda_j^{(k)}\neq 0\},
\quad\quad
\lambda_{ij}:=\lambda_i^{(k_{ij})}-\lambda_j^{(k_{ij})}.
\eeq

\begin{defn}
\label{9giugno2023-2}
For each pair $(i,j)$, with $1\leq i\neq j \leq \ell$, the {\it Stokes rays of \eqref{2giugno2023-1} at $z=\infty$}   are the half lines in $\Tilde{\C^*}$ defined by 
$$ 
\on{Re} \lambda_{ij} z^{\sigma_{k_{ij}}}=0, \quad \on{Im} \lambda_{ij} z^{\sigma_{k_{ij}}}<0.
$$ 
\end{defn}
Hence, all the Stokes rays for $(i,j)$ have directions
\begin{equation}
\label{1ottobre2024-1}
\arg z= \tau_{ij}+\frac{2\pi m}{\sigma_{k_{ij}}}, \quad \hbox{ where  } m\in\mathbb{Z},\quad  \tau_{ij}:=\frac{1}{\sigma_{k_{ij}}}\left(\frac{3\pi}{2} -\arg\lambda_{ij}\right),
\end{equation}
and  $\arg\lambda_{ij}$ is a chosen determination of the argument. 

\begin{rem}
\label{9giugno2023-2-bis} 
For any $i,j=1,\dots,\ell$, with $i\neq j$, let $T_{ij}\subseteq\R$ be the complement of the discrete set of real numbers $\tau_{ij}+k\pi/\sigma_{k_{ij}}$, with $k\in\Z$.
We have
\begin{equation}
\label{26maggio2023-2} 
q_i(z)-q_j(z)= \lambda_{ij} z^{\sigma_{k_{ij}}}(1+o(1)),
\end{equation}
where $o(1)\to 0$ when $|z|\to \infty$, uniformly with respect to $\arg z$ varying in a closed interval in $T_{ij}$. 
\end{rem}

We enumerate the Stokes rays  as follows. Let $\eta\in\R$ be such that $\eta\neq \arg\lambda_{ij} \on{mod}\pi$ for any $i\neq j$. Consider all the distinct values taken by the determinations   $\arg\lambda_{ij}$, for $1\leq i\neq j \leq \ell$, lying  in the interval  $]\eta-\pi,\eta[$, and denote them by
\beq
\label{9giugno2023-4}
\eta_{\widetilde{\mu}-1}<\dots<\eta_2<\eta_1<\eta_0,\quad\quad \widetilde{\mu}\in\N_{\geq 1}.
\eeq
The directions $\tau_{ij}$, corresponding to the determinations $\arg\lambda_{ij}\in  \{\eta_0,\cdots,\eta_{\widetilde{\mu}-1}\}$
will be labeled  in increasing partial order
\beq
\label{sigmarho}
\tau_0\leq \tau_1\leq\cdots\leq \tau_{\mu-1},\quad\hbox{ where } \widetilde{\mu}\leq  \mu\in\mathbb{N}, 
\eeq 
according to the following rules. Let $(i,j)\neq (c,d)$, with $1\leq i\neq j \leq \ell$ and $1\leq c\neq  c\leq \ell$.
\begin{itemize}
\item[(1)] If $\tau_{ij}<\tau_{cd}$, we label them as $\tau_\rho$ and $\tau_{\rho+1}$ respectively, for some $\rho\in\{0,1,\cdots,\mu-2\}$.
\item[(2)]   If $\tau_{ij}=\tau_{cd}$, and $\sigma_{k_{ij}}=\sigma_{k_{cd}}$, we label both of them as  $\tau_\rho$ (i.e $\tau_\rho=\tau_{ij}=\tau_{cd}$),  for some $\rho\in\{0,1,\cdots,\mu-1\}$.
\item[(3)]   If $\tau_{ij}=\tau_{cd}$, and $\sigma_{k_{ij}}\neq \sigma_{k_{cd}}$, we label them as $\tau_\rho$ and $\tau_{\rho+1}$ respectively, or equivalently as $\tau_{\rho+1}$ and $\tau_{\rho}$, for some $\rho\in\{0,1,\cdots,\mu-2\}$. In this  case the  equality $\tau_{\rho}=\tau_{\rho+1}$ holds. 
\end{itemize}

Note that 
   $\mu= \widetilde{\mu}$ if and only if  there are  strict inequalities in  \eqref{sigmarho}, namely when    $\arg\lambda_{ij}=\arg\lambda_{cd}$ only if   $\sigma_{k_{ij}}=\sigma_{k_{cd}}$.

\begin{example}
If  $\arg \lambda_{ij}=\arg\lambda_{cd}=3\pi/2\in]\eta-\pi,\eta[$ for $(i,j)\neq(c,d)$, but  $\sigma_{k_{ij}}\neq \sigma_{k_{cd}}$, then $\tau_{ij}=\tau_{cd}=0$, and  $\tau_{\rho}=\tau_{\rho+1}=0$ in \eqref{sigmarho} for some $\rho\in\{0,\dots,\mu-2\}$.
\end{example}

\begin{defn}[of the exponent $\sigma^{(\rho)}$]
If  $\tau_{ij}$ is  labelled as $\tau_\rho$ in \eqref{sigmarho}, where $0\leq \rho\leq \mu-1$, then the corresponding $\sigma_{k_{ij}}$ will be denoted by  $\sigma^{(\rho)}$. 
\end{defn}

\begin{prop}
\label{16novembre2024-2}
  All directions of the Stokes rays are  given by 
\begin{equation}
\label{29maggio2023-2-1}
\arg z=\tau_\rho+\frac{k \pi}{\sigma^{(\rho)}},\quad \rho\in\{0,\dots,\mu-1\},\quad k \in\mathbb{Z}.
\end{equation}
\end{prop}
\begin{proof}
Since $\lambda_{ji}=-\lambda_{ij}$,  all the possible determinations of the $\arg\lambda_{ij}$, for $i\neq j\in \{1,\dots,\ell\}$,  are the numbers   $\eta_\alpha+m\pi$, for $\alpha=0,...,\widetilde{\mu}-1$ and  $m\in\mathbb{Z}$.  Then, the  statement follows straightforwardly from \eqref{1ottobre2024-1} and the labelling previously introduced.
\end{proof}
\begin{rem}
It is clear now why in  the rule (3) above we must  label differently $\tau_{ij}$ and $\tau_{cd}$ even though  $\tau_{ij}=\tau_{cd}$. This is  because  if  $\tau_\rho=\tau_{\rho+1}$ with $\sigma^{(\rho)}\neq \sigma^{(\rho+1)}$, then $\tau_\rho$  and $\tau_{\rho+1}$ generate different Stokes directions \eqref{29maggio2023-2-1}. See also Example \ref{17novembre2024-1} below. 

\end{rem}

The simplest  to describe is the following generic case.
\begin{defn}
\label{19giugno2023-1}
$\Lambda$ in \eqref{2giugno2023-1}-\eqref{21giugno2023-2} is  {\it generic} if  $\lambda_i^{(0)}\neq \lambda_j^{(0)}$ for all $1\leq i\neq j\leq \ell$. 
\end{defn}

\begin{example}
If in \eqref{2giugno2023-1} we have $\Lambda^{(k)}=c_k\Lambda^{(0)}$, with $c_k\in\C$,   for all $k=1,\dots,h$, then $\Lambda$ is generic. 
\end{example}

For generic $\Lambda$, we have $k_{ij}=0$, $\sigma_{k_{ij}}=\sigma_0=1$, and  $\lambda_{ij}=\lambda_i^{(0)}-\lambda_j^{(0)}$, for all $i\neq j$. 
In this case, $\sigma^{(0)}=\dots=\sigma^{(\mu-1)}=1$ and  
\begin{equation}
\label{29maggio2023-1-1}
\tau_\rho=\frac{3\pi}{2}-\eta_\rho,\quad \rho=0,1,\cdots,\mu-1 ~(=\widetilde{\mu}-1).
\end{equation}
In this case, all the directions \eqref{29maggio2023-2-1}  can be conveniently  labelled as 
 \begin{equation}
 \label{29maggio2023-2}
 \tau_\nu=\tau_\rho+k\pi,\quad \rho\in\{0,\dots,\mu-1\},\quad \nu:=\rho+k\mu, \quad k\in\Z.
 \end{equation}
They satisfy the strict inequality
\begin{equation}
\label{7settembre2023-1}
\tau_{\nu}<\tau_{\nu+1}\quad \forall~\nu\in \Z.
\end{equation}


\subsection{Adequate tuples, and first application of the main theorem}\label{secadeqtripl1}

\begin{defn}
\label{18settembre2023-3}
We say that $\bm k=(k_0,\cdots,k_{\mu-1})\in \Z^\mu$ is {\it adequate with respect to $(\si^{(\rho)};\tau_\rho)_{\rho=0}^{\mu-1}$} if the open interval 
\beq
\label{26maggio2023-10}
\Si_{\bm k}:=\bigcap_{\rho=0}^{\mu-1}\,\left]\tau_\rho+\frac{k_\rho \pi}{\si^{(\rho)}}-\pi;\,\tau_\rho+\frac{k_\rho \pi}{\si^{(\rho)}}\right[
\eeq is non-empty.
\end{defn}
Geometrically, adequateness means that the Stokes rays with directions $\arg z=\tau_\rho+\frac{k_\rho \pi}{\si^{(\rho)}}$, for $\rho=0,\dots,\mu-1$, are in an open interval of angular amplitude {\it less than} $\pi$. 
Applying Theorem \ref{mainth},  we will prove  the following theorem. 

\begin{thm}
\label{5giugno2023-3}
Consider the differential equation \eqref{diffeq}
with  $\Lambda(z)$ as in \eqref{2giugno2023-1}, and a perturbation term $R\in \hol(\Tilde{\C}_{a};\gl(n,\C))$ satisfying the good decay condition on any open interval $I$.
Let $\bm k\in\Z^\mu$ be adequate with respect to $(\si^{(\rho)};\tau_\rho)_{\rho=0}^{\mu-1}$. 
Then, there exists a
 unique fundamental matrix solution $Y_{\boldsymbol{k}}\colon \H_{I,a}\to GL(n,\C)$, with $I=\Si_{\bm k}+\frac{\pi}{2}$,  with  behaviour
\beq
\label{9giugno2023-1}
Y_{\boldsymbol{k}}(z)=\left(I_n+o(1)\right) \exp\bigl\{Q(z)\bigr\},
\eeq 
for  $z\to \infty$    in $\H_{J,a}$, in any closed interval  $J\subset I$. Here, $Q$ is defined as in \eqref{7settembre2023-2}. Equivalently,  the behaviour \eqref{9giugno2023-1} holds for $z\to \infty$ in any closed subsector of 
$$ 
S_{\boldsymbol{k}}:=\{z\in \Tilde{\C^*}~|~a_{\boldsymbol{k}}<\arg z <b_{\boldsymbol{k}}+\pi \}, \quad \text{where }\Si_{\bm k}=\left]a_{\bm k};b_{\bm k}\right[.
$$
\end{thm}

\begin{cor}
\label{7giugno2023-2} 
Consider the differential equation \eqref{diffeq} as in Theorem \ref{5giugno2023-3}. If $\Lambda$ is generic, then for every $\nu\in \mathbb{Z}$ there exists a unique fundamental matrix solution $Y_\nu \colon\H_{I,a}\to GL(n,\C)$, with $I=]\tau_{\nu-1},\tau_\nu[+\frac{\pi}{2}$,   
 such that
 $$
Y_\nu(z)=\left(I_n+o(1)\right) \exp\bigl\{Q(z)\bigr\},
$$
  $z\to \infty$    in $\H_{J,a}$, in any closed interval $J\subset I$. Equivalently,  the above behaviour  holds for $z\to \infty$ in any closed subsector of 
$$ 
S_\nu:=\{z\in \Tilde{\C^*}~|~\tau_{\nu-1}<\arg z <\tau_\nu +\pi\}.$$ 
Here, the direction $\tau_\nu$ of a Stokes ray is defined in \eqref{29maggio2023-2}.
\end{cor}

\proof[Proof of Corollary \ref{7giugno2023-2}]
For generic $\Lambda$,  $\sigma^{(\rho)}=1$ for any $\rho$. Thus, \eqref{26maggio2023-10} becomes
\begin{equation}
\label{26maggio2023-10-bis}
\bigcap_{\rho=0}^{\mu-1}\,\left]\tau_\rho+(k_\rho-1) \pi;\,\tau_\rho+k_\rho \pi\right[.
\end{equation}
  It is always possible to find  $\boldsymbol{k}$ such that \eqref{26maggio2023-10-bis} is non-empty. 
Indeed, for any arbitrary $k\in\Z$, we have the following choices 
$$
\boldsymbol{k}=(k_0,\dots,k_{\mu-1})=
\left\{
\begin{aligned}
&(k,\dots,k,k),
\\&(k,\dots,k,k-1),
\\&(k,\dots,k-1,k-1),
\\
&\vdots
\\
&(k-1,\dots,k-1,k-1).
\end{aligned}
\right.
$$
The above, together with  definition \eqref{29maggio2023-2}, proves that the non-empty \eqref{26maggio2023-10-bis} are  the intervals $]\tau_{\nu-1},\tau_\nu[$, with $\nu\in \mathbb{Z}$. Then, Corollary \ref{7giugno2023-2} follows form Theorem \ref{5giugno2023-3}. 
\endproof

\begin{rem}
  Adequate $\bm k$ may not exist. Consider a case with  $\mu=\widetilde{\mu}=2$  and  $\eta_0=2\pi$, $\eta_1=3\pi/2$. Then 
$\sigma^{(0)} \tau_0=-\pi/2$, $ \sigma^{(1)} \tau_1=0$. 
 If $\sigma^{(0)}=\sigma^{(1)}=1/2$, the Stokes rays have directions $ 
\tau_0+2k_0\pi$ and $ \tau_1+2k_1\pi=2k_1\pi$. 
Then \eqref{26maggio2023-10}  defines the intervals  
$$ 
\left](2k_0-2)\pi;(2k_0-1)\pi\right[\,\,\cap\,\,\left](2k_1-1)\pi;2k_1\pi\right[=\emptyset.
$$
\end{rem}

\subsection{Proof of Theorem \ref{5giugno2023-3}}\label{sec33}
Take $\lambda\in\C$ and $\sigma\in \R$. 
An oriented line $\ell_{b,\phi}$ has equation  $b=|z|\cos(\phi-\arg z)$ and direction
$$\tau:=\phi-\pi/2.
$$ 
Hence, 
$$
\on{Re}(\lambda z^\sigma)\Bigr|_{\ell_{b,\phi}}=|\lambda||b|^\sigma \frac{\cos(\sigma\theta+\arg\lambda)}{\sin^\sigma(\theta-\tau)},\quad 
\quad
\theta:=\arg z.
$$ 
\begin{lem}
\label{26maggio2023-9}
 Take $\tau_*, \tau\in \R$,  $0<\sigma^{(*)}\leq 1$ and define $\eta_*:=
\dfrac{3\pi}{2}-\sigma^{(*)}\tau_*$. 
 The function
 \begin{equation}
\label{26maggio2023-4}
f_*(\theta):= \frac{\cos(\sigma^{(*)}\theta+\eta_*)}{\sin^{\sigma^{(*)}}(\theta-\tau)}
\equiv
\frac{\sin(\sigma^{(*)}(\theta-\tau_*))}{\sin^{\sigma^{(*)}}(\theta-\tau)},
\end{equation}
 is monotonic w.r.t.  $\theta\in]\tau,\tau+\pi[$ if and only if there is $k\in\Z$ such that 
 \begin{equation}
 \label{26maggio2023-7}
 \tau<\tau_*+\frac{k\pi}{\sigma^{(*)}}<\tau+\pi.
\end{equation}
\end{lem}

\begin{proof}
Take $\theta^\prime:=\theta-\tau$, $0<\theta^\prime<\pi$. Then
\begin{equation} 
\label{22giugno2023-2}
\frac{df_*}{d\theta}
=
-\sigma^{(*)}\frac{\cos((\sigma^{(*)}-1)\theta^\prime+\sigma^{(*)}\tau+\eta_*)}{\sin^{\sigma^{(*)}+1}\theta^\prime}
\end{equation}
The denominator in \eqref{22giugno2023-2} is positive. The argument in the numerator satisfies the inequalities
$$ 
\sigma^{(*)}\tau+\eta_*+(\sigma^{(*)}-1)\pi~<~ (\sigma^{(*)}-1)\theta^\prime+\sigma^{(*)}\tau+\eta_*~<~\sigma^{(*)}\tau+\eta_*.
$$
Hence, the numerator has constant sign if and only if there is $k\in \Z$ such that 
$$
 \sigma^{(*)}\tau+\eta_*+(\sigma^{(*)}-1)\pi>\frac{\pi}{2}+k\pi     
 \quad\hbox{and} \quad
 \sigma^{(*)}\tau+\eta_*<\frac{3\pi}{2}+k\pi 
.
$$
This is exactly \eqref{26maggio2023-7}. 
\end{proof}

\begin{lem} 
\label{26maggio2023-12}
In the notations of Lemma \ref{26maggio2023-9},  let $\tau$ satisfy 
\begin{equation}
\label{26maggio2023-3}
\tau\neq \tau_* +\frac{k\pi}{\sigma^{(*)}}\quad \hbox{ and }\quad  \tau+\pi\neq \tau_* +\frac{k \pi }{\sigma^{(*)}},\quad \text{for any }k\in\Z.
\end{equation}
Only one of the following mutually excluding cases occurs. 

1) The  function \eqref{26maggio2023-4} is   monotonic,  $\lim_{\theta\to \tau_+}f_*(\theta)$ and $\lim_{\theta\to \tau_-+\pi} f_*(\theta)$  are infinite with  opposite signs.

 2)  The function \eqref{26maggio2023-4} is not monotonic  and both limits are infinite with the same sign. 
\end{lem}
\begin{proof}
The denominator of \eqref{26maggio2023-4} 
is positive and vanishes as $\theta\to \tau,\tau+\pi$. By \eqref{26maggio2023-3}, the numerator does not vanish for  $\theta\to \tau,\tau+\pi$. Therefore, 
$$ 
\lim_{\theta\to \tau_+}|f_*(\theta)|=\lim_{\theta\to \tau_-+\pi} |f_*(\theta)|=+\infty.
$$
Since $0<\sigma^{(*)}\leq 1$, there can only be   one zero  $\theta=\tau_* +k \pi/\sigma^{(*)}$   of the numerator  lying in $]\tau,\tau+\pi[$,  for a suitable $k\in\mathbb\Z$. If follows form Lemma \ref{26maggio2023-9} that there exists  such a zero if and only if $f_*(\theta)$ is monotonic. In this case $ \lim_{\theta\to \tau_+}f_*(\theta)$ and $\lim_{\theta\to \tau_-+\pi} f_*(\theta)$  have different signs. Otherwise, 
 $f_*$ has no zero. By Lemma \ref{26maggio2023-9} it is  not monotonic, and therefore both limits have the same sign. 
\end{proof}

Theorem \ref{5giugno2023-3} now follows now from Theorem \ref{mainth} and the following result.
\begin{prop} 
  The  conditions  ($L$-1), ($L$-2) for  \eqref{diffeq} with matrix \eqref{2giugno2023-1} hold on an open interval $I$  if and only if there exists an adequate\footnote{The term is used here in the sense  of Definition \ref{18settembre2023-3}} $\bm k\in\Z^\mu$ such that $\Si_{\bm k}=I-\pi/2$.    
  
\end{prop}
\begin{proof} 
The $L$-condition involves integrals 
$ \on{Re}\int_w^z (\Lambda_{\alpha\alpha}(\zeta)-\Lambda_{\beta\beta}(\zeta))d\zeta$, $1\leq \alpha,\beta\leq n$ and $\al\neq \bt$, 
whose computation reduces to 
$$ 
\on{Re}(q_i(z)-q_j(z))-\on{Re}(q_i(w)-q_j(w))=(\on{Re}(\lambda_{ij} z^{\sigma_{k_{ij}}})-\on{Re}(\lambda_{ij} w^{\sigma_{k_{ij}}}))(1+o(1)).
$$
along a line $\ell_{b,\phi}$, where $o(1)$ vanishes for $|z|$ and $|w|\to \infty$, when $\tau=\phi-\pi/2$ satisfies \eqref{26maggio2023-3} with $*=0,1,\dots,\mu-1$. 
The study of the above expressions reduces to that of 
$$ 
 f_\rho(\theta)-f_\rho(\widetilde{\theta}),\quad \rho=0,\dots,\mu-1,
$$
where 
$$
f_\rho:=
\frac{\sin(\sigma^{(\rho)}(\theta-\tau_\rho))}{\sin^{\sigma^{(\rho)}}(\theta-\tau)}
$$
 is as in Lemma \ref{26maggio2023-9}  for $*=\rho$, and $\theta=\arg z$, $\widetilde{\theta}=\arg w$ along  $\ell_{b,\phi}$.  
Only two cases may occur by Lemma \ref{26maggio2023-12}. Namely, when $f_\rho$ is monotonic,  then $f_\rho(\theta)-f_\rho(\widetilde{\theta})$  has constant sign and   is either  bounded from above with $\on{inf}=-\infty$, or  bounded from below with $\on{sup}=+\infty$.  
When $f_\rho$ is not monotonic,  and has infinite limits at $\tau$ and $\tau+\pi$ with the same sign,  then  $f_\rho(\theta)-f_\rho(\widetilde{\theta})$ is unbounded and $\on{inf}=-\infty$, $\on{sup}=+\infty$.  Thus, by Lemma \ref{26maggio2023-9}, the  $L$-condition holds only if and only if  the conditions \eqref{26maggio2023-7} hold for all $\rho=0,\dots,\mu-1$.  \end{proof}

\subsection{Subdominant solutions} \label{sec34}
In the notations of Section \ref{secstok}, the block diagonal partition \eqref{21giugno2023-2} of $\La(z)$ induces a partition of the set of indexes 
\[\{1,\dots,n\}=\bigcup_{j=1}^\ell A_j,\quad A_j=\{s_1+\dots+s_{j-1}+1,\,\,\dots,\,\, s_1+\dots+s_{j-1}+s_j\}.
\]
If an index $i\in\{1,\dots, n\}$ is subdominant on $I$, as in Definition \ref{21giugno2023-1}, and if $i\in A_{j_o}$, then all the elements of $A_{j_o}$ are subdominant on $I$. In total we have $s_{j_o}$ subdominant vector solutions.
 
 \begin{defn}
  We say that the block index $j_o\in\{1,\dots,\ell\}$ is subdominant on $I$ if  the elements of the set $A_{j_o}$ are subdominant on $I$.   \end{defn}
  
Fix a $j_o\in\{1,\dots,\ell\}$. Then, for all $i\in\{1,\dots,\ell\}\setminus\{j_o\} $ and $ \arg\lambda_{ij_o}\in]\eta-2\pi,\eta[$, 
we consider the directions of Stokes rays given by 
  $$
  \tau_{ij_0}=\frac{1}{\sigma_{k_{ij_o}}}\left(\frac{3\pi}{2}-\arg\lambda_{ij_o}\right).
    $$
  We label them as
  $$ 
  \tau^{[0]}\leq \dots \leq \tau^{[\mu_o-1]},
  $$
  with the same rules as for \eqref{sigmarho}.  It is important to notice that here,  differently form  \eqref{sigmarho}, the lower bound for the determinations of $\arg\lambda_{ij_o}$ is $\eta-2\pi$.   We denote as  $\sigma^{[\beta]}$ the exponent $\sigma_{k_{ij_o}}$ appearing in  $\tau^{[\beta]}$.
  
  The analogous of Proposition \ref{16novembre2024-2}  holds. 
  
  \begin{prop}
  \label{16novembre2024-3}
The directions of all  Stokes rays associated with the pairs $(i,j_o)$, with $i\in\{1,\dots,\ell\}\setminus\{j_o\}$, are equal to 
$$
\arg z= \tau^{[\beta]}+\frac{2k\pi}{\si^{[\beta]}},\quad k\in\Z,\quad \beta=0,\dots,\mu_o-1.
$$
\end{prop}

In the generic case (see Definition \ref{19giugno2023-1}), all $\sigma_{k_{ij_0}}=1$ and the above description simplifies as follows.  The Stokes rays have directions labelled as in \eqref{29maggio2023-2}. In particular, we consider the consecutive  directions
    $$
    \tau_0<\dots<\tau_{\mu-1}<\tau_\mu<\dots<\tau_{2\mu-1},
    $$
    which  correspond to  all the determinations $\arg \lambda_{ij} \in ]\eta-2\pi,\eta[$, for   $1\leq i\neq j\leq \ell$.\footnote{If we denote the determinations $\arg \lambda_{ij} \in ]\eta-2\pi,\eta[$  as  $\eta_0,\dots,\eta_{2\mu-1}$ with ordering 
$$ 
\eta-2\pi<~\eta_{2\mu-1}<\dots<\eta_\mu ~<\eta-\pi<~\eta_{\mu-1}<\dots<\eta_0 ~<\eta,
$$
then 
$ 
\tau_\rho=\frac{3\pi}{2}-\eta_\rho$, with $\rho=0,1,\dots,2\mu-1$.
}
    The directions $\tau^{[\beta]}$ defined above form  a subset $\{ \tau_{\rho_0},\dots,\tau_{\rho_{\mu_o-1}}\}$ of $  \{\tau_0,\dots,\tau_{\mu-1},\tau_\mu,\dots,\tau_{2\mu-1}\}$, with ordering 
 $$ \tau_{\rho_0}<\dots<\tau_{\rho_{\mu_o-1}}
.
$$
 Proposition \ref{16novembre2024-3} has an immediate corollary.
  \begin{cor}
    The directions of all the Stokes rays associated with the pairs  $(i,j_o)$, where $i\in\{1,\dots,\ell\}\setminus\{j_o\}$, are equal to 
\beq
\arg z= \tau_{\rho_\gm}+2k\pi,\quad k\in\Z,\quad \gm=0,\dots,\mu_o-1.
\eeq
\end{cor}
   
   For simplicity, we state a theorem on subdominat solutions only in the generic case
  \begin{thm}
  \label{22giugno2023-5}
  Consider  the  differential equation \eqref{diffeq}
with generic $\Lambda(z)$,   with partition \eqref{21giugno2023-2}, and with a perturbation term $R\in \hol(\Tilde{\C}_{a};\gl(n,\C))$ satisfying the good decay condition on any open interval $I$.  Fix $j_o\in \{1,\dots,\ell\}$, and let $\tau_{\rho_0}<\tau_{\rho_1}<\dots<\tau_{\rho_{\mu_o-1}}$, be  $\mu_o$ consecutive directions of the Stokes rays associated with  the pairs $(i,j_o)$ for any  $i\in \{1,\dots,\ell\}$,  $i\neq j_o$.  
If 
$$\tau_{\rho_{\mu_o-1}}-\tau_{\rho_0}<\pi,
$$
 then, 
for any $k\in \Z$, there exist $s_{j_o}$ %
vector solutions 
$$y^{(k)}_{\alpha}\in\mathscr O(\Tilde{\C}_a;\C^n), \quad\quad
\alpha=s_1+\dots+s_{j_o-1}+m,\quad m=1,\dots,s_{j_o},
$$
 each of which is uniquely characterized by the asymptotic behaviour 
  $$
y^{(k)}_\alpha(z)=  (e_\alpha+o(1))\exp\left(\sum_{q=1}^h \lambda_{j_o}^{(q)}z^{\sigma_q}\right),
  $$
  for $z\to \infty$ in any closed subsector of the sector %
  $$ 
{S}_k^{(j_o)}:= \Bigl\{z\in \Tilde{\C}_a~\Bigl|~     \tau_{\rho_{\mu_o-1}}-\pi<\arg z-2k\pi< \tau_{\rho_0}+2\pi \Bigr\}.
  $$
  \end{thm}
  
  The proof is based on Theorem \ref{subth} and the following lemmas. 
  
  \begin{lem}\label{lemm1}
  The index $j_o\in\{1,\dots,\ell\}$ is subdominant on an interval $I$ %
  if and only if there exists $\boldsymbol{k}:=(k_0,\dots,k_{\mu_o-1})\in \Z^{\mu_o}$ such that 
  \begin{multline} \label{21giugno2023-5}
  \left(
  \bigcap_{\beta=0}^{\mu_o-1}\,\,\left]
  \tau^{[\beta]}+\frac{2k_\beta\pi}{\sigma^{[\beta]}};\,\,\tau^{[\beta]}+\frac{2k_\beta\pi}{\sigma^{[\beta]}}+\pi\right[\right)	\,\,\cap\,\, \left( \bigcap_{i\neq j_o}\,\,\left] \tau^{[\beta]}+\frac{(2k_\beta+1)\pi}{\sigma^{[\beta]}}-\pi;\,\, \tau^{[\beta]}+\frac{(2k_\beta+1)\pi}{\sigma^{[\beta]}}\right[\right)\\=I-\frac{\pi}{2}.
  \end{multline} 
  \end{lem}
  
  \begin{proof} The quantity $\int_w^z (\Lambda_{\alpha\alpha}-\Lambda_{\beta\beta}) dt$, with $\alpha,\beta\in\{1,\dots,n\}$,  in Definition   \ref{21giugno2023-1},  is replaced by 
\[
{\rm Re}(q_i(z)-q_{j_o}(z))-{\rm Re}(q_i(w)-q_{j_o}(w)),\quad 
\hbox{ where } z\geq w \hbox{ along }\iota_{b,\phi},\quad i\in\{1,\dots,\ell\}\backslash \{j_o\}.
\]
Suppose that  $\iota_{b,\phi}$ is not in the direction of a Stokes ray of either $(i,j_o)$ or $(j_o,i)$, so that \eqref{26maggio2023-2} holds. 
If $\phi-\frac{\pi}{2}<\arg z<\phi+\frac{\pi}{2}$ in \eqref{22giugno2023-1}, then  Lemmas \ref{26maggio2023-9} and \ref{26maggio2023-12} apply, and  
\beq
\label{21giugno2023-3}
{\rm Re}(\lambda_{ij_o} z^{\sigma_{ij_o}})-{\rm Re}(\lambda_{ij_o} w^{\sigma_{ij_o}})>0, \quad z>w,
\eeq
if and only if $f_{ij_o}$, defined as in \eqref{26maggio2023-4} by the relation
$$
f_{ij_o}(\theta):= \frac{\cos(\sigma_{ij_o}\theta+\eta_{ij_o})}{\sin^{\sigma_{ij_o}}(\theta-\tau)},\qquad \eta_{ij_o}:=\arg\la_{ij_o}\in]\eta-2\pi,\eta[,
$$
  is monotonically decreasing for $\tau<\theta<\tau+\pi$. This holds  because \eqref{21giugno2023-3} reduces to  $ 
 f_{ij_o}(\theta)-f_{ij_o}(\widetilde{\theta})>0$, with $ \theta<\widetilde{\theta}$ along  $\iota_{b,\phi}$, with  $\theta=\arg z$, $\widetilde{\theta}=\arg w$ and $w<z$. 
   Monotonic decrease of $f_{ij_o}(\theta)$ occurs  if and only if the cosine in the numerator of the analog of \eqref{22giugno2023-2} for $df_{ij_o}/d\theta$ is positive. In other words,  
 $k\pi$ is replaced by $(2k+1)\pi$ in \eqref{26maggio2023-7}, namely if and only if there exists $k\in\Z$ such that  
\beq
\label{22giugno2023-6}
  \tau<\tau_{ij_o}+\frac{(2k+1)\pi}{\sigma_{ij_o}}<\tau +\pi,\qquad \tau_{ij_o}:=
  \frac{1}{\sigma_{ij_o}}\left(\dfrac{3\pi}{2}-\eta_{ij_o}\right).
\eeq

Using as before  $ \tau=\phi-\pi/2$, if $\phi-\frac{3\pi}{2}<\arg z<\phi-\frac{\pi}{2}$ in  \eqref{22giugno2023-1}, we get that 
$b=|z|\cos(\phi-\arg z-\pi)$ along a line $\iota_{b,\phi}$. Hence, 
$$
\on{Re}(\lambda z^\sigma)=|\lambda||b|^\sigma \frac{\cos(\sigma\theta+\arg\lambda)}{\sin^\sigma(\tau-\theta)},\quad \quad  \tau-\pi<\theta=\arg z<\tau.
$$ 
The analogs of Lemmas \ref{26maggio2023-9} and \ref{26maggio2023-12} are proved by replacing \eqref{26maggio2023-4}  and \eqref{22giugno2023-2} respectively with 
\begin{align}
& g_*(\theta):= \frac{\cos(\sigma^{(*)}\theta+\eta_*)}{\sin^{\sigma^{(*)}}(\tau-\theta)},
\\
\label{22giugno2023-4}
& \frac{dg_*}{d\theta}
=
\sigma^{(*)}\frac{\cos((1-\sigma^{(*)})\theta^\prime+\sigma^{(*)}\tau+\eta_*)}{\sin^{\sigma^{(*)}+1}\theta^\prime},\quad \theta^\prime:=\tau-\theta,\quad 0<\theta^\prime<\pi.
\end{align}
Again, \eqref{21giugno2023-3} reduces to $ 
 g_{ij_o}(\theta)-g_{ij_o}(\widetilde{\theta})>0$, with $ \theta>\widetilde{\theta}$, where 
 $$
g_{ij_o}(\theta):= \frac{\cos(\sigma_{ij_o}\theta+\eta_{ij_o})}{\sin^{\sigma_{ij_o}}(\tau-\theta)}.
$$
Hence, $ g_{ij_o}$ must be monotonically increasing w.r.t. $\tau-\pi<\theta<\tau$. Namely, the cosine in the numerator  of $\dfrac{dg_{ij_o}}{d\theta}$ must be positive. The numerator of \eqref{22giugno2023-4} is positive if and only if for some $k\in\Z$,
$$ 
\frac{\pi}{2}+(2k+1)\pi<\sigma^{(*)}\tau+\eta^{(*)}<(1-\sigma^{(*)})\theta^\prime+\sigma^{(*)}\tau+\eta^{(*)}<(1-\sigma^{(*)})\pi+\sigma^{(*)}\tau+\eta^{(*)}<\frac{3\pi}{2}+(2k+1)\pi,
$$
that is 
$
\tau-\pi<\tau^{(*)}+\dfrac{2k\pi}{\sigma^{(*)}}<\tau
$. 
Hence,  $ g_{ij_o}$ is monotonically increasing in the interval $]\tau-\pi,\tau[$  if and only if there exists $k$ such that  
\beq
\label{22giugno2023-7}
\tau-\pi<\tau_{ij_o}+\frac{2k\pi}{\sigma_{ij_o}}<\tau.
\eeq

The last case is $\arg z=\tau$ ($=\phi-\pi/2$) in \eqref{22giugno2023-1}. In this case ${\rm Re}(\lambda z^\sigma)=|\lambda||z|^\sigma \cos(\sigma\tau+\arg \lambda)$. Then \eqref{21giugno2023-3} is 
$$ 
|\lambda_{ij_o}| \Bigl(|z|^{\sigma_{ij_o}}-|w|^{\sigma_{ij_o}}\Bigr)\cos(\sigma_{ij_o}\tau+\eta_{ij_o})>0, \quad |z|>|w|.
$$
The above is true if and only if for some $k\in \Z$,
\beq
\label{23giugno2023-1}
\tau_{ij_o}+\frac{2k\pi}{\sigma_{ijo}}<\tau<\tau_{ij_o}+\frac{(2k+1)\pi}{\sigma_{ijo}}.
\eeq
Now, apply \eqref{22giugno2023-6}, \eqref{22giugno2023-7} and \eqref{23giugno2023-1} to all pairs $(i,j_o)$.
  \end{proof}

  \begin{lem}
  \label{21giugno2023-6}
  If $\Lambda$ is generic as in  Definition \ref{19giugno2023-1}, then  $j_o$ is a subdominant index if and only if $\tau_{\rho_{\mu_o}-1}-\tau_{\rho_0}<\pi$.   In this case, $j_o$ is subdominant on the intervals
  $$ 
  I_{k}:=\Bigl\{\phi=\tau+\pi/2~\Bigr|~\tau_{\rho_{\mu_o-1}} <\tau-2k\pi< \tau_{\rho_0}+\pi\Bigr\}, \quad k\in \Z.
  $$ 
 
  \end{lem}
  
 \begin{proof}
The l.h.s. of \eqref{21giugno2023-5} equals %
  \[\bigcap_{\gm=0}^{\mu_o-1}\left] \tau_{\rho_\gamma}+2k_\gamma\pi;\, \tau_{\rho_\gamma}+(2k_\gamma+1)\pi\right[.
  \]
  This set is non-empty if and only if $k_0=k_1=\dots=k_{\mu_o-1}=k\in \Z$, and the resulting intersection equals
  \[
  \bigcap_{\gm=0}^{\mu_o-1}\left] \tau_{\rho_\gamma}+2k\pi;\, \tau_{\rho_\gamma}+(2k+1)\pi\right[=\left]\tau_{\rho_{\mu_o}-1}+2k\pi;\,\tau_{\rho_0}+(2k+1)\pi\right[.
  \]
  \end{proof}

  \begin{proof}[Proof of Theorem \ref{22giugno2023-5}]
  The conditions on the Stokes rays are given in Lemma \ref{21giugno2023-6}. The existence and uniqueness of   vector solutions  follow from Theorem \ref{subth}, the sectors follow from Lemma \ref{21giugno2023-6}. 
  \end{proof}

\subsection{Parametric case}
\label{18settembre2023-6}
We consider equation \eqref{diffeq} with $\Lambda$ as in \eqref{2giugno2023-1},  depending on the parameter
$$\omega:=(\boldsymbol{\lambda}_1,\dots,\boldsymbol{\lambda}_\ell)\in \Omega\subset \mathbb{C}^m,\quad m= \ell \times(h+1),
$$
where  $\boldsymbol{\lambda}_i$ is defined in \eqref{18settembre2023-4}.   
In order to apply Theorem \ref{mainth-bis}, we assume that:
\begin{itemize}
\item[a)] the open connected domain  $\Omega $ is  {\it bounded};
\item[b)]   $\boldsymbol{\lambda}_i\neq \boldsymbol{\lambda}_j$ for any $i\neq j$, and moreover $\lambda_{ij}\neq 0$  on the closure $\overline{\Omega}$, see \eqref{19settembre2023-1};
\item[c)]  
  $\sigma_0,\dots,\sigma_h$ are {\it independent} of the parameters.
  \item[d)] $R=R(z,\omega)$ satisfies Definition  \ref{5giugno2023-1-bis}  on any open interval $I$.
\end{itemize}
Under assumption a),  the uniform integrability condition of Definition \ref{9settembre2023-6}  is  satisfied.\footnote{ \, For any $i$, take $g_i(z)= C \sum_{k=0}^h |z|^{\sigma_k-1}$, where $ C:=\max_{\overline{\Omega}}(\max_{i=1,...,\ell} \sum_k \lambda_i^{(k)})$.}  
In b), the condition  $\lambda_{ij}\neq 0$, for all $i\neq j$,  allows to use Definition \ref{9giugno2023-2} of  the Stokes  rays.  By \eqref{26maggio2023-2}, this is equivalent to the requirement that {\it the  degrees of $q_i(z,\omega)-q_j(z,\omega)$  are constant on $\Omega$}. For example, if $\Lambda$ is generic on $\overline{\Omega}$,  this  is satisfied. 
Notice that $\arg\lambda_{ij}$ depends continuously on $\omega$. In order to label  the Stokes rays we also assume that:
\begin{itemize}
\item[e)] 
 the numeration \eqref{9giugno2023-4} is preserved on $\overline{\Omega}$, for a {\it fixed $\eta$}. 
\end{itemize}
The directions of Stokes rays continuously depend on $\omega$, and e) above implies that  if $\tau_\rho<\tau_\sigma$ (or $\tau_\rho=\tau_\sigma$) at $\omega_o\in \Omega$, then $\tau_\rho(\omega)<\tau_\sigma(\omega)$  (or $\tau_\rho(\omega)=\tau_\sigma(\omega)$) for every $\omega\in \overline{\Omega}$.  
Notice that condition e) is always satisfied if $\Om$ is sufficiently small.
\vskip1,5mm

Definition \ref{18settembre2023-3} is substituted by the following
\begin{defn}
\label{18settembre2023-3-bis}
 A vector $\bm k\in \Z^\mu$ is called {\it adequate with respect to $(\si^{(\rho)};\tau_\rho)_{\rho=0}^{\mu-1}$ on $\Omega$} if the open interval 
$$
\Si_{\bm k}:=\bigcap_{\omega\in\overline{\Omega}}\left(\bigcap_{\rho=0}^{\mu-1}\,
\left]
\tau_\rho(\omega)+\frac{k_\rho \pi}{\si^{(\rho)}}-\pi;
\,
\tau_\rho(\omega)+\frac{k_\rho \pi}{\si^{(\rho)}}
\right[
\right)
$$ 
is non-empty.
\end{defn}

The uniform $L$-condition of Definition \ref{9settembre2023-1-bis} is satisfied if and only if there is an adequate  $\boldsymbol{k}$. 
Under the above assumptions,
Theorem \ref{mainth-bis} implies the following  generalizations of 
Theorem \ref{5giugno2023-3}, Corollary \ref{7giugno2023-2} and Theorem \ref{22giugno2023-5}.

\begin{thm}
\label{23ottobre2023-1}
Consider the differential equation \eqref{diffeq-bis}
with  $\Lambda(z,\omega)$ as in \eqref{2giugno2023-1} satisfying a), b), c) and e) above,  and a perturbation term $R\in \hol(\Tilde{\C}_{a}\times \Omega;\gl(n,\C))$ with  uniform good decay conditions  on $\Tilde{\C}_{a}$ (Definition  \ref{5giugno2023-1-bis}). 
Let $\bm k\in\Z^\mu$ be adequate with respect to $(\si^{(\rho)};\tau_\rho)_{\rho=0}^{\mu-1}$ on $\Omega$.
Then, there exists a
 unique holomorphic fundamental matrix  solution $Y_{\boldsymbol{k}}\colon \H_{I,a}\times\Omega\to GL(n,\C)$, where $I=\Si_{\bm k}+\frac{\pi}{2}$,  with asymptotic behaviour  
\beq
\label{9giugno2023-1-bis}
Y_{\boldsymbol{k}}(z,\omega)=\left(I_n+o(1)\right) \exp\bigl\{Q(z,\omega)\bigr\},\quad \hbox{\it uniformly in $\Omega$}, 
\eeq 
for  $z\to \infty$  in any closed subsector of 
$$ 
S_{\boldsymbol{k}}:=\{z\in \Tilde{\C^*}~|~a_{\boldsymbol{k}}<\arg z <b_{\boldsymbol{k}}+\pi \}, \quad \text{where }\Si_{\bm k}=\left]a_{\bm k};b_{\bm k}\right[.
$$
\end{thm}

\begin{cor}
\label{23ottobre2023-3}
Under  the assumptions of Theorem \ref{23ottobre2023-1}, if $\Lambda$ is generic, then for every $\nu\in \mathbb{Z}$,  there exists a unique  holomorphic fundamental matrix  solution $Y_\nu: S_\nu\times \Omega \to GL(n,\C)$, where
$$ 
S_\nu:=\bigcap_{\omega\in \overline{\Omega}}\{z\in \Tilde{\C^*}~|~\tau_{\nu-1}(\omega)<\arg z <\tau_\nu(\omega) +\pi\},$$ 
with asymptotic   behaviour
 $$
Y_\nu(z)=\left(I_n+o(1)\right) \exp\bigl\{Q(z)\bigr\},\quad \hbox{\it uniformly in $\Omega$}, 
$$
  for $z\to \infty$ in any closed subsector of 
$ 
S_\nu$.  

\end{cor}

\begin{cor}
\label{25ottobre2023-1} Assume that $\La$ is generic, and fix $j_o\in \{1,\dots,\ell\}$.
In the same notations of Theorem 
\ref{22giugno2023-5}, if $\underset{\omega\in \overline{\Omega}}{\sup}(\tau_{\rho_{\mu_o-1}}-\tau_{\rho_0})<\pi$, then for any $k\in \Z$, there exist $s_{j_o}$ %
vector solutions 
$$y^{(k)}_{\alpha}\in\mathscr O(\Tilde{\C}_a\times\Omega;\C^n), \quad\quad
\alpha=s_1+\dots+s_{j_o-1}+m,\quad m=1,\dots,s_{j_o},
$$
  characterized by the uniform asymptotic behaviour 
  $
y^{(k)}_\alpha(z,\omega)=  (e_\alpha+o(1))\exp\left(\sum_{q=1}^h \lambda_{j_o}^{(q)}z^{\sigma_q}\right)$,
  for $z\to \infty$ in any closed subsector of   $$ 
{S}_k^{(j_o)}:= \bigcap_{\omega\in \overline{\Omega}}\Bigl\{z\in \Tilde{\C}_a~\Bigl|~     \tau_{\rho_{\mu_o-1}}-\pi<\arg z-2k\pi< \tau_{\rho_0}+2\pi \Bigr\}.
  $$
\end{cor}

\subsection{Applications to the ODE/IM correspondence}\label{ODEIM} The following result is a special case of the results developed in Sections \ref{secstok}--\ref{18settembre2023-6}.

\begin{thm}\label{thm:matrixodeim}
For $M>0$, let
$\widetilde{\mathbb{C}}_M=
 \lbrace z \in \widetilde{\mathbb{C}^*} , |z| > M \rbrace.$
 Consider the differential equation
 \begin{equation}\label{eq:irregularnormal}
  \frac{d\psi(z)}{dz}= \Bigl(- A \, p(z;\omega) + R(z;\omega) \Bigr) \psi,
   \end{equation}
 where
  \begin{itemize}
  \item  $A \in \gl(n,\C)$ is a constant diagonalisable matrix with eigenvectors $\psi_j$
 and eigenvalues $\nu_j$, $j=1,\dots, n$, which are not necessarily pairwise distinct.
 \item 
$p(z;\omega)=  z^{\beta_0}+\sum_{k=1}^{H} c_k(\omega)  z^{\beta_k}$,
where the exponents $\beta_k$'s are real and  ordered so that
$$
\beta_0>\beta_1> \cdots>\beta_H= -1.
$$ 
 Moreover,
the coefficients $c_k(\omega)$ are  bounded  analytic functions of the
parameter $\omega$, which belongs to a domain $\Omega \subset \mathbb{C}$.
Define the primitive
 \begin{equation}\label{eq:Qzla}
  P(z;\omega):= \frac{z^{\beta_0+1}}{\beta_0+1} +\sum_{k=1}^{ { H-1}} c_k(\omega)
\frac{z^{\beta_k+1}}{\beta_k+1} { + c_H(\omega)\ln z}  .
 \end{equation}

 \item  $R\in\hol(\widetilde{\mathbb{C}}_M\times \Omega;~\gl(n,\C))$, with behaviour 
$|R(z;\omega)|=O(z^{-1-\delta})$  at infinity, uniformly with respect to $\omega \in \Omega$, in any arbitrary
closed sector of $\widetilde{\mathbb{C}}_M$.
\end{itemize}

\vskip 0.15 cm
\noindent
\begin{enumerate}
\item  Assume that for an interval $[a,b]$  the following condition holds:
 \begin{equation}\label{eq:nolines}
  \forall \tau \in [a,b], 
  \quad
  {\rm Re}\Bigl( (\nu_j -\nu_{j'} )e^{\sqrt{-1}\tau }\Bigr)=0 
  \mbox{ if and only if } \nu_j=\nu_{j'}.
 \end{equation}
Then, there exists a unique basis of solutions $\psi_j(z;\omega)$, analytic in
$\widetilde{\mathbb{C}}_M \times \Omega$, with the  asymptotic behaviour
\begin{align} \nonumber
 \psi_j(z,\omega) & = \left(\psi_j + O(z^{-\delta}) \right) e^{-\nu_j
P(z;\omega)},\\ \label{eq:basesinfinity}
 & \mbox{when } z \to \infty \mbox{ in the sector }
\frac{a}{\beta_0+1}\leq \arg z \leq\frac{b+\pi}{\beta_0+1},
\end{align}
uniformly with respect to $\omega \in \Omega$.

\vskip 0.15 cm
\noindent
\item Assume that $j_0$ is a subdominant index for an interval $[a,b]$, namely
\begin{equation}\label{eq:nosubdominantlines}
  \forall \tau  \in [a,b] \mbox{ and } j \neq j_0, \quad
  {\rm Re}\Bigl( (\nu_{j_0} -\nu_j)e^{\sqrt{-1}\tau }\Bigr) > 0 .
 \end{equation}
Then, there exists a unique solution $\Psi(z;\omega)$, called subdominant, such that
\begin{align}\nonumber
 \Psi(z;\omega) & =   \left(\psi_{j_0} + O(z^{-\delta}) \right)
e^{-\nu_{j_0} P(z;\omega)}, \\ \label{eq:subdominantgeneral}
 & \mbox{as } z \to \infty, \mbox{ in the sector }
 \frac{a-\pi}{\beta_0+1} \leq \arg z
 \leq \frac{b + \pi}{\beta_0+1},
\end{align}
uniformly with respect to $\omega \in \Omega$. Moreover, the function
$\Psi(\cdot,\cdot)$ is
analytic in $\widetilde{\mathbb{C}}_M \times \Omega$.
\end{enumerate}
\end{thm}

\begin{proof}

The proof follows from two observations. 
The first one is that  equation \eqref{eq:irregularnormal} is a particular case of the more general equation
$$
\frac{dY}{dx}=\Bigl(\sum_{k=0}^H \widehat{\Lambda}^{(k)}(\omega) x^{\beta_k} +\widehat{R}(x,\omega)\Bigr)Y, 
$$
where
$ \widehat{\Lambda}^{(k)} \in \mathcal{O}(\Omega,\h(n,\C))$, with exponents $\beta_0>\beta_1> \cdots>\beta_H= -1$,  
and $R=O(x^{-1-\widehat{\delta}})$ uniformly in $\omega$,  with $\widehat{\delta}>0$.  By the change of variables $z=x^{\beta_0+1}$, the above is 
  transformed to an equation  \eqref{diffeq-bis}. In partcular,  $\Lambda$ has form \eqref{2giugno2023-1} (with $h=H-1$),  but for an additional
 term with power $ z^{-1}$. More  explicitly, we get
$$
  \Lambda=\sum_{k=0}^{H-1} \Lambda^{(k)} z^{\sigma_k -1}+\frac{\Lambda^{(H)}}{z},
  $$
where
$$
\begin{aligned} 
& \sigma_k:= \frac{\beta_k+1}{\beta_0+1} \hbox{ (for $1\leq k\leq H-1$)},
            \quad \quad
              1=\sigma_0>\sigma_1>\dots>\sigma_{H-1}>0,
\\
\noalign{\medskip}
&  \Lambda^{(k)}:= \frac{\widehat{\Lambda}^{(k)}}{\beta_0+1},\quad 1\leq k\leq H,
\quad \quad 
R(z,\omega):= \frac{\widehat{R}(z^{1/(\beta_0+1)},\omega)}{z^{\beta_0/(\beta_0+1)}}.
\end{aligned}
$$
Notice that $R=O(|z|^{-1-\delta})$, with $\delta= \widehat{\delta}/(\beta_0+1)$. 
 The theory developed above for equation \eqref{2giugno2023-1}, in 
 the parametric case,   immediately applies to the above case, provided that $k_{ij}$ defined 
 in \eqref{19settembre2023-1} satisfies\footnote{Here $\Lambda^{(H)}= \lambda_1^{(H)}I_{s_1}\oplus\cdots\oplus  \lambda_\ell^{(H)}I_{s_\ell}$, and the definition of $k_{ij}$ is obviously extended.}
\beq
\label{23ottobre2023-4}
k_{ij} \leq h.
\eeq
 In this case Theorem \ref{23ottobre2023-1} applies, with 
$$
Q(z,\omega)=q_1(z,\omega)I_{s_1}\oplus\dots\oplus q_\ell(z,\omega)I_{s_\ell}+ \Lambda^{(h+1)}\ln z
,
\quad 
\quad
q_i(z,\omega):=\sum_{k=0}^h \lambda_i^{(k)}z^{\sigma_k}.
$$
  and Corollary \ref{23ottobre2023-3} holds if  $\Lambda$ is generic.   Notice that \eqref{23ottobre2023-4} holds  for generic  $\Lambda$.

The second observation is that Corollary \ref{23ottobre2023-3} also holds for a generic $\Lambda$ with  the particular structure
$$ 
\Lambda_0 \hbox{ constant},\quad  \Lambda_k=c_k(\omega) \Lambda_0,\quad 1\leq k \leq h, 
$$
where $c_1,\dots,c_h$ are bounded holomorphic functions on a domain $\Omega$.  Indeed, in this case $\lambda_{ij}=\lambda_i^{(0)}-\lambda_j^{(0)}\neq 0$ are constant and the numeration \eqref{9giugno2023-4} is preserved. 

The two observations above apply in case of equation \eqref{eq:irregularnormal}. As a consequence, Corollary \ref{23ottobre2023-3} holds and proves point (1). Point (2) follows from Corollary \ref{25ottobre2023-1}.
\end{proof}

We apply the above theorem to the most studied instance of the ODE/IM correspondence, the duality between
quantum $\mathfrak{g}^{(1)}$-Drinfeld--Sokolov hierarchy \cite{BLZ2,FF0} -- also known
as Quantum $\mathfrak{g}$-KdV -- and ${}^L\mathfrak{g}^{(1)}$-opers on $\C^*$, where
 $\mathfrak{g}^{(1)}$ is the untwisted affinisation of a simple Lie algebra $\mathfrak{g}$ and ${}^L\mathfrak{g}^{(1)}$ its Langlands dual algebra.
Here for sake of simplicity, we restrict to the case of simply-laced $\mathfrak{g}$, so that
${}^L\mathfrak{g}^{(1)}=\mathfrak{g}^{(1)}$.
\vskip1,5mm
To provide some further context, we recall that, on one side of the correspondence, Quantum $\mathfrak{g}$-KdV is a family of theories, depending on $\rank \mathfrak{g}$ complex moduli, which parameterise the boundary conditions of the theory, and one real modulus, the ``central charge''.
On the other side, $\mathfrak{g}^{(1)}$-\textit{opers} are gauge-equivalence classes of linear $\mathfrak{g}^{(1)}$-connections (on a smooth algebraic curve), which generalize the notion of \textit{scalar linear differential operator}, see \cite{BeDr05}.
The opers that we consider are not meromorphic because, the exponent of one of their coefficients (the parameter $k$ in \eqref{eq:ffoper} below) is identified with the central charge. Hence, it takes real values.

\subsubsection*{Lie algebra preliminaries}
Here we collect well-known facts about simple Lie algebras that we need in the sequel. We refer to \cite{K} for all the details.
\vskip1,5mm
Let $\left(\mathfrak{g},[-,-]\right)$ be a simply-laced simple Lie algebra of rank $m$ over $\C$.
We have a 
Cartan decomposition
\begin{equation}
 \mathfrak{g}= \mathfrak{n}_- \oplus \mathfrak{h} \oplus \mathfrak{n}_+ ,
\end{equation}
where $\mathfrak{h}$ is a Cartan (i.e.\,\,a maximal commutative) subalgebra and $\mathfrak{n}_{\pm}$ are maximal positive/negative nilpotent subalgebras.
The algebra $\frak g$ has Chevalley generators
$\lbrace e_i,f_i,h_i \rbrace_{i=1,\dots,m}$, with $e_i\in \mathfrak{n}_{+}$, $
f_i \in \mathfrak{n}_{-}$, $h_i \in \mathfrak{h}$.
\vskip1,5mm
We denote by $\rho^\vee \in \mathfrak{h}$ the dual Weyl vector,
which is uniquely defined by the relations
\begin{equation}
[\rho^\vee,f_i]=-f_i,\quad i=1,\dots,m.
\end{equation}
The element $\rho^\vee$ induces the principal grading
\begin{equation}
 \mathfrak{g}= \bigoplus_{l=-h+1}^{h-1}\mathfrak{g}^{(l)},\qquad \mathfrak{g}^{(l)}= \lbrace g \in \mathfrak{g}, [\rho^\vee,g]=l \, g \rbrace,
\end{equation}
where $h$ is the Coxeter number of the algebra (an integer greater or equal than $2$).
\vskip1,5mm
We have $\mathfrak{h}=\mathfrak{g}^{(0)}$ and $\mathfrak{n}_{\pm}=\oplus_{l=1}^{h-1}\mathfrak{g}^{(\pm l)}$.
The top-graded subspace $\mathfrak{g}^{(h-1)}$ is a one-dimensional subspace and it is spanned
by the highest weight vector $e_{\theta}$ of the Lie algebra, which is the unique (up to a scalar multiple) element annihilated by all $e_i$'s, that is
\begin{equation}
 [e_i,e_{\theta}]=0, \quad i=1,\dots,m.
\end{equation}

\subsubsection*{Feigin-Frenkel Connections}
According to \cite{FF,MRV,MR1}, the family of connections corresponding to the $\mathfrak{g}^{(1)}$-Drinfeld--Sokolov hierarchy -- sometimes called F-F (Feigin--Frenkel) connections --  %
are defined by the following differential operators
\begin{equation}\label{eq:ffoper}
 \mathcal{L}=\partial_z+ \frac{f+\ell}{z}+ \left(1+ \omega \, z^{-{k}} \right) e_{\theta}+ \sum_{j \in J}
 \frac{-\theta^{\vee} + X(j)}{z-w_j},\qquad f=\sum_{i=1}^m f_i,
\end{equation}
where $J$ is a possibly empty finite set, $z$ is a global coordinate on $\widetilde{\C}_M$ with $M > \max \lbrace |w_j|, j \in J \rbrace$, and
the parameter $\omega$ takes value in  $\C$. The differential operator $\mc L$ acts on sections of the trivial bundle $\Tilde{\C}_M\times V$, with $V$ any finite dimensional $\mathfrak{g}$-module.
\vskip1,5mm
The parameters $\ell \in \mathfrak{h}$ and $0<k<1$ are free parameters, corresponding to the moduli
of Quantum $\mathfrak{g}$-KdV theories. In particular, $k$ is identified with the central charge.
On the contrary, the parameters  $\theta^\vee \in \mathfrak{h}$, $w_j \in \mathbb{C}^*$ and $X(j)$ are constrained by the requirement that the monodromy at $w_j$'s is trivial
\footnote{For example, $\theta^\vee$ is the co-root dual to the highest root of the algebra.}. However,
for the sake of our discussion they can be considered as additional free parameters since the asymptotic
properties of solutions are independent of their values.

\subsubsection*{Reduction to the normal form}

We notice that the dominant term of  the F-F connection  at $z=\infty$ is the nilpotent element $e_{\theta} \in \mathfrak{n}_+$. Therefore, in order to apply Theorem \ref{thm:matrixodeim}, we need to transform $\mc{L}$
 into an equivalent connection with a semi-simple dominant term.
 We achieve this by means of the gauge transform
\begin{equation}\label{eq:Gp}
 G=\exp\left( p(z;\omega)^{-1} \tilde{n}\right) \circ \big(p(z;\omega)\big)^{- \rho^\vee},
\end{equation}
where
\begin{equation}\label{eq:pzla}
 p(z;\omega)= z^{\frac{1}{h}} \left( 1 + \sum_{l=1}^{\lfloor  \frac{1}{(1-k)h} \rfloor}
 c_l \, \omega^l \, z^{l(1-k)}\right) , \quad c_l = \frac{1}{l!} \left(\frac{\partial^l}{\partial w^l} (1-w)^{\frac{1}{h}}\right)_{|w=0},
\end{equation}
and
$\tilde{n} $ is the unique element in $\C \lbrace e_1,\dots,e_m \rbrace$ such that
$[\tilde{n},f]=\ell+ \frac{\rho^\vee}{h^\vee} \rho^\vee- |J| \theta^\vee$.
The action of the gauge $G$ on the F-F connection $\mathcal{L}$ can be computed using the following rules, see
\cite[Section 2]{MRV}:
\begin{itemize}
 \item If $g \in \mathfrak{g}^{(l)}$ and $q$ is a meromorphic function, not identically zero, then 
 \begin{align}\label{eq:Gauge1}
   & (q(z))^{\rho^\vee}  \, \partial_z= \partial_z-\frac{q'(z)}{q(z)} \rho^\vee, \quad  (q(z))^{\rho^\vee} \, g = (q(z))^{l} g,
   \end{align}
   \item If $g\in \mathfrak{g}, n \in \mathfrak{n}_+$ and $p$ is a meromorphic function, then 
\begin{align}\label{eq:Gauge2}
& \exp\lbrace p(z) n\rbrace \, \partial_z= \partial_z-
 p'(z) \, n , \quad \exp\lbrace b(z) n\rbrace \, g= \sum_{j\geq 0} \frac{b^j(z)}{j!}
 (\,\mbox{ad}_n)^j g,
\end{align}
where $\mbox{ad}_n g=[n,g]$; we stress that the latter series terminates,  since $n$ is nilpotent.
\end{itemize}
Taking into account \eqref{eq:Gauge1}-\eqref{eq:Gauge2},
we deduce that there exists  $\delta>0$ such that
\begin{align}\label{eq:GL}
&\mathcal{L'}:=G. \mathcal{L}= \partial_z+ z^{-1}\, p(z;\omega) \Lambda+ R(z;\omega), \quad \Lambda= f+ e_{\theta}.
\end{align}
where $R(z;\omega)$ is analytic in $\widetilde{\C_M} \times \C$ and $R(z,\omega)=O(z^{-1-\delta})$, uniformly with respect to $\omega$ %
in any compact subset of $\C$; see \cite{MRV} for details.
\vskip1,5mm
The element $\Lambda= f+ e_{\theta}$ is called the cyclic element of $\mathfrak{g}$, see \cite{K}.
It is well-known that $\Lambda$ is a regular semi-simple element of the Lie algebra and it is
therefore diagonalisable in every finite dimensional
$\mathfrak{g}$-module $V$ \footnote{While the exact nature of the spectrum of $\Lambda$, which was computed in \cite{MRV,MRV2}, is important for the ODE/IM correspondence,
for what concerns the mere existence of distinguished solutions at $z=\infty$ it is sufficient
to know that $\Lambda$ is semi-simple.}.

The following theorem is a direct corollary of
Theorem \ref{thm:matrixodeim} and  formula \eqref{eq:GL}.
\begin{thm}\label{cor:odeim}
Given a F-F connection $\mathcal{L}$ as in \eqref{eq:ffoper} and a
finite dimensional $\mathfrak{g}$-module $V$ of dimension $n$,
we consider the linear ordinary differential equation
\begin{equation}
 \mc{L}\psi=0,\quad \psi: \widetilde{\mathbb{C}}_M \to V.
\end{equation}
We denote by $\psi_1,\dots,\psi_n$ an eigenbasis,
for $\Lambda$, by $\nu_1, \dots \nu_n$ the corresponding eigenvalues, and we let
 $P$ be the primitive of $z^{-1}p(z,\omega)$, as defined in \eqref{eq:pzla}, given by
\begin{equation}
\label{eq:PAction}
P(z;\omega)=
\left\{
\begin{aligned}
 & h \, z^{\frac{1}{h}} \left( 1 + \sum_{l=1}^{\lfloor  \frac{1}{(1-k)h} \rfloor}
 \frac{c_l \omega^l z^{-l(k-1)} }{1- \frac{l}{h (1-k)}} \right) , \quad  \frac{1}{h(1-k)} \notin \mathbb{N},
 \\
 \noalign{\medskip}
 & h \, z^{\frac{1}{h}} \left( 1 + \sum_{l=1}^{q-1}
 \frac{c_l\omega^l z^{-l(k-1)}}{1- l\,q}  \right) +
  c_{q}\log z
  , \quad   q:=\frac{1}{h(1-k)} \in \mathbb{N}.
\end{aligned} \right.
\end{equation}

\begin{enumerate}
 \item Assume that an interval $[a,b]$ is such that the following condition holds:
 \begin{equation*}
  \forall \tau  \in [a,b], \quad
  {\rm Re}\Bigl( (\nu_j-\nu_{j'}) e^{\sqrt{-1}\tau }\Bigr)=0 \mbox{ if and
only if } \nu_j=\nu_{j'}.
 \end{equation*}
Then, there exist  $\delta>0$ and a unique basis of solutions $\psi_j(z;\omega), j=1,\dots,n$, analytic in
$\widetilde{\mathbb{C}}_M \times \C$, with the  asymptotic behaviour
\begin{equation}
\label{eq:psibasis}
\begin{aligned}
  \psi_{j}(z;\omega)=  &z^{\frac{1}{h}\rho^\vee} \left( \psi_{j} + O\big(z^{-\delta}\big) \right)
  e^{-\nu_{j} P(z;\omega)},
  \\
  & 
  \mbox{ as } z \to  \infty, \;  h a \leq \arg z \leq h (b+\pi).
  \end{aligned}
 \end{equation}
Moreover, the above estimate holds uniformly with respect to $\omega$, as $\omega$ varies in any compact of $\C$.

\vskip 0.15 cm
\noindent
\item Assume that $j_0$ is a subdominant index for an interval $[a,b]$, namely
\begin{equation*}
  \forall \tau  \in [a,b] \mbox{ and } j \neq j_0, \quad
  {\rm Re}\Bigl( (\nu_{j_0}-\nu_j ) e^{\sqrt{-1}\tau }\Bigr) >0 .
 \end{equation*}
Then, there exist $\delta>0$ and a unique (subdominant) solution $\psi_{j_0}(z;\omega)$, analytic in 
$\widetilde{\mathbb{C}}_M \times \C$,
such that
\begin{equation}
\label{eq:psisub}
\begin{aligned}  \Psi(z;\omega)=&  z^{\frac{1}{h}\rho^\vee} \left( \psi_{j_0} + O\big(z^{-\delta}\big) \right) e^{-\mu_{j_0} P(z;\omega)},
\\
&\mbox{ as } z \to  \infty, \; h (a-\pi)\leq \arg z \leq h (b+\pi).
\end{aligned}
 \end{equation}
 Moreover, the above estimate holds uniformly with respect to $\omega$, as $\omega$ varies in any compact of $\C$.
 \end{enumerate}
\end{thm}
 \begin{proof}
 After \eqref{eq:GL}, we have that
 \begin{equation}\label{eq:Lppsi}
  \mathcal{L}' \phi(z)=0, \quad \phi:= G \psi,
 \end{equation}
where $\mathcal{L}'$ and $G$ are as in \eqref{eq:GL} and \eqref{eq:Gp}.
We notice that
\begin{equation}\label{eq:Gasymp}
 G = z^{-\frac{\rho^\vee}{h}} \left( I_n + O(z^{-\delta'})\right),\quad \mbox{ for some } \delta'>0,
\end{equation}
where $I_n$ is the identity matrix.
In fact, on the one side $\big(p(z;\omega)\big)^{- \rho^\vee} = z^{-\frac{\rho^\vee}{h}} \left( I_n + O(z^{-1+k})\right)$; on the other side,
since $\tilde{n}$ is nilpotent,
$\exp\left( p(z;\omega)^{-1} \tilde{n}\right)$ is a finite sum, whence
$\exp\left( p(z;\omega)^{-1} \tilde{n}\right)=( I_n +O (z^{-\frac{1}{h}}))$.
Applying Theorem \ref{thm:matrixodeim} to \eqref{eq:Lppsi} and using \eqref{eq:Gasymp}, we obtain the required result.
 \end{proof}

\begin{rem}
While  item (1) of Theorem \ref{cor:odeim} is completely new,
a weaker version of item (2) 
is already available, see \cite[Theorem 3.4]{MRV}. In {\it loc.\,cit.},
it is shown that if $\tau _0$ is a subdominant direction for the index $j_0$, the subdominant
solution exists with asymptotic \eqref{eq:psisub} is the small sector
 $|\arg z-\tau _0 |\leq \frac{h \pi}{2}$.
\end{rem}

\begin{rem}
For the sake of definiteness, we have concentrated here on the study of F-F connections for $\mathfrak{g}$ simply-laced.
However, Theorem \ref{thm:matrixodeim} can be applied to study the case $\mathfrak{g}$ not-simply-laced (in which case
they are called F-F-H connections, see \cite{FH,MRV2,MR2}) as well as other families of opers that have been proposed
in the ODE/IM literature, as for example in \cite{LZ,GLVW,MMR}.
\end{rem}

\section{The  case of  ODEs with meromorphic coefficients}
\label{16giugno2023-1}
\subsection{Preliminaries on formal solutions}\label{secforsol}
We consider the classical asymptotic problem for the differential equation
\beq
\label{30maggio2023-1}
\frac{dY}{dz}=z^{r-1}A(z) Y,\quad \quad A\in \hol({\C}_{a};\gl(n,\C)),
\eeq
where $r\in\N$ is  called {\it Poincar\'e rank} at infinity  and $A$ is analytic at infinity.  Theorem \ref{mainth}  can be applied  to obtain the existence and uniqueness result of fundamental matrix solutions, with asymptotic behaviour in a ``wide'' sector. The uniqueness issue and the characterization of the sector are the main achievements  of our approach. 
\vskip1,5mm
Let $\C[\![z^{-1}]\!]$ be the algebra of formal power series in $z^{-1}$. It is well known \cite[Th. I]{BJL79-2} that equation \eqref{30maggio2023-1} admits {\it formal} fundamental matrix solutions 
\beq
\label{19giugno2023-2} 
Y_F(z)=F(z) G(z), 
\eeq
where 
$$
F(z)=z^N\sum_{k=0}^\infty F_kz^{-k}~\in~ z^N\cdot \C[\![z^{-1}]\!]\otimes \gl(n,\C),\quad  N\in \Z,
$$
 with $\det F(z)\neq 0$ and   $ F_0\neq 0$ (though $\det F_0=0$, in general), 
and 
\begin{equation} 
\label{6giugno2023-1}
G(z)= z^J U e^{Q(z)}.
\end{equation}
Here,  $J$ is a  matrix in Jordan form, $U$ is a matrix constructed using integer powers of $\exp(2\pi \sqrt{-1} /p)$, and 
$$ Q(z)
=\sum_{k=0}^{p r-1} Q_k z^{\alpha_k}=Q_0z^r+\dots + Q_{pr-1}z^{1/p}
$$
is a diagonal matrix whose entries are polynomials in $z^{1/p}$, without constant 
term  and maximal degree $r$, being 
$$
p\in\N_{>0},\quad \alpha_k:=r-\frac{k}{p}\in \Q_{>0}
.
$$
 A concrete algorithm for computing $Q(z)$  is described  in \cite{Wag89}.   In particular \cite[pag. 398]{Wag89}, the eigenvalues of $r Q_0$  are the eigenvalues of $ 
A_0=\lim_{z\to \infty} A(z)
$, with the same multiplicity. 
   The matrices $Q, J, U$  do not mutually commute.  
We can modify $N$ and  $F$ by  substituting  $J$ with  $J-\Tilde{K}$ and  $F$ with $  F \cdot z^{\Tilde{K}}$, where $\Tilde{K}$ is a diagonal matrix with integer entries. In this way, without loss of generality, we may assume that %
$0\leq \on{Re}(\hbox{eigenvls. of $J$})<1$.  Although $N$ can be fixed in such a way, $F(z)$ may not be unique. %
On the other hand,  $G$ is determined by $A(z)$, and it is a formal invariant of the 
differential equation.\footnote{Having fixed $N$, two differential equations  are formally meromorphically equivalent if and only if they have - up to permutation -  the same $Q$ and $J$ ($U$  always has the predetermined structure) \cite[Sect\, 3]{BJL79-2}.} 

Each $F(z)$ has a unique factorization \cite[Sect.\,4]{BJL79-2} 
$$ 
F(z)=F_a(z)P(z) z^K,\quad\quad F_a(z)=\sum_{k=0}^\infty F^{(a)}_k z^{-k},\quad \det F_0^{(a)}\neq 0,
$$ 
with a formal %
 series $F_a(z)$,    a lower triangular matrix $P(z)$ polynomial in $z$,  $P(0)=\on{diag}P(z)=I_n$, and 
$ 
K$
 is a diagonal matrix with integer entries. 

\subsection{$r$-adequate tuples, and second application of the main theorem}\label{secmaetw}
 Given a formal series $F(z)$, the standard approach in the literature \cite{Sh62,Was65,Sh68}   is to prove the existence of a fundamental matrix solution of  \eqref{30maggio2023-1} holomorphic  at infinity on a small sector, where it has asymptotic behaviour   $F(z)$.  This approach requires the repeated application of shearing transformations and block diagonal reductions a finite number of times.  At each step, the Sibuya-Wasow asymptotic existence theorem applies to a precisely specified sector, but the sector varies at each step.  As a result, the final angular opening cannot be determined, and it can only be shown that a fundamental asymptotic solution corresponding to $F$ exists on every  sufficiently narrow sector (see for example theorem 19.1 in \cite{Was65}). Furthermore, this solution is not unique.  
The {\it minimal}  angular  opening  is provided by the normal sectors defined in   \cite[Sect.\,4]{BJL79-3}.   
\vskip1,5mm
Using our approach based on Theorem \ref{mainth}, we will establish the existence and, most importantly, the uniqueness of  fundamental matrix solutions with a prescribed asymptotic behaviour $F(z)$ on  {\it explicitly determined sectors} with {\it optimal} angular opening. This will be done in Theorem \ref{5giugno2023-7},  under the $r$-{\it adequate} condition (Definition \ref{radeq} below).

 \vskip1,5mm
 Up to a permutation, $Q(z)$ is partitioned into blocks
$$ 
Q(z)= q_1(z)I_{s_1}\oplus \cdots\oplus q_\ell(z)I_{s_\ell},
\quad\quad 
q_i(z):= \sum_{k=0}^{pr-1} \lambda_i^{(k)} z^{\alpha_k},\quad \lambda_i^{(k)}\in\C,\quad  i=1,\dots,\ell,
$$
where $ I_{s_i}$ is the identity matrix of dimension $s_i$, with $s_1+\dots+s_\ell=n$, and 
$$
\boldsymbol{\lambda}_i\neq \boldsymbol{\lambda}_j
,\quad \forall ~1\leq i\neq j \leq \ell
,
$$
 where $\boldsymbol{\lambda}_i:=(\lambda^{(0)}_i,\lambda^{(1)}_i,\dots,\lambda^{(h)}_i)$, and $h:=pr-1$. 
 
Let  $ \lambda_{ij}:=\lambda_i^{(k_{ij})}-\lambda_j^{(k_{ij})}$, with 
$k_{ij}:=\on{min}\{k~|~\lambda_i^{(k)}-\lambda_j^{(k)}\neq 0\}$, and label the distinct values of $\arg \lambda_{ij}\in ]\eta-\pi,\eta[$, for $1\leq i\neq j \leq \ell$ as in \eqref{9giugno2023-4}. 
In complete analogy with Definition \ref{9giugno2023-2},  the {\it Stokes rays} of $Q(z)$ for the pair $(i,j)$, are  the half lines in $\Tilde{\C^*}$ defined by $ 
\on{Re} \lambda_{ij} z^{\alpha_{k_{ij}}}=0$ and $ \on{Im}  \lambda_{ij} z^{\alpha_{k_{ij}}}<0
$. 
The Stokes rays for the pair  $(i,j)$ have directions
$$
\arg z= \tau_{ij}~ \on{mod} \frac{2\pi}{\alpha_{k_{ij}}}, \quad \hbox{ with }\quad \tau_{ij}:=\frac{1}{\alpha_{k_{ij}}}\left(\frac{3\pi}{2} -\arg\lambda_{ij}\right).
$$
  The directions 
$\tau_{ij}$  for all   $\arg\lambda_{ij}\in ]\eta-\pi,\eta[$ are label as in \eqref{sigmarho}, with the same rules (with symbol $\sigma$ replaced by $\alpha$).

Let $\alpha^{(\rho)}$  be the $\alpha_{k_{ij}}$ appearing in the $\tau_{ij}$ labelled as $\tau_\rho$.  We obtain the analog of Proposition \ref{16novembre2024-2}.

\begin{prop} The directions of all  Stokes rays are given by 
\begin{equation}
\label{29maggio2023-2-bis}
\arg z=\tau_\rho+\frac{k \pi}{\alpha^{(\rho)}},\quad \rho\in\{0,\dots,\mu-1\},\quad k\in\mathbb{Z}.
\end{equation}
\end{prop}

\begin{defn}\label{radeq}
We say that $\bm k\in \Z^\mu$ is {\it $r$-adequate with respect to $(\al^{(\rho)};\tau_\rho)_{\rho=0}^{\mu-1}$} if the open interval 
\beq\label{26maggio2023-10-tris}
\Ups_{\bm k}:=\bigcap_{\rho=0}^{\mu-1}\,\left]\tau_\rho+\frac{k_\rho \pi}{\al^{(\rho)}}-\frac{\pi}{r};\,\tau_\rho+\frac{k_\rho \pi}{\al^{(\rho)}}\right[
\eeq is non-empty. 
\end{defn}
Adequateness means that all Stokes rays with directions $\tau_\rho+{k_\rho \pi}/{\al^{(\rho)}}$, for $0\leq \rho
\leq \mu-1$, lie within a sector of angular opening less than $\pi/r$.
\begin{thm}
\label{5giugno2023-7}
Let     $F(z)=F_a(z)P(z) z^K $ be the formal series associated to a formal solution \eqref{19giugno2023-2}  of  equation  \eqref{30maggio2023-1}. Consider the following truncations of $F(z)$ and $F_a(z)$ respectively:
$$ 
F^{(M)}(z):=z^N\sum_{k=0}^M F_kz^{-k}
,\quad \quad F_a^{(M)}(z):=\sum_{k=0}^{M} F_k^{(a)}z^{-k},\quad M\in\N_{>0}.
$$
Let $\bm k\in\Z^\mu$ be $r$-adequate with respect to $(\al^{(\rho)};\tau_\rho)_{\rho=0}^{\mu-1}$.
There exists a
 unique fundamental matrix solution $Y_{\boldsymbol{k}}\colon \Tilde{\C}_{a}\to GL(n,\C)$ of \eqref{30maggio2023-1}, with asymptotic behaviour
\begin{align}
\label{12giugno2023-2}
Y_{\boldsymbol{k}}(z)&=\Bigl(F^{(M)}(z)+O(z^{N-M-1})\Bigr)z^JUe^{Q(z)}
\\
\label{12giugno2023-2-bis}
&=\Bigl(F_a^{(M)}(z)+O(z^{-M-1})\Bigr)P(z) z^Kz^JUe^{Q(z)},
\end{align}
for  $z\to\infty $ 
 in any closed sub-sector of the sector 
$$ 
\mathscr{S}_{\boldsymbol{k}}:=\{z\in \Tilde{\C^*}~|~a_{\boldsymbol{k}}<\arg z <b_{\boldsymbol{k}}+\pi/r\},\qquad \text{where }\Ups_{\bm k}=\left]a_{\bm k}, b_{\bm k}\right[.
$$
\end{thm}

\begin{rem}
\label{5ottobre2024-2}
If $\arg z=a_{\bm k}$ and $\arg z= b_{\bm k}+\pi/r$ are not directions of Stokes rays, then the sector $\mathscr{S}_{\boldsymbol{k}}$  can be extended up to the closest Stokes rays by a standard argument. 
\end{rem}

Before proving the theorem, we state some consequences for the generic case. 
\begin{defn}
\label{6giugno2023-5}
$A$  in \eqref{30maggio2023-1} is {\it generic}  if    $\lambda_i^{(0)}\neq \lambda_j^{(0)}$ for all $i\neq j$. 
\end{defn}
The definition makes sense, because  the eigenvalues $\lambda_i^{(0)}$  are the eigenvalues of $ 
A_0/r=\lim_{z\to \infty} A(z)/r
$ \cite[pag. 398]{Wag89}.
 In the generic case,  $\lambda_{ij}=\lambda_i^{(0)}-\lambda_j^{(0)}\neq 0$, and  $\alpha_{k_{ij}}=\alpha^{(0)}=\dots=\alpha^{(\mu-1)}=r$. It follows form \eqref{29maggio2023-2-bis} that all the possible directions of Stokes rays  can be labelled as 
 $$
 \tau_\nu=\tau_{\rho}+\frac{k\pi}{r},\quad \rho\in\{0,\dots,\mu-1\},\quad \nu:=\rho+k\mu,\quad k\in\mathbb{Z}.
 $$
 They satisfy the strict inequalities  $$ 
\tau_{\nu}<\tau_{\nu+1}\quad \forall~\nu\in \Z.
$$
\begin{rem}
In the general case, a labelling of the Stokes rays was introduced in  \cite{BJL79-3}. Thought the notation $\tau_\nu$ appears in  \cite{BJL79-3}, it coincides with ours only in the generic case.
\end{rem}

\begin{cor}
\label{6giugno2023-7}
If $A$ is generic, then for any $\nu\in \Z$ system \eqref{30maggio2023-1} admits    a unique fundamental matrix solution $Y_\nu\colon \Tilde{\C}_{a}\to GL(n,\C)$, with behaviour \eqref{12giugno2023-2} 
for  $z\to\infty $ 
 in any closed sub-sector of the sector 
\beq
\label{15giugno2023-1} 
\mathscr{S}_\nu:=\Bigl\{z\in \Tilde{\C^*}~|~\tau_{\nu-1}<\arg z<\tau_{\nu}+\frac{\pi}{r}
\Bigr\}.
\eeq
 \end{cor}

\begin{proof} We claim that $\Ups_{\bm k}=\left]\tau_{\nu-1};\tau_\nu\right[$ for some $\nu\in\Z$. Indeed, any sector $\tau<\arg z<\tau+\pi/r$ (with $\arg z=\tau$ not being a Stokes ray) contains exactly $\mu$ Stokes rays with directions $\tau_{\nu}<\tau_{\nu+1}<\cdots<\tau_{\nu+\mu-1}$, for some $\nu\in \Z$. Notice that $\tau_{\nu+\mu}=\tau_{\nu}+\pi/r$, for any $\nu\in\Z$.
The largest sector which contains $\tau_{\nu},\dots,\tau_{\nu+\mu-1}$  is $ 
\mathscr{S}_\nu$. %
\end{proof}

As a corollary of Theorems \ref{mainth} and \ref{5giugno2023-7}, we also recover the following  classical result \cite{BJL79-1}:
\begin{cor}
\label{6giugno2023-6}
Assume that $A_0:=\lim_{z\to \infty} A(z)$ has pairwise distinct eigenvalues. Fix $F_0^{(a)}$  such  that $r Q_0:=(F_0^{(a)})^{-1}A_0 F_0^{(a)}$ is diagonal. 
For any $\nu\in \Z$, equation  \eqref{30maggio2023-1} admits   a unique fundamental matrix solution $Y_\nu\colon \Tilde{\C}_{a}\to GL(n,\C)$ with asymptotic behaviour
$$ 
Y_\nu(z)=\Bigl(F_a^{(M)}(z)+O(z^{-M-1})\Bigr) z^Je^{Q(z)},\quad J\in \h(n,\C),
$$ 
for  $z\to\infty $ 
 in any closed sub-sector of the sector $ 
\mathscr{S}_\nu$ in \eqref{15giugno2023-1}, where  $F_a^{(M)}(z)=\sum_{k=0}^{M} F_k^{(a)}z^{-k} $ is 
 uniquely determined by $A(z)$ and $F_0^{(a)}$. 
 Moreover, 
$A(z)$
  uniquely determines  $Q(z)$ and $J$ by 
\beq
\label{19giugno2023-3}
F_a^{(r)}(z)^{-1}A(z)F_a^{(r)}(z)= \sum_{k=0}^{r-1} (r-k)Q_k z^{-k} +J z^{-r}+O(z^{-r-1}),\quad z\to \infty.
\eeq
\end{cor}

In order to prove  Theorem \ref{5giugno2023-7}, we need the following

\begin{lem}
\label{6giugno2023-4}
For sufficiently large  $M^*\in \N$,  the gauge transformation $Y=F^{(M^*)}(z)z^J U 
\widehat{Y}$  transforms \eqref{30maggio2023-1} into the equation 
\beq
\label{30maggio2023-2}
\frac{d\widehat{Y}}{dz}= \left(\frac{dQ(z)}{dz}+R^\prime(z)\right)\widehat{Y},
\eeq
where  $R^\prime\in \hol(\Tilde{{\C}}_{a};\gl(n,\C))$ has behaviour $R^\prime(z)=O(z^{-1-\delta^\prime})$, $\delta^\prime>0$, for $|z|\to \infty$  and bounded $\arg z$.
 \end{lem}

\begin{proof}  
The formal gauge transformation 
$
Y=F(z)G
$
takes \eqref {30maggio2023-1} to the normal form \cite[Sect. 3 d.]{BJL79-2}
\begin{equation} 
\label{6giugno2023-2}
\frac{d G}{dz}= z^{r-1}\mathcal{P}(z) G,\quad \quad \mathcal{P}(z)=\mathcal{P}_0+\frac{\mathcal{P}_1}{z}+\dots +
\frac{\mathcal{P}_{r-1}}{z^{r-1}}+\frac{J}{z^r}.
\end{equation}
 From  \eqref{6giugno2023-1} it follows that  
$$ 
z^{r-1}\mathcal{P}(z)~=z^J U~ \frac{dQ(z)}{dz}~U^{-1} z^{-J}~+\frac{J}{z} .
$$
Hence, we can take $M^*\in \N$ sufficiently big  to ensure  that the gauge transformation 
$ 
Y=F^{(M^*)}(z)\widetilde{Y}$ satisfies 
$$ 
\frac{d\widetilde{Y}}{dz}=\Bigl( z^{r-1}\mathcal{P}(z) +\widetilde{R}(z)\Bigr)\widetilde{Y}
\equiv \left(z^J U~ \frac{dQ(z)}{dz}~U^{-1} z^{-J}~+\frac{J}{z}+\widetilde{R}(z)\right)  \widetilde{Y},
$$ 
with  $\widetilde{R}\in \hol(\Tilde{\C}_{a};\gl(n,\C)) $ of order $O(z^{-\Delta-1})$ for some  $\Delta\geq 1$.
Another gauge transformation $ 
\widetilde{Y}=z^J U 
\widehat{Y}
$  
yields
$$ 
\frac{d\widehat{Y}}{dz}= \left(\frac{dQ(z)}{dz}+R^\prime(z)\right)\widehat{Y},\quad \quad R^\prime(z):=U^{-1} z^{-J}\widetilde{R}(z) z^J U.
$$
 If $M^*$ is sufficiently big, we can arrange so that $R^\prime(z)=O(z^{-1-\delta^\prime})$, $\delta^\prime>0$.
  \end{proof}

\begin{proof}[Proof of Theorem \ref{5giugno2023-7}] 
With the change of variables $ 
x=z^{r}
$, $r\arg z=\arg x$, equation  \eqref{30maggio2023-2}  becomes  
\begin{equation}
\label{6giugno2023-3}
 \begin{aligned}
&  \frac{d\widehat{Y}}{dx} = 
\left(\Lambda(x)+R(x)
\right) \widehat{Y},&
\\
 & \Lambda(x):=\sum_{k=0}^{pr-1}\sigma_kQ_k x^{\sigma_k-1},\quad \sigma_k:=\frac{\alpha_k}{r}=1-\frac{k}{pr}, & R(x):=\frac{1}{r}x^{1/r-1}R^\prime(x^{1/r}).
\end{aligned}
\end{equation}
Since 
$R(z)=O(x^{-1-\hat{\delta}})$, $\hat\delta:=\delta^\prime/r$, 
Theorem  \ref{5giugno2023-3} can be applied to equation  \eqref{6giugno2023-3} and,  using the  above change of variables, we conclude that   \eqref{30maggio2023-2} has a unique fundamental matrix solution with behaviour
$$ 
\widehat{Y}(z)=(I_n+O(z^{-\delta^\prime}))e^{Q(z)},\quad z\to \infty, \quad z\in \mathscr{S}_{\boldsymbol{k}}.
$$
Recalling the proof of Lemma \ref{6giugno2023-4}, we see that $
\widetilde{Y}=z^J U\widehat{Y}
$ 
 has the  behaviour 
$$ 
\widetilde{Y}(z)=~\Bigl(I_n+z^J UO(z^{-\delta^\prime})U^{-1}z^{-J} \Bigr)~z^J Ue^{Q(z)}~= (I_n+O(z^{-\Delta}))z^J Ue^{Q(z)}.
$$
Then, 
$$ 
Y(z)=F^{(M^*)}(z)\widetilde{Y}(z)=~ \Bigl(F^{(M^*)}(z)+F^{(M^*)}(z)O(z^{-\Delta})\Bigr)~z^J Ue^{Q(z)}.
$$
Now, $F^{(M^*)}(z)\cdot O(z^{-\Delta})~=O(z^{N-\Delta})$. 
 Consider an additional truncation $F^{(M)}$  of $F^{(M^*)}$, with $0<M\leq M^*$. The smallest exponent in $F^{(M)}$  is $z^{N-M}$. Then, we take $M^*$ sufficiently big so that  $\Delta>M$. For such  $M^*$, \eqref{12giugno2023-2} holds. 
In an analogous way, we  prove \eqref{12giugno2023-2-bis}.
\end{proof}

\begin{proof}[Proof of Corollary \ref{6giugno2023-6}]
Fix $F_0^{(a)}$ such that  $ 
\Lambda_0:=(F_0^{(a)})^{-1} A_0F_0^{(a)}$ is diagonal.\footnote{There is a freedom $F_0^{(a)}\mapsto 
F_0^{(a)}C$, $C\in\h(n,\C)$.} Then    $F_a^{(M)}(z)=\sum_{k=0}^{M} F_k^{(a)}z^{-k}$ can be  uniquely  determined\footnote{It is a standard formal computation. Writing the Taylor 
expansion $A(z)=\sum_{k=0}^\infty A_k z^{-k}$,  the coefficients $F_k^{(a)}$ of 
$F_a^{(M)}(z)$  are uniquely determined step-by-step in terms of  $F_0^{(a)}$ and 
the $A_k$. Uniqueness depends on the fact that  $A_0$ has distinct eigenvalues.} 
  with $M\geq r$, such that the gauge transformation $Y=F_a^{(M)}(z) \widehat{Y}$ gives  
\begin{equation}
\label{6giugno2023-9}
\begin{aligned} 
\frac{d\widehat{Y}}{dz}&= \left[z^{r-1}F_a^{(M)}(z)^{-1}A(z)F_a^{(M)}(z)
-F_a^{(M)}(z)^{-1}\frac{dF_a^{(M)}(z)}{dz}
\right]
\widehat{Y}
\\
\noalign{\medskip}
&= \left[z^{r-1}\Bigl(\Lambda_0+\sum_{k=1}^r \Lambda_k z^{-k}\Bigr) +\widetilde{R}(z)\right]\widehat{Y},
\end{aligned}
\end{equation}
with uniquely determined diagonal 
matrices $\Lambda_k$ and  remainder  $\widetilde{R}(z)=O(z^{r-M-2})$ analytic at infinity.  System \eqref{6giugno2023-9} is in the form \eqref{30maggio2023-2}, with $p=1$, 
 $U=I$,  $ Q_k=\Lambda_k/(r-k)$ for $k=0,\cdots, r-1$, and $J=\Lambda_r$. 
Then,  the first part of Corollary \ref{6giugno2023-6} follows from  Corollary \ref{6giugno2023-7}. Moreover, letting  $M=r$, \eqref{6giugno2023-9}  implies \eqref{19giugno2023-3}.

\end{proof}

\begin{example}
\label{17novembre2024-1}
 We apply Theorem \ref{5giugno2023-7} to the equation  
 $$ 
 \frac{dY}{dz}=\left(z^3\begin{pmatrix}
 \sqrt{-1} & 0 & 0 
 \\
 0 & \sqrt{-1} & 0 
 \\
 0 & 0 & -\sqrt{-1}
 \end{pmatrix}+ \begin{pmatrix}
 \sqrt{-1} & 0 & 0 
 \\
 0 & 0& 0 
 \\
 0 & 0 & 0
 \end{pmatrix}+R(z)\right)Y.
 $$
 It is already in the form \eqref{30maggio2023-2}, with 
 $$
 Q(z)= \frac{z^4}{4}\begin{pmatrix}\sqrt{-1} & 0 & 0 
 \\
 0 & \sqrt{-1} & 0 
 \\
 0 & 0 & -\sqrt{-1}
 \end{pmatrix}
 +z\begin{pmatrix}
 \sqrt{-1} & 0 & 0 
 \\
 0 & 0& 0 
 \\
 0 & 0 & 0
 \end{pmatrix},
$$
so that 
$$
  \alpha_0=4,\quad \alpha_3=1,\quad \alpha_1=\alpha_2=0,\quad p=1,
 $$
and
 $$
 \boldsymbol{\lambda}_1=(\sqrt{-1}/4,0,0,\sqrt{-1}),
 $$
  $$
 \boldsymbol{\lambda}_2=(\sqrt{-1}/4,0,0,0),
 $$ 
 $$
 \boldsymbol{\lambda}_3=(-\sqrt{-1}/4,0,0,0).
 $$
Hence, $$k_{12}=3, \quad k_{13}=k_{23}=0.
$$ 
Le us take $\eta=2\pi$. The determinations of $\arg\lambda_{ij}$  in $]\eta-\pi,\eta[$ are only the determination of  $\arg(-\sqrt{-1})$  equal to $3\pi/2$. It follows that
$$ 
\tau_0=\frac{1}{4}\left(\frac{3\pi}{2}-\frac{3\pi}{2}\right)=0,\quad \hbox{ with $4=\alpha_0=:\alpha^{(0)}$},
$$
$$ 
\tau_1=\frac{1}{1}\left(\frac{3\pi}{2}-\frac{3\pi}{2}\right)=0,\quad \hbox{ with $1=\alpha_3=:\alpha^{(1)}$}
$$
All the Stokes rays have directions 
$$
\tau_0+\frac{k\pi}{\alpha^{(0)}}\equiv \frac{k\pi}{4},\quad \tau_1+{k\pi}{\alpha^{(1)}}\equiv k\pi,\quad 
\quad
k\in\mathbb{Z}.
$$
To check adequateness of $\boldsymbol{k}=(k_0,k_1)$, it sufficies to check if there are  sectors  of amplitude less than  $\pi/4$  containing two rays $\arg z= {k_0\pi}/{4}$ and $\arg z=  k_1\pi$. This happens only for the rays whose projection onto $\C^*$ is  either the positive or the negative real axis.  
Hence, Theorem \ref{5giugno2023-7} and Remark \ref{5ottobre2024-2} imply that there is a unique fundamental matrix solution in each sector 
$$ \{z\in \Tilde{\C^*}~|~-\frac{\pi}{4}<\arg z+m\pi  <\frac{\pi}{4}\},\quad m\in\mathbb{Z},
$$
with asymptotic behaviour for $z\to \infty$ given by 
$$ 
Y_F(z)=(I+\sum_{j=1}^\infty F_j z^{-j}) e^{Q(z)}.
$$
On the other hand, Theorem \ref{5giugno2023-7} does not hold  on the sectors
 $$ \{z\in \Tilde{\C^*}~|~0<\arg z +\frac{m\pi}{2} <\frac{\pi}{2}\},\quad \{z\in \Tilde{\C^*}~|~\frac{\pi}{4}<\arg z +m\pi  <\frac{3\pi}{4}\}.
 $$
Uniqueness genuinely fails in these sectors. If   $Y(z)$  is a fundamental matrix solution asymptotic to $Y_F(z)$ in $\{0<\arg z +\frac{m\pi}{2} <\frac{\pi}{2}\}$, or in $ \{\frac{\pi}{4}<\arg z +m\pi  <\frac{3\pi}{4}\}$, then any matrix solution of the form $Y(z)C$ has the same asymptotics, for every
$$
C=\begin{pmatrix}
1 & c & 0 
\\
0 & 1 &0 
\\
0 & 0 &1 
\end{pmatrix}
\quad \hbox{or}
\quad
C=\begin{pmatrix}
1 & 0 & 0 
\\
c & 1 &0 
\\
0 & 0 &1 
\end{pmatrix},\quad\quad c\in\mathbb{C}.
$$ 
D. Guzzetti thanks prof. T. Mochizuki for suggesting this example.\end{example}

\subsection{Subdominant solutions}\label{sec4sub}
We state a result on subdominant vector solutions only in case of distinct eigenvalues. The indices $i$ and $ j_o$ below refer to matrix entries.  The following corollary is a direct application of Theorem \ref{22giugno2023-5}, the proofs of  Lemmas \ref{6giugno2023-4}, of Theorem \ref{5giugno2023-7} and of Corollary \ref{6giugno2023-6}.

\begin{cor}
In the assumptions and notation of Corollary \ref{6giugno2023-6}, fix an index $j_o$ and  let 
 $f_{j_o}^{(M)}(z)$ be the  $j_o$-th column of $F_a^{(M)}(z)$. Moreover, let $\tau_{\rho_o}<\dots<\tau_{\rho_{\mu_o-1}}$ be  consecutive directions of the Stokes rays of the pairs $(i,j_o)$,  $i\in\{1,\dots,n\}$, $i\neq j_o$. If $\tau_{\rho_{\mu_o}-1}-\tau_{\rho_0}<\frac{\pi}{r}$, %
  then there exist unique subdominant vector solutions $y_{j_o}^{(k)}$, $k\in \Z$, with behaviour 
$$ 
y_{j_o}^{(k)}(z)= \bigl(f_{j_o}^{(M)}(z)+ O(z^{-M-1})\bigr) \exp\left\{\sum_{q=0}^{pr-1}\lambda_{j_o}^{(q)} z^{\alpha_q} + J_{j_oj_o} \ln z\right\},\quad z\to \infty
$$
in every closed subsector of the sector %
$$ 
{S}_k^{(j_o)}:= \Bigl\{z\in \Tilde{\C}_a~\Bigl|~     \tau_{\rho_{\mu_o-1}}+(2k-1)\frac{\pi}{r}<\arg z< \tau_{\rho_0}+(2k+2)\frac{\pi}{r} \Bigr\}.
  $$
 \end{cor}

The above has a  practical consequence, which improves \cite[pag. 86-87]{Was65}.

\begin{prop}
\label{12nov2021-1}
 In the case of Corollary \ref{6giugno2023-6}, let $ y(z)$  be a column vector solution with   behaviour 
$$
  y(z)=\left(\sum_{k=0}^M f_k z^{-k}+O(z^{-M-1})\right)~\exp\left\{\lambda_{j}^{(0)} z^{r}+ \sum_{q=1}^{pr-1}\lambda_{j}^{(q)} z^{\alpha_q} + J_{jj} \ln z\right\},\quad z\to \infty \hbox{ in } S,
 $$
$f_k\in \C^n$, $f_0
 \neq 0$, in a sector $S$ not containing Stokes rays,  such that   ${\rm Re}((\lambda_i^{(0)}-\lambda_{j}^{(0)})z^r)>0$, $i\in\{1,...,n\}\backslash\{j\}$ and  $z\in S$.  
   Then, $ y(z)$ is unique and the behaviour holds on a wider sector characterized as follows. If $\tau_{\nu-1}$ and $\tau_{\nu}$ are the nearest Stokes rays outside $S$, the wider  sector contains 
   $$
   \{z\in \Tilde{\C^*} ~|~\tau_{\nu-1}-\pi/r<\arg z <\tau_{\nu}+\pi/r\}.$$
\end{prop}

\subsection{Remark on  the parametric case}\label{sec4par}  The results of Section \ref{18settembre2023-6} can be applied to equation \eqref{30maggio2023-2}, after the change of variables $x=z^r$. 
 However, the reduction of    the initial equation \eqref{30maggio2023-1} to \eqref{30maggio2023-2} involves constructing the  truncation $F^{(M^*)}(z,\omega)$ of a formal solution and the computation of  $J$. These objects are obtained through  a finite repetition  of consecutive shearing transformations and diagonalizations \cite{Was65}, where the analytic properties of the quantities involved are not generally well controlled. For this reason, we will not further explore the parametric case of \eqref{30maggio2023-1}, which lies beyond the scope of this paper.
  
 The case of ramified irregular singularities with parameters remains an open area of research. 
 The reader may refer to  \cite{BV85,Sch01}, where some sufficient conditions for the existence of  holomorphic formal solutions with respect to the parameters are provided. In particular, the condition that the degrees of $q_i(z,\omega)-q_j(z,\omega)$ are independent of $\omega$ is imposed. This condition is called ``well behaviour'' in \cite{Sch01}. 
 
 \begin{rem}
 In \cite{CG18,CDG19}, the reader can find a detailed treatment of both formal and asymptotic solutions, in the parametric case, for possibly non-generic systems with irregular singularities of Poincar\'e rank 1. More precisely, in {\it loc.\,cit.}\,\,the leading term of the coefficients at an irregular singularity is assumed to be diagonal, with not-necessarily simple spectrum. See also \cite{CG17,CDG20} for geometrical applications. In the isomonodromic case, the dependence of solutions on parameters has been extensively studied in the literature. In addition to \cite{CG18, CDG19}, readers may refer to \cite{BM05, BTL13, Bo02, Bo12, Bo14, Guz21, Guz22}.
 \end{rem}

\end{document}